\documentclass[psamsfonts]{tran-l}
\usepackage{amssymb,amsmath,exscale,latexsym,amsthm,graphics,enumerate}  
\usepackage{epsfig}    % eps-Bilder einlesen
\usepackage{epic}      % irgendsone Graphikerweiterung
\usepackage{eepic}     % nochne erweiterung
\usepackage{subfigure} % mehre pics nebeneinander
\usepackage[usenames]{color} % use color in text
\usepackage{psfrag}

\theoremstyle{plain}  \newtheorem{theorem}{Theorem}
                      \newtheorem*{theorem1A*}{Theorem 1A} 
                      \newtheorem*{theorem1B*}{Theorem 1B}
                      \newtheorem*{theorem*}{Theorem}
                      \newtheorem{lemma}{Lemma}[section]
                      \newtheorem{cor}[lemma]{Corollary}

\theoremstyle{definition}
        \newtheorem{definition}[lemma]{Definition}

\theoremstyle{remark}
        \newtheorem*{remark}{Remark}
        \newtheorem*{remarks}{Remarks}
        \newtheorem*{convention}{Convention}
        
        \newtheorem*{guide}{Guide to notation}

\theoremstyle{remark}
        \newtheorem*{notation}{Notation}

\theoremstyle{remark}

\numberwithin{equation}{section}
\newcommand{\dist}{\operatorname{dist}}
\newcommand{\Hdist}{\operatorname{Hdist}}

\newcommand{\const}{\operatorname{const}}

\newcommand{\diam}  {\operatorname{diam}}

\newcommand{\clos}  {\operatorname{clos}}
\newcommand{\inte}  {\operatorname{int}}

\newcommand{\id} {\operatorname{id}}
\newcommand{\R}{\mathbb{R}}      % reelle Zahlen
\newcommand{\C}{\mathbb{C}}      % komplexe Zahlen
\newcommand{\N}{\mathbb{N}}      % natürliche Zahlen
\newcommand{\Z}{\mathbb{Z}}      % ganze Zahlen
\newcommand{\CDach}{\widehat{\mathbb{C}}}% Riemannsche Zahlenkugel
\newcommand{\D}{\mathbb{D}}      % Einheitskreis

\providecommand{\defn}[1]{\emph{#1}}
\providecommand{\abs}[1]{\lvert#1\rvert}
\providecommand{\Abs}[1]{\left|#1\right|}
\providecommand{\norm}[1]{\lVert#1\rVert}

\newcounter{mylistnum}

\newcounter{mycount}
\newenvironment{mylist}
{\begin{list}{(\arabic{mycount})}
    {
      \setlength{\leftmargin}{0.0in}
      \setlength{\itemindent}{0.1in}
      \setlength{\itemsep}{\medskipamount}
      \usecounter{mycount}
    }
}
{\end{list}}

\newcommand{\SC} {\mathcal{S}} 

\newcommand{\BC} {\mathcal{B}}

\newcommand{\PC} {\mathcal{P}}

\newcommand{\RC} {\mathcal{R}}

\newcommand{\SB} {\mathbb{S}}

\newcommand{\B} {\mathbb{B}}

\newcommand{\X} {\mathbf{X}}

\newcommand{\Inter}{\operatorname{Interior}}

\newcommand{\core}{\operatorname{core}}

\newcommand{\BOX}{$\square\!\square$}
\begin{document}
 
\title{Snowballs are Quasiballs}

\author{Daniel Meyer}

\address{P.O. Box 68\\
Gustaf H\"{a}llstr\"{o}min katu 2b\\
FI-00014 University of Helsinki\\
Finland
%Department of Mathematics\\
%2074 East Hall\\
%530 Church Street\\  
%Ann Arbor, MI 48109-1043
}

\email{dmeyermail@gmail.com}

\thanks{This research was partially supported by an NSF postdoctoral
  fellowship and by NSF grant DMS-0244421.}



\date{August 16, 2007}

\begin{abstract}
  We introduce \defn{snowballs}, which are compact sets in $\R^3$
  homeomorphic to the unit ball. They are $3$-dimensional analogs
  of domains in the plane bounded by snowflake curves. For each
  snowball $\BC$ a quasiconformal map $f\colon \R^3\to \R^3$ is
  constructed that
  maps $\BC$ to the unit ball. 
\end{abstract}

\maketitle

%\textpages
 
% ==Introduction 
% \input{intro}
% intro.tex
%
% Introduction, bla bla why we are doing this
%
\section{Introduction}

\subsection{Quasiconformal and quasisymmetric Maps}

\label{sec:quas-quas-maps}
The Riemann mapping theorem asserts that conformal
maps in the plane are ubiquitous. However, 
in higher dimensions all conformal maps are M\"{o}bius transformations
(by a theorem of Liouville). The most fruitful generalization of
conformality is the following.
A homeomorphism $f\colon \R^n\to\R^n$ is called \emph{quasiconformal} \index{quasiconformality} if there is a constant $K<\infty$ such that for all $x\in\R^n$,
\begin{equation}\label{eq:defqc}
      K(x):=\varlimsup_{\epsilon\to0}\frac{\displaystyle{\max_{|x-a|=\epsilon}}|f(x)-f(a)|}{\displaystyle{\min_{|x-b|=\epsilon}}|f(x)-f(b)|}\leq K.
    \end{equation}
%     One can define quasiconformality in more general metric spaces, but has to use another definition.
For conformal maps the above limit is $1$ everywhere. A conformal map ``maps infinitesimal balls to infinitesimal balls'', while a quasiconformal map $f$ ``maps infinitesimal balls to infinitesimal ellipsoids of uniformly bounded eccentricity''.
Alternatively, at almost every point there is an infinitesimal
ellipsoid that is mapped to an infinitesimal ball by $f$ (the inverse
$f^{-1}$ is quasiconformal as well). Thus $f$ assigns an
\emph{ellipsoid-field} to the domain. Quasiconformal maps are much
better understood in the plane than in higher dimensions. 
The reason is that by the \emph{measurable
  Riemann mapping theorem} for every given ellipse-field in the plane
(with uniformly bounded eccentricity), we can find a quasiconformal map
$f$ realizing this ellipse-field. No such theorems exist in higher
dimensions. The classical reference on quasiconformal maps in $\R^n$
is \cite{vaisala_n_dim}. 

\medskip
A closely related notion is the following. A homeomorphism $f\colon X\to Y$ of metric spaces is called \emph{quasisymmetric} \index{quasisymmetry} if there is a homeomorphism $\eta\colon[0,\infty)\to[0,\infty)$ such that
        \begin{equation*}
          \frac{|x-a|}{|x-b|}\leq t \Rightarrow \frac{|f(x)-f(a)|}{|f(x)-f(b)|}\leq\eta(t),
        \end{equation*}
        for all $x,a$, and $b$, with $x\neq b$. 

Quasisymmetry is
a global notion, while quasiconformality is an infinite\/simal
one. Every quasisymmetry is quasiconformal (pick
$K=\eta(1)$). In fact in $\R^n, n\geq 2,$ the two notions
coincide. This is actually true for a large class of metric
spaces; see \cite{JuhPek}. The classical paper on
quasisymmetry is \cite{TV}. A recent exposition can be found
in \cite{JuhAn}.

\subsection{Quasicircles and Quasispheres}
\label{sec:quasicircles}
While quasiconformal maps share many properties with conformal ones,
they are not smooth in general. For example, one can map the snowflake
(or von Koch curve) to the unit circle by a quasiconformal map (of the
plane).  
In general, we call the image of the unit circle under a quasiconformal
map of the plane a \emph{quasicircle}\index{quasicircle}. Ahlfors's
\emph{$3$-point condition} \index{Ahlfors' $3$-point
  condition}\cite{MR27:4921} gives a complete geometric
characterization: a Jordan curve $\gamma$ in the plane is a
quasicircle if and only 
if  for each two points $a,b$ on $\gamma$ the (smaller) arc between
them has diameter comparable to $\abs{a-b}$. This condition is easily
checked for the snowflake. On the other hand, every quasicircle can be
obtained by an explicit snowflake-type construction (see
\cite{MR2003b:30022}). 
% \begin{figure}%[htbp]
%   \centering
%        \scalebox{.8}{\includegraphics{snowflake.eps}}
%        \caption{Mapping of the snowflake.}
%   \label{fig:snowflake}
% \end{figure}

Analogous questions in higher dimensions are much harder. At the moment
a classification of \emph{quasispheres}/\defn{quasiballs} (images of
the unit sphere/ball under a quasiconformal map of the whole space $\R^3$) seems
to be out of reach. In fact very few non-trivial examples of such maps
have been exhibited. Some such maps (in a slightly different setting)
can be found in \cite{freequasi}. 
First snowflake-type examples were constructed in
\cite{MR99j:30023} and \cite{MR2001c:49067}. These quasispheres do not
contain any rectifiable curves. That quasisymmetric embeddings of
certain surfaces exist seems to follow from ongoing work of Cannon,
Floyd, and Parry (\cite{CFPfinSub}), the main tool used being Cannon's
\emph{combinatorial Riemann mapping theorem} \cite{comRiem}. These surfaces are
defined abstractly, so no extension to an ambient space (like $\R^3$)
is possible. A different (though related) approach is to use circle
packings as in \cite{BonKlei}. The quasispheres considered there
are \emph{Ahlfors $2$-regular}, so in a sense are already
$2$-dimensional. Their result provides one step in the proof of
\emph{Cannon's conjecture}, which deals with uniformizing (mapping to
the unit sphere by a quasisymmetry) topological spheres appearing as
the boundary at infinity of Gromov hyperbolic groups.

\subsection{Results and Outline}
\label{sec:results}

Here we consider \emph{snowspheres} $\mathcal{S}$\index{snowsphere}\label{not:sbowsphere} which are topologically 2-dimensional analogs of the snow\-flake, homeomorphic to the unit sphere $\SB=\{x\in\R^3:\abs{x}=1\}$. They are boundaries of \emph{snowballs}\index{snowball} $\mathcal{B}$, which are homeomorphic to the unit ball $\mathbb{B}=\{x\in\R^3:\abs{x}\leq 1\}$. A complete definition is given in Section \ref{cha:snowb-snowsph}. We give a slightly imprecise description here, avoiding technicalities. 

\begin{figure}
  \centering
  \scalebox{0.7}{\includegraphics{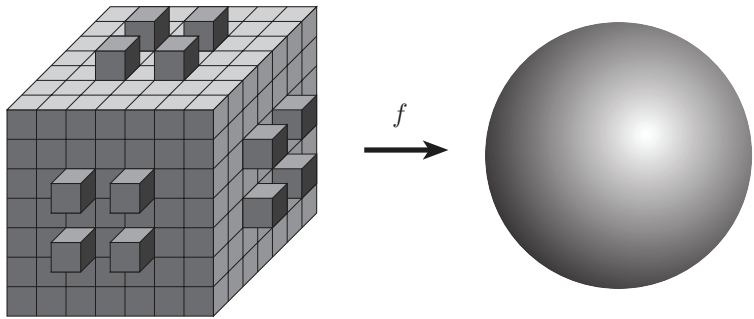}}
  \put(-146,74){$f$}
  \caption{Embedding of the snowball.}\label{fig:embSnowb}
\end{figure}

Start with the unit cube. Divide each face into $N\times N$ squares of
side-length $1/N$ (called $1/N$-squares). Put cubes of side-length
$1/N$ on some $1/N$-squares. We require that the small cubes are added
in a pattern that respects the symmetry group of the cube. This means
that on every side of the unit cube the pattern is the same, as well
as that on each side we can rotate and reflect without changing the
pattern. 
Figure \ref{fig:embSnowb} illustrates one example with
$N=7$. The boundary of the resulting domain is a polyhedral surface
built from $1/N$-squares, called the \emph{first approximation of the
  snowsphere}. Subdivide each $1/N$-square again, and put
cubes of side-length $1/N^2$ on them in the same pattern as
before. Thus we obtain a domain bounded by a polyhedral surface built
from $1/N^2$-squares (the \emph{second approximation of the
  snowsphere}). Iterating this process we get a
\defn{snowball} $\mathcal{B}$ as (the closure of) the limiting domain, 
with a \defn{snowsphere} $\SC$ as its boundary.  

\begin{remarks}
  One has to impose relatively mild conditions to ensure that the
  snowsphere $\SC$ is a topological sphere, i.e., does not have
  self-intersections. In every step a different pattern and a
  different 
  number $N_j$ may be used. We then have to assume that $\sup_j N_j<\infty$.  
\end{remarks}

The main theorem we prove is the following.
  \begin{theorem}
    \label{thm:1}
    For every snowball $\mathcal{B}$ there is a quasiconformal map
    \begin{equation*}
      f\colon\R^3\to\R^3
    \end{equation*}
    that maps $\mathcal{B}$ to the unit ball $\mathbb{B}$. 
  \end{theorem}
Obviously then $f(\SC)=\mathbb{S}$. The proof is broken up into two parts.
\begin{theorem1A*}
  \label{thm:1A}
  Every snowsphere $\SC$ can be mapped to the unit sphere $\mathbb{S}$ by a quasisymmetry
  \begin{equation*}
    f\colon\SC\to \mathbb{S}.
  \end{equation*}
\end{theorem1A*}

This theorem will be proved in Section \ref{cha:unifsphere}. 
%The construction is relatively easy. 
We first equip the \mbox{$j$-th} approximation of the snowsphere with a
\emph{conformal structure} in a standard way. By the
\emph{uniformization theorem} it is conformally equivalent to the
sphere. The proof of the quasisymmetry of the map $f$ relies
essentially on two facts. The first is that the number of small
squares intersecting in a vertex is bounded by $6$ throughout the
whole construction. This means 
that if one looks at a square and adjacent squares, only finitely many
combinatorially different situations occur. The second ingredient is
that \emph{combinatorial equivalence implies conformal
  equivalence}. Thus in combinatorially equivalent sets the distortion
is comparable by \emph{Koebe's theorem}. Only finitely many constants
appear, one for each of the (finitely many) combinatorial situations
of suitable neighborhoods. This idea already appeared in
\cite{snowemb}. 

\medskip
The remainder of the paper concerns the extension of the map $f$ to
$f\colon \R^3\to \R^3$. 
The construction is explicit, though somewhat technical. 
In Section \ref{sec:simplices-extensions} some maps and extensions that
will be useful later on are provided. 
The snowball is decomposed in Section \ref{sec:decomposing-snowball}
in a \emph{Whitney-type} fashion, where the size of a
piece is comparable to its distance from the boundary (the
snowsphere). 
In Section \ref{sec:DecBall}
the pieces are mapped to the unit ball and reassembled
there. 
One has to make sure that $f$ agrees on intersecting pieces (is well
defined). 
The explicit construction of the map $f\colon \SC\to\mathbb{S}$ allows us to control distortion. 

\smallskip
In Section \ref{sec:proof-main-theorem} the remaining part of Theorem \ref{thm:1} is proved. 
\begin{theorem1B*}
  \label{thm:1B}
  The map $f$ from Theorem 1A can be extended to a quasiconformal map
  \begin{equation*}
    f\colon\R^3\to\R^3.
  \end{equation*}
\end{theorem1B*}

Thus one obtains a large class of quasispheres. The \emph{Xmas
  tree example} from \cite{snowemb} shows that there are quasispheres
(in $\R^3$) having Hausdorff dimension arbitrarily close to $3$. On
the other hand, one can construct quasispheres having Hausdorff
dimension $2$ that are not Ahlfors $2$-regular.

\subsection{Notation} $\CDach=\C\cup\{\infty\}$ is the Riemann sphere,
$\SB=\{x\in \R^3:\abs{x}=1\}$ the unit sphere, $\B=\{x\in
\R^3:\abs{x}\leq 1\}$ the (closed) unit ball, $\D=\{z\in \C :
\abs{z}<1\}$ the unit disk.   

The Euclidean norm in $\R^n$ is denoted by $\abs{x}$, the Euclidean
metric by $\abs{x-y}$. The sphere $\SB$ and the unit ball $\B$ are
equipped with the Euclidean metric inherited from $\R^3$, unless
otherwise noted. We identify $\CDach$ with $\SB$, meaning $\CDach$ is
equipped with the chordal metric.
Maximum norm and metric are denoted by $\norm{x}_\infty$ and $\norm{x-y}_\infty$.

For two non-negative expressions $f,g$ we write $f\asymp g$ if there
is a constant $C\geq 1$ such that $\frac{1}{C}g\leq f\leq C g$. We
will often refer to $C$ by $C(\asymp)$, for example we will write
$C(\asymp)=C(n,m)$ if $C$ depends on $n$ and $m$. 

Similarly we write $f\lesssim g$ or $g\gtrsim f$ for two non-negative
expressions $f,g$ if there is a constant $C>0$ such that $f\leq C
g$. The constant $C$ is referred to as $C(\lesssim)$ or $C(\gtrsim)$.

The interior of a set $S$ is denoted by $\inte S$, the closure by
$\clos S$, while
$U_{\epsilon}(S):=\{x:\dist(x,S)<\epsilon\}$ denotes the open
\emph{$\epsilon$-neighborhood} of a set $S$. 

Let
\begin{align}
  d_A(B) & :=\inf\{\epsilon : B\subset U_\epsilon(A)\} \label{eq:defdA}
  \\ \notag
  & = \sup \{\dist(b,A) : b\in B\}. 
\end{align}
The \emph{Hausdorff distance} \index{Hausdorff distance} between two sets $A,B$ is
\begin{equation*}
  \Hdist(A,B):=\max\{d_A(B),d_B(A)\}.
\end{equation*}

\begin{lemma}
  \label{lem:dist_triangle}
  Let $A,B,C$ be arbitrary sets; then
  \begin{align}
    \Hdist(A,B) & \leq \Hdist(A,C) + \Hdist(C,B),
    \\ \label{eq:triag_distdC}
    \dist(A,B) & \geq \dist(A,C) - d_C(B)
    \\ \label{eq:triag_dist}
    & \geq \dist(A,C) - \Hdist(C,B).
  \end{align}
\end{lemma}
\begin{proof}
  The first inequality is clear.

  To see the second inequality, let $b\in B$ be arbitrary; then
  \begin{align*}
    \dist(A,C) & = \inf_{\substack{a\in A\\ c\in C}}\abs{a-c}\leq
    \inf_{a\in A} \abs{a-b}+\inf_{c\in C} \abs{b-c}
    \\
    & =\inf_{a\in A} \abs{a-b}+\dist(b,C) \leq \inf_{a\in A}\abs{a-b} + d_C(B). 
  \end{align*}
  Taking the infimum with respect to $b\in B$ yields
  (\ref{eq:triag_distdC}). The last inequality follows from $d_A(B)
  \leq \Hdist(A,B)$.   
\end{proof}
We identify $\R^2$ with the $xy$-plane in $\R^3$; similarly when
writing ``$[0,1]^2\subset\R^3$'', we identify $[0,1]^2$ with
$[0,1]^2\times\{0\}$, etc. 

\subsection{Polyhedral Surfaces}
\label{sec:polyhedral-surfaces}

Snowspheres will be approximated by polyhedral surfaces. We recall
some well-known facts. Let $S\subset \R^3$ be a polyhedral
surface homeomorphic to the sphere $\SB$.
% The general statement whether two homeomorphic
% manifolds are PL-homeomorphic is known as the (manifold version of
% the) \emph{Hauptvermutung}. It is false in dimension $n\geq 4$, see
% \cite{theHauptvermutung}.
The following is Theorem 17.12 in \cite{MoiseTop23}.
\begin{theorem*}[PL-Sch\"onflies Theorem for $\R^3$]
  \label{thm:PLSchoenflies}
  There is a PL-(piecewise linear) homeomorphism $h\colon \R^3\to
  \R^3$ such that $h(\partial [0,1]^3)=S$.
\end{theorem*}

\begin{cor}\label{cor:PLSchoenbiLip}
  Let $S$ be a polyhedral surface homeomorphic to $\SB$. Then the
  closure of the bounded component of $\R^3\setminus S$ is
  bi-Lipschitz equivalent to the cube $[0,1]^3$. 
\end{cor}

% ==Definition of snowball/snowsphere
%\input{snowdefs} 
%
% snowdefs.tex
%
% Chapter 2, snowballs and snowspheres get defined
%
\section{Snowballs and Snowspheres}
\label{cha:snowb-snowsph}
\subsection{Generators}\label{sec:gen}
We first introduce some terminology. By the \emph{pyramid above}\index{pyramid
  above/below/double} (denoted by $\PC^+$) the unit square
$[0,1]^2\subset\R^2\subset \R^3$ we mean the pyramid with base
$[0,1]^2$ and tip $(\frac{1}{2},\frac{1}{2}, \frac{1}{2})$ (which is
the center of the unit cube $[0,1]^3$). The \emph{pyramid below} the
unit square is the one with base $[0,1]^2$ and tip
$(\frac{1}{2},\frac{1}{2}, -\frac{1}{2})$. We denote by $\mathcal{P}$
the \emph{double pyramid}\index{double pyramid} \label{not:doublepyr}
of the unit square, which is the union of the two pyramids defined
above. The double pyramid $\mathcal{P}(Q)$ of any square $Q\subset
\R^3$ is defined as the image of the double pyramid $\mathcal{P}$
under a similarity (of $\R^3$) that maps the unit square to $Q$. If we
give $Q$ an orientation we also speak of its pyramids \emph{above} and
\emph{below}. 

Consider two distinct unit squares in the grid $\Z^3$. Their double
pyramids intersect at most in a (common) face, which means they have
disjoint interiors. 
 
\smallskip
An $N$-\emph{generator}\index{generator} \label{not:generator} (for an integer $N\geq 2$) is a polyhedral surface $G\subset \R^3$ built from squares of side-length $\frac{1}{N}=\delta$ ($\delta$-\emph{squares}). We require:

\begin{enumerate}[\upshape (i)]
\item \label{def:gen1}
  $G$ is homeomorphic to the unit square $[0,1]^2$.
  
\item \label{def:gen2}
  The boundary of $G$ (as a surface) consists of the four sides of the unit square:
  \[\partial G = \partial[0,1]^2.\]
  
\item \label{def:gen3}
  $G$ is contained in the double pyramid $\mathcal{P}$ and intersects
  its boundary only in the boundary (the four edges) of the unit
  square: 
  \[G\subset \mathcal{P} \mbox{ and } G\cap \partial \mathcal{P}
  =\partial [0,1]^2.\] 
    
\item \label{def:gen4}
  The angle between two adjacent $\delta$-squares is a multiple of
  $\frac{\pi}{2}$ (so it is $\frac{\pi}{2},\pi$, or
  $\frac{3\pi}{2}$). 
    
\item \label{def:gen5}
  The generator $G$ is \emph{symmetric}, meaning it is invariant under
  orientation preserving symmetries of the unit square $[0,1]^2$; more precisely under
  rotations by multiples of $\pi/2$ around the axis
  $\{(\frac{1}{2},\frac{1}{2},z)\}$, and reflections on the planes
  $\{x=\frac{1}{2}\},\{y=\frac{1}{2}\}, \{x=y\}$, and $\{y=1-x\}$. 
  % save counter of enumeration
  \setcounter{mylistnum}{\value{enumi}}

\end{enumerate}

\begin{definition}
  We say a surface that can be decomposed into squares having edges in
  a grid $\delta\Z^3$ \emph{lives}\index{lives in a grid}  in the grid
  $\delta\Z^3$. Similarly, we say a domain \emph{lives} in a grid
  $\delta\Z^3$ if this is true for its boundary. 
\end{definition}

So an $N$-generator lives in the grid $\frac{1}{N}\Z^3$. 
For a given $N$ there can be only finitely many such generators.

\begin{figure}
  \centering
  \includegraphics{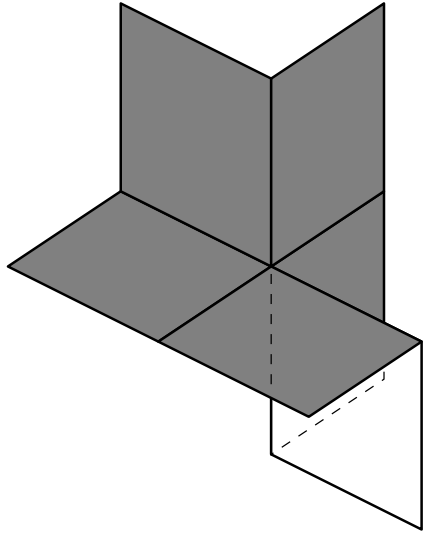}

  \caption{The forbidden configuration.}
  \label{fig:annoy}
\end{figure}
\smallskip
One last assumption about generators will be made, though it is not strictly
necessary. However, it will simplify the decomposition of the snowball
$\mathcal{B}$ in Section \ref{sec:decomposing-snowball} considerably.  
We do not allow the situation indicated in Figure \ref{fig:annoy} to
occur. To be more precise consider an interior \defn{vertex} of $G$,
meaning a point $v\in (G\cap \delta\Z^3)\setminus\partial [0,1]^2$. At
$v$ it is possible that $3,4,5$ or $6$ $\delta$-squares intersect. We
do not allow $6$ $\delta$-squares around $v$ which form successive
angles of $\pi,\frac{3}{2}\pi,\pi/2,\pi,\frac{3}{2}\pi,\pi/2$. 
%Put differently we do not allow the $\delta$-squares around $v$ to be
%spanned by succesive vectors of $\delta e_1,\delta e_2,- \delta e_3,
%-\delta e_1, \delta e_3, -\delta e_2, \delta e_1$.
All
other (allowed) possibilities (up to rotations/reflections) of how
$\delta$-squares may intersect in a vertex are indicated in Figure
\ref{fig:localRj}. % The final assumption is the following.

\begin{enumerate}[\upshape (i)]
  % reset counter
  \setcounter{enumi}{\value{mylistnum}}
\item The generator $G$ does not contain a \defn{forbidden configuration} as in Figure \ref{fig:annoy}. 
\end{enumerate}

In the next section we will define the \defn{approximations} $\SC_j$
of the snowsphere, which will be built successively from generators.

{\smallskip\noindent
{\itshape Remarks.}
%
%  \smallskip
  \begin{itemize}
  \item
    Condition \eqref{def:gen1} in the definition of a generator is
    clearly necessary for $\mathcal{S}_j$ to be homeomorphic to the
    sphere $\mathbb{S}$.  
  \item Condition \eqref{def:gen2} enables us to replace the $\delta_j$-squares by a scaled copy of a generator. 
  \item
    The third condition \eqref{def:gen3} guarantees that the
    approximations $\mathcal{S}_j$ (and ultimately the snowsphere
    $\SC$) are topological spheres. See the next subsection.
  \item
    The fourth condition \eqref{def:gen4} is equivalent to saying that a generator lives in the grid $\frac{1}{N}\Z^3$. It is most likely superfluous. However, we were not able to find a convincing argument for this. 
  \item
    The fifth condition \eqref{def:gen5} is necessary for our method
    to work. Avoiding it would be very desirable. Indeed, tackling the
    non-symmetric case might be the first step towards a general
    theory. 
  \item The last condition is imposed to avoid more technicalities
    when decomposing the snowball in Section
    \ref{sec:whitn-decomp-snowb}. See the Remark on page \pageref{rem:forbiddenconf}. 
  \end{itemize}
}

\subsection{Approximations of the Snowsphere}\label{sec:approx_snowspheres}
A \emph{snowball} $\mathcal{B}$ \index{B@$\mathcal{B}$} is a
three-dimensional analog of the domain bounded by the snowflake
curve. It is a compact set in $\R^3$ homeomorphic to the closed unit ball
$\mathbb{B}=\{x\in \R^3: \abs{x}\leq 1\}$. The corresponding
\emph{snowsphere} $\mathcal{S}:=\partial \mathcal{B}$ is homeomorphic
to the unit sphere $\mathbb{S}=\{x\in\R^3:\abs{x}=1\}=\partial
\mathbb{B}$. 
We will obtain $\SC$ as the Hausdorff limit of \defn{approximations}
$\SC_j$. 
To obtain $\SC_{j+1}$ from $\SC_j$ we 
replace small squares by scaled generators.

\smallskip
The \defn{$0$-th approximation of the snowsphere} $\mathcal{S}_0$ is
the surface of the unit cube, $\mathcal{S}_0 := \partial [0,1]^3$. Now
replace each of the six faces of $\mathcal{S}_0$ by a rotated copy of
an $N_1$-generator to get $\mathcal{S}_1$, the \defn{first
  approximation of the snowsphere}. It is a polyhedral surface built
from $\frac{1}{N_1}$-squares. We construct $\mathcal{S}_2$ by
replacing each $\frac{1}{N_1}$-square of $\mathcal{S}_1$ by a scaled
(by the factor $\frac{1}{N_1}$) and rotated copy of an
$N_2$-generator. Inductively the $j$-th \emph{approximations of the
  snowsphere} $\mathcal{S}_j$ \label{not:japprox} are constructed. Each
$\mathcal{S}_j$ is a polyhedral surface built from squares of
side-length 
\begin{equation}\label{eq:defdeltaj}\index{deltaj@$\delta_j$}
  \delta_j:=\frac{1}{N_1}\times\dots\times\frac{1}{N_j}.
\end{equation}
It will be convenient to
set $\delta_0:=1$ and $\delta_\infty:=0$.
%This should be viewed as the relevant scaling factor
Note that when constructing $\mathcal{S}_{j+1}$ from $\mathcal{S}_{j}$
each $\delta_j$-square is replaced by the \emph{same}
$N_{j+1}$-generator. We do however allow two $\delta_j$-squares $Q_1$
and $Q_2$ to be replaced by scaled copies of the $N_{j+1}$-generator
with \emph{different orientation}. So the generator can ``stick out''
on one square and ``point inwards'' on another. In each step
a different generator may be used. We do require that
\begin{equation}\label{eq:Nmax}\index{Nmax@$N_{\max}$}
  N_{\max} :=\max_j{} N_j <\infty.
\end{equation}
This implies that only finitely many different generators are used.
 The construction may be paraphrased as follows. Pick a finite set of
 generators. In each step pick a generator from this set to construct
 the next approximation.

All relevant constants will depend on $N_{\max}$ only. Such a constant
is called \emph{uniform}\index{uniform constant}. 

\begin{lemma}
  \label{lem:Sspheres}
  The approximations $\SC_j$ are topological spheres.
\end{lemma}
\begin{proof}
  Let $g_0\colon \SB\to \SC_0=\partial [0,1]^3$ be a
  homeomorphism. For every $N_{j+1}$-generator $G_{j+1}$ we can find a
  homeomorphism $[0,1]^2\to G_{j+1}$ which is constant on $\partial
  [0,1]^2$. Apply this homeomorphism to every $\delta_j$-square in
  $\SC_{j}$ to get a continuous and surjective map
  \begin{equation*}
    g_{j+1}\colon \SC_{j}\to \SC_{j+1},
  \end{equation*}
  which is constant on the $1$-skeleton of $\SC_{j}$ (edges of
  $\delta_j$-squares in $\SC_j$). 
  To see injectivity consider two distinct
  $\delta_{j}$-squares $Q,Q'\subset \SC_{j}$. 
  Then $G:=g_{j+1}(Q), G':=g_{j+1}(Q') \subset\SC_{j+1}$ are scaled
  (by $\delta_j$) copies of the $N_{j+1}$-generator. Note that they
  are contained in the double pyramids, $G\subset\PC(Q),
  G'\subset\PC(Q')$. By condition (\ref{def:gen3}) of generators
  \begin{align*}
    g_{j+1}(\inte Q) & =\inte G \subset \inte \PC(Q) \text{ and}
    \\
    g_{j+1}(\inte Q') & \subset \inte \PC(Q'). 
  \end{align*}
  Thus $g_{j+1}(\inte Q) \cap g_{j+1}(\inte Q')=\inte
  \PC(Q)\cap\inte \PC(Q')=\emptyset$. Note also that $\inte \PC(Q)$ does not
  intersect the $1$-skeleton of $\SC_j$. Thus $g_{j+1}$ is injective,
  hence a homeomorphism. This shows by induction that every
  approximation $\SC_j$ is a topological sphere. 
\end{proof}

The approximations $\SC_j$ are polyhedral surfaces. Thus
$\R^3\setminus\SC_j$ has two components by the PL-Sch\"onflies theorem.

\smallskip
Call the edges/vertices of a $\delta_j$-square in $\SC_j$
$\delta_j$\defn{-edges/vertices}. Then the approximations $\SC_j$ form
a \defn{cell complex} in a natural way. 
Namely the $\delta_j$-squares/edges/ver\-ti\-ces in $\SC_j$, are the $2$-,
$1$-, and $0$-cells. 

\subsection{Snowspheres}\label{sec:S_snowsphere}
Note that $\Hdist(\SC_j,\SC_{j+1})\leq \delta_j\leq 2^{-j}$.
Thus we can define 
the snowsphere $\mathcal{S}$ as the limit of the approximations 
$\mathcal{S}_j$ in the Hausdorff topology. It is possible to prove that
$\SC$ is a topological sphere as in Lemma \ref{lem:Sspheres}. However
we would have to make additional assumptions on the maps
$g_j$. Therefore we postpone the proof that $\SC$ is homeomorphic to
$\SB$ until Corollary \ref{cor:S_S2}.

We call the closure of the bounded components of $\R^3\setminus \SC$
the \defn{snowball} $\BC$. It will follow from Theorem 1B
that $\BC$ is homeomorphic to a closed ball. See also Corollary
\ref{cor:BintS}. 
 
%The domain bounded by
%$\mathcal{S}$ will be the snowball $\mathcal{B}$. 
When a snowsphere
$\mathcal{S}$ is given, ``$N_j$-generator'' will always refer to the
one used in the $j$-th step of the construction. 

It will often be convenient to consider only one ``face''
$\mathcal{T}$ of the snowsphere, \index{face of
  snowsphere/approximation} i.e., the part of it that was
constructed from one of the sides of the surface of the unit
cube. More precisely let $\mathcal{T}_0=[0,1]^2$ be the unit square,
$\mathcal{T}_1$ be the $N_1$-generator, $\mathcal{T}_2$ the surface
obtained by replacing each $\frac{1}{N_1}$-square by a scaled copy
of the $N_2$-generator, and so on. Then $\mathcal{T}:=\lim_j
\mathcal{T}_j$\label{not:face}\index{T@$\mathcal{T}, \mathcal{T}_j$}
in the Hausdorff topology.

\begin{figure}
  \centering
  \mbox{
  \subfigure[Generator with enclosing pyramid.]
      {
    \scalebox{1.}{\includegraphics{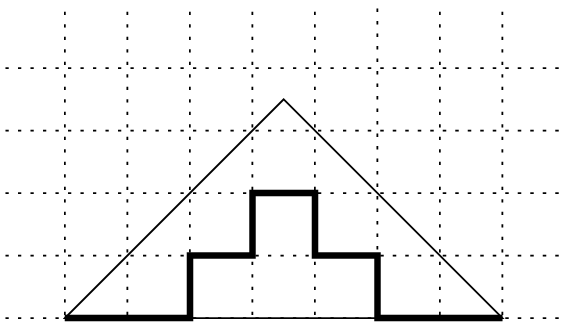}}
    \label{pic:gen_encl}
      }\quad
  \subfigure[Pyramids on each $\frac{1}{N_1}$-square.]
      {
    \scalebox{1.}{\epsfig{file=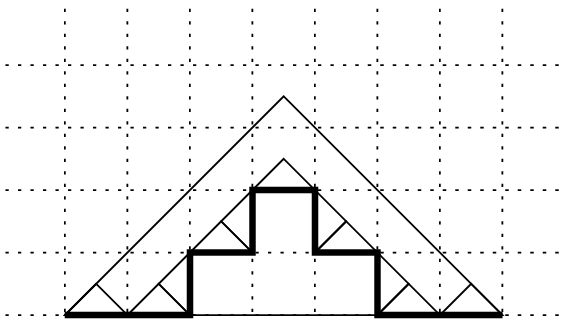}}
      \label{pic:gen_pyra}
      }
  }
  \caption{Generator and pyramids in the grid $\frac{1}{N_1}\Z^3$.}
  \label{pic:gen_e_p}
\end{figure}

Consider the $N_1$-generator ($=\mathcal{T}_1$) and its enclosing double pyramid $\mathcal{P}$. Figure \ref{pic:gen_encl} shows a 2-dimensional picture where we cut through the middle (along the plane $y=\frac{1}{2}$). Only the upper pyramid is depicted. For convenience the picture also indicates the grid $\frac{1}{N_1}\Z^3$ (or rather its 2-dimensional intersection $\frac{1}{N_1}\Z^2$). We note that 
\begin{itemize}
  \item the height of $\mathcal{T}_1$ is at most $\frac{1}{2} -\frac{3}{2}\frac{1}{N_1}$.
\end{itemize} 
 Here the precise meaning of ``height''\index{height} is the maximal distance of a point in the generator from the base square $[0,1]^2$. This is easily seen from Figure \ref{pic:gen_encl}. Indeed, the 
 next layer of $\frac{1}{N_1}$-cubes (having height
 $\frac{1}{2}-\frac{1}{2N_1}$) would intersect the boundary of the
 double pyramid (or lie outside). If $N_1$ is even the height is at
 most $\frac{1}{2}-\frac{2}{N_1}$. 

 The projection of any
 generator to the $xy$-plane is the square $[0,1]^2$. Thus we note the
 following consequence of the above:
 \begin{equation}
   \label{eq:HdistSj}
   \Hdist_\infty(\SC_j,\SC_{j+1})\leq\Hdist(\SC_j,\SC_{j+1})\leq
   \left( \frac{1}{2} -\frac{3}{2}\frac{1}{N_{j+1}}\right) \delta_j.
 \end{equation}
 Here ``$\Hdist_\infty$'' is the Hausdorff distance taken with respect
 to the maximum metric; see Subsection \ref{sec:whitn-decomp-snowb}.

Put pyramids on the $\frac{1}{N_1}$-squares of $\mathcal{T}_1$. These
stay inside the double pyramid $\mathcal{P}$; see Figure
\ref{pic:gen_pyra}. Consider the pyramids of \emph{interior
  $\frac{1}{N_1}$-squares}, \index{non-boundary squares}i.e., squares
that do not intersect the boundary of the unit square
$\partial[0,1]^2$. These have distance at least
$\frac{\sqrt{2}}{2}\frac{1}{N_1}$ from the surface of the enclosing
double pyramid $\mathcal{P}$.   

If we now replace each $\frac{1}{N_1}$-square by the $N_2$-generator
to get $\mathcal{T}_2$, we see that $\mathcal{T}_2$ stays inside the
$\frac{1}{N_1}$-pyramids depicted in Figure \ref{pic:gen_pyra}. 
Induction yields that all $\mathcal{T}_j$ and hence $\mathcal{T}$ are
contained in the double pyramid $\mathcal{P}$. Furthermore, if $Q_j$ is
an interior $\delta_j$-square of $\mathcal{T}_j$, then the double
pyramid of $Q_j$ has distance 
$\sqrt{2}\delta_j/2$ from the boundary $\partial \mathcal{P}$.
% By
% iteration we get that the interior $\delta_j$-squares of
% $\mathcal{T}_j$ have distance at least $\frac{\sqrt{2}}{2}\delta_j$
% from the boundary $\partial\mathcal{P}$. 
We conclude

\begin{itemize}
\item $\mathcal{T}$ is contained in the double pyramid $\mathcal{P}$
  and intersects its boundary only in the boundary of the unit square: 
  \[\mathcal{T}\subset \mathcal{P} \mbox{ and } \mathcal{T}\cap
  \partial \mathcal{P} =\partial [0,1]^2.\] 
  
\item The height of $\mathcal{T}$ is at most
  $\frac{1}{2}-\frac{1}{N_1}\leq\frac{1}{2}-\frac{1}{N_{\max}}$. ($*$) 

\end{itemize} 

Again by ``height'' we mean the maximal distance of a point in
$\mathcal{T}$ from the base square $[0,1]^2$. 
The projection of $\mathcal{T}$ to the $xy$-plane is still the square $[0,1]^2$.
Thus we conclude by ($*$) above
 that the Hausdorff distance between $\mathcal{S}_j$ and $\mathcal{S}$
 satisfies
\begin{equation}\label{eq:dSjS}
  \Hdist(\mathcal{S}_j,\mathcal{S})\leq \delta_j\left(\frac{1}{2}-\frac{1}{N_{\max}}\right).
\end{equation}

\medskip
Recall that the $j$-th approximation of the snowsphere $\mathcal{S}_j$
was built from $\delta_j$-squares. The part of the snowsphere which
was constructed by replacing one such $\delta_j$-square $Q\subset\SC_j$
(infinitely often) by generators is called a \emph{cylinder} of
\emph{order} $j$ (or
\emph{$j$-cylinder}). By the previous argument this cylinder is contained in
the double pyramid $\PC(Q)$ of $Q$, so we can define more precisely
\begin{equation*}
  X_j=X_j(Q):=\PC(Q)\cap\SC
\end{equation*}
to be the \defn{$j$-cylinder with base} $Q$.
The \emph{set of all}
$j$-cylinders is denoted by $\mathbf{X}_j$. It will be convenient to
let $\SC$ be the (only) $-1$-cylinder. Set $\delta_{-1}:=2$ so that
\begin{equation*}
  \diam X_j \leq \sqrt{2} \delta_j,
\end{equation*}
for every $j$-cylinder $X_j$.

For every point $x\in\mathcal{S}$ there is a (not necessarily unique)
sequence $(X_j)_{j\in\N}$, where $X_j$ is a $j$-cylinder such that 
\begin{equation}\label{eq:X0Xn}
  X_0\supset X_1\supset X_2\supset \dots \supset \bigcap_j X_j=\{x\}.
\end{equation}

If we use the same generator with the same orientation throughout the construction of $\mathcal{S}$, 
we get a \emph{self-similar snowsphere}. \index{snowsphere!self similar}In that case each cylinder is a (scaled and rotated) copy of $\mathcal{T}$.

Now consider a $\delta_j$-square $Q\subset \SC_j$, its double pyramid $\mathcal{P}(Q)$, and its cylinder $X_j=X_j(Q)$. Then $X_j$ is contained in $\mathcal{P}(Q)$ and intersects it only in the boundary of $Q$ by the same reasoning as above: 
\[X_j\cap\mathcal{P}(Q)=\partial Q.\] 

Now let $R\subset\SC_j$ be a second $\delta_j$-square. Their double pyramids
$\mathcal{P}(Q)$ and $\mathcal{P}(R)$ intersect only at the boundary:
$\mathcal{P}(Q)\cap\mathcal{P}(R)=\partial
\mathcal{P}(Q)\cap\partial\mathcal{P}(R)$ (they have disjoint
interior). It follows that the cylinders
$X_j(Q)\subset\mathcal{P}(Q)$ and
$X_j(R)\subset\mathcal{P}(R)$ intersect only in the intersection
of $Q$ and $R$: 
\begin{equation*}
  X_j(Q)\cap X_j(R)=Q\cap R.   
\end{equation*}
Thus two distinct non-disjoint $j$-cylinders can intersect in an
edge or a vertex (contained in $\delta_j\Z^3$). Hence the $j$-cylinders
form a cell complex in a natural way.  
\begin{lemma}
  \label{lem:comb_cyl_squares}
  The set of $\delta_j$-squares in the approximations $\SC_j$ is
  combinatorially equivalent to the set of $j$-cylinders. More precisely
  map each $\delta_j$-edge/vertex to itself and each $\delta_j$-square
  $Q\subset\SC_j$ to its cylinder 
  $X_j(Q)\in \X_j$,
  \begin{equation*}
    Q \mapsto X_j(Q).
  \end{equation*}
  This map is a cell complex isomorphism. 
\end{lemma}

%We conclude that the snowsphere $\mathcal{S}$ has no self-intersections.

\subsection{Combinatorial Distance on $\SC$}

As a subset of $\R^3$, the snowsphere $\mathcal{S}$ inherits the
Euclidean metric that we denote by $|x-y|$. Often it will be
convenient to describe distances in purely combinatorial terms. 
Given points $x,y\in \SC$ let
\begin{equation}
  \label{eq:defjxy}
  j(x,y):=\min\{j: \text{ there exist disjoint } j\text{-cylinders } X_j\ni x,
  Y_j\ni y\}.
\end{equation}
One may view $\SC$ as the Gromov-Hausdorff limit of
$j$-cylinders. 
The $j=j(x,y)$-th approximation $\SC_j$ is the first in which it is
possible to distinguish $x$ and $y$.

\begin{lemma}\label{lem:pseudcomp}
  For all $x,y\in\mathcal{S}$ we have
  \begin{equation}
    \abs{x-y} \asymp \delta_j, 
  \end{equation}
  where $j=j(x,y)$ and a constant
  $C(\asymp)=C(N_{\max})$.
\end{lemma}

\begin{proof}
  Let $x,y\in \SC$ be arbitrary, and
  let $j:=j(x,y)$.
  Consider $(j-1)$-cylinders $X_{j-1}\ni x$ and $Y_{j-1}\ni y$. Then $X_{j-1}\cap Y_{j-1}\neq \emptyset$, by the definition of $j(x,y)$. 

Therefore
  \begin{equation}
    |x-y|\leq\diam X_{j-1}+\diam Y_{j-1}=\sqrt{2}\delta_{j-1}+\sqrt{2}\delta_{j-1}\leq 2\sqrt{2}N_{\max}\delta_j.
  \end{equation}
  For the other inequality let $X_j\ni x$ and $Y_j\ni y$ be disjoint
  $j$-cylinders. Note that two disjoint $j$-cylinders are closest when
  their bases are opposite faces of a $\delta_j$-cube. Their distance
  then is at least
  \begin{equation*}
  \delta_j-2\delta_j\left(\frac{1}{2}-\frac{1}{N_{\max}}\right) =
  \frac{2\delta_j}{N_{\max}},    
  \end{equation*}
  which is 
  the distance of base squares $-$ twice the height of $j$-cylinders, by Subsection \ref{sec:S_snowsphere}. Hence
  \begin{equation}
    |x-y|\geq \dist(X_j,Y_j)\geq\frac{2\delta_j}{N_{\max}},
  \end{equation}
  which finishes the proof.
\end{proof}

The last lemma shows that $\delta(x,y):= \delta_{j(x,y)}$ is a
\defn{quasimetric}. However $\delta(x,y)$ will violate the triangle
inequality. 

\subsection{Example}\label{sec:example}
Our main example to illustrate our construction will be the self-similar snowball with generator as illustrated in Figure \ref{fig:5gen}. It is the unit square divided into $25$ $\frac{1}{5}$-squares where we put a $\frac{1}{5}$-cube onto the middle square.

\begin{figure}[b]
  \centering
  \scalebox{0.8}{\includegraphics{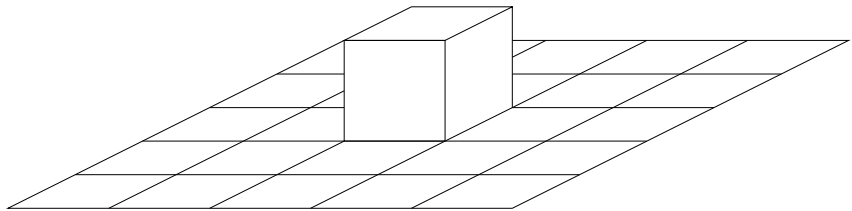}}
  \caption{Generator used for our main example.}
  \label{fig:5gen}
\end{figure}

\begin{notation}
  When referring to this particular example we will always use
  ``$\;\,\widehat{}\;\,\,$'', i.e., $\widehat{\mathcal{S}}$ denotes
  this snowsphere, \index{S
    hat@$\widehat{\SC}$}$\widehat{\mathcal{S}}_j$ its $j$-th
  approximation, and so on. Then $\widehat{\delta}_j=5^{-j}$.
\end{notation}

% ==Uniformizing the the snowsphere
%\input{unifsphere}

\section{Uniformizing the Snowsphere}
\label{cha:unifsphere}

\subsection{Introduction}
\label{sec:introunif}

In this section we map the snowsphere $\mathcal{S}$ to the unit sphere $\mathbb{S}$ by a quasisymmetry $f$, i.e., prove Theorem 1A. We call $f$ a \emph{uniformization} of the snowsphere $\mathcal{S}$. Recall from equation (\ref{eq:X0Xn}) 
that for every point $x\in\mathcal{S}$ there is a sequence $X_0\supset
X_1\supset X_2\supset \dots,\; X_j\in \X_j$, such that $\bigcap_j
X_j=\{x\}$. It will therefore be enough to map the $j$-cylinders
$X_j\subset\mathcal{S}$ to $j$-\emph{tiles} \index{tiles}
\label{not:prime} $X'_j\subset\mathbb{S}$, which will again satisfy
$X'_0\supset X'_1\supset X'_2\supset \dots$ . ``Cylinders'' live in
the snowsphere $\SC$ and ``tiles'' on the unit sphere
$\mathbb{S}$. Generally objects in $\SB$ will be denoted with a
``prime'' ($X',x'$, and so on), to distinguish them from objects in
the snowsphere $\SC$ and its approximations $\SC_j$. We will then
define 
\begin{equation}
  f(x)=x', \; \text{ where } \{x'\}=\bigcap_j X'_j.
\end{equation}
The decomposition of the unit sphere $\mathbb{S}$ into $j$-tiles $X'_j$ is done by using the uniformization of the $j$-th approximation of the snowsphere $\mathcal{S}_j$.

\smallskip
The proof that the map $f$ is a quasisymmetry relies on two
facts. First, at most $6$ $j$-cylinders (and thus $j$-tiles) can
intersect in a common vertex.
Second, two sets of $j$-tiles and $k$-tiles which ``have the same combinatorics'' are actually \emph{conformally equivalent}. The quasisymmetry is then essentially an easy consequence of the Koebe distortion theorem.

\subsection{Uniformizing the approximations $\SC_j$}
\label{sec:mapcyl}

Consider the $j$-th approximation $\mathcal{S}_j$ of the snowsphere $\mathcal{S}$. This is a polyhedral surface where each face is a $\delta_j$-square. We will view $\mathcal{S}_j$ as a \emph{Riemann surface}. To do this we need \emph{conformal coordinates} on $\mathcal{S}_j$, meaning that changes of coordinates are conformal maps. 

\subsubsection{Conformal Coordinates on the Approximations
  $\mathcal{S}_j$}
\label{sec:conf-coord-appr}
\index{conformal coordinates}
\begin{itemize}
\item For each $\delta_j$-square $Q$ the affine, orientation preserving map $\inte Q\to \inte [0,1]^2$ is a chart.
\item For two \emph{neighboring}\index{neighboring!squares} 
$\delta_j$-squares $P$, $Q$ (i.e., ones which share an edge), the map 
$$\inte (P\cup Q)\to \inte ([0,2]\times [0,1]),$$ which maps $P$
(affinely, orientation preserving) to $[0,1]^2$, $Q$ (affinely,
orientation preserving) to $[1,2]\times[0,1]$, and $P\cap Q$ to
$\{1\}\times [0,1]$, is a chart. Using (hopefully) intuitive notation
we sometimes write: \emph{$P\cup Q$ may be mapped conformally to
  \BOX}. So  
$P$ and $Q$ are \emph{conformal reflections}\index{conformal reflection} of each other in these coordinates.  
\item Consider a vertex $v$. Let $Q_1,\dots,Q_n$ be the $\delta_j$-squares containing $v$, labeled with positive orientation around $v$. Map the neighborhood $\inte (\bigcup Q_k)$ of $v$ by $z\mapsto z^{4/n}$. More precisely the chart is constructed as follows. Map $Q_1$ 
to the unit square $[0,1]^2$ as above with $v\mapsto 0$. The unit
square $[0,1]^2$ is subsequently mapped by the map $z\mapsto
z^{4/n}$. Map the second $\delta_j$-square $Q_2$ as before to
$[0,1]^2$ (again with $v\mapsto 0$), which is then mapped by $z\mapsto
e^{2\pi i/n}z^{4/n}$. Alternatively we could have mapped $Q_2$ to
$[-1,0]\times[0,1]$ and subsequently by the map $z\mapsto z^{4/n}$. So
the image of $Q_2$ is a conformal reflection of the image of $Q_1$,
along the shared side $[0,e^{2\pi i/n}]$. The third $\delta_j$-square
$Q_3$ is mapped to $[0,1]^2$, and then by $z\mapsto e^{4\pi
  i/n}z^{4/n}$ and so on. Again the image of $Q_3$ is a reflection of
the image of $Q_2$, analogously for the other
$\delta_j$-squares. Since each mapped $\delta_j$-square forms an angle
of $2\pi/n$ at $0$, the last matches up with the first, meaning they
are conformal reflections of each other.       
\end{itemize}
It is immediate that changes of coordinates are conformal. 
The charts are illustrated in Figure \ref{fig:confcharts}.
 
\begin{figure*}
  \centering
  \scalebox{.9}{\includegraphics{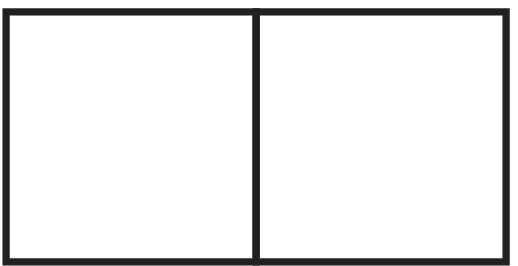}}
  \psfrag{a}[][][2]{$z\mapsto z^{4/3}$}
  \scalebox{.5}{\includegraphics{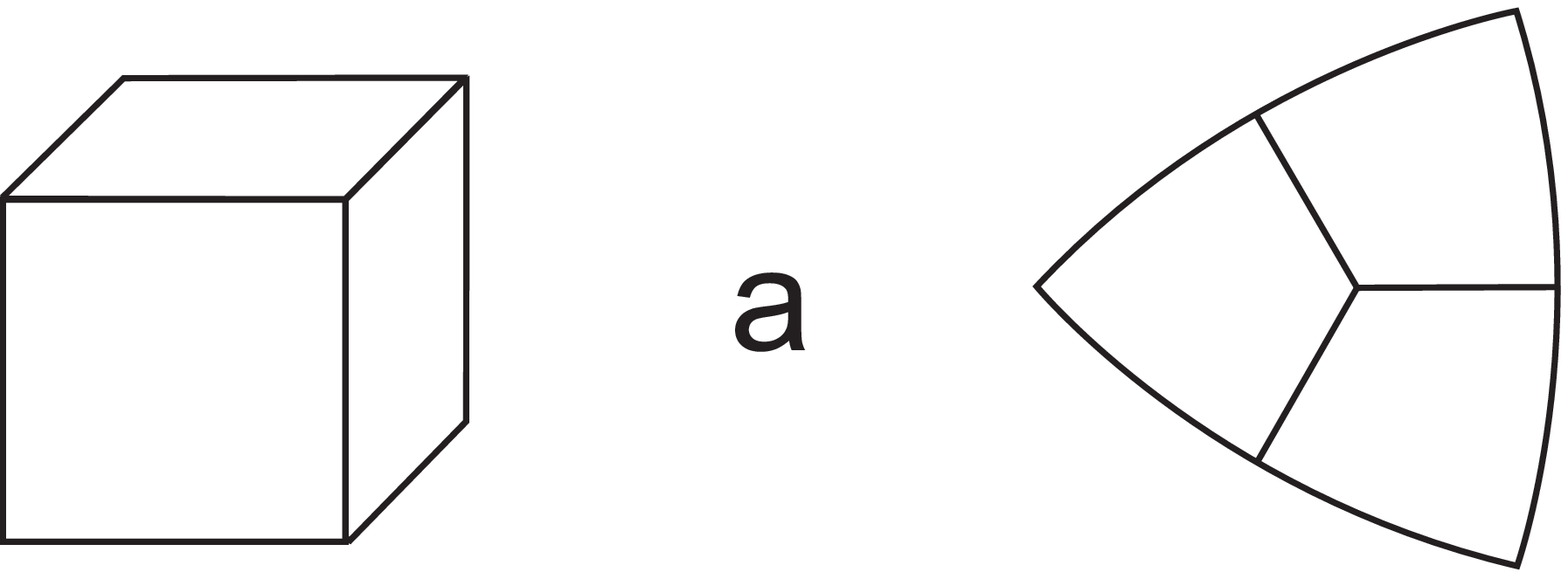}}
  \psfrag{b}[][][2]{$z\mapsto z^{4/5}$}
  \scalebox{.5}{\includegraphics{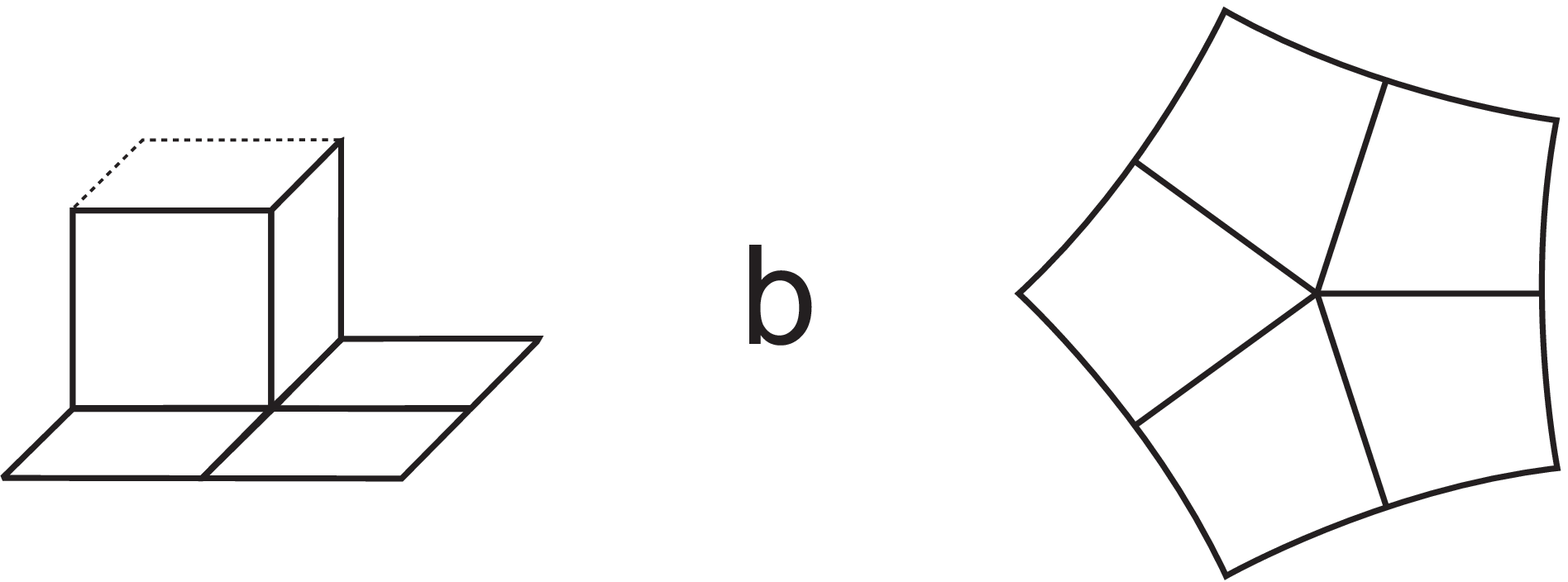}}
  \caption{Conformal coordinates on a polyhedral surface.}
  \label{fig:confcharts}
\end{figure*}

With these charts each approximation $\SC_j$ of the snowsphere is a compact, simply connected Riemann surface. 
Therefore $\SC_j$ is conformally equivalent to the sphere $\CDach$ by
the \emph{uniformization theorem}. 
Identify $\CDach$ with $\SB\subset\R^3$. 
It is not yet clear, however, what
the relation is between uniformizations of different approximations
$\SC_j$ and $\SC_k$. We therefore construct the uniformizations of
the $\SC_j$ inductively, where this will be apparent. 

Start with $\mathcal{S}_0$, which is the surface of the unit cube $\partial[0,1]^3$. Equip $\mathcal{S}_0$ with a conformal structure as above and map it conformally to the Riemann sphere $\CDach$ using the uniformization theorem. The images of the faces of $\SC_0$ decompose the sphere $\CDach$ into $0$-\emph{tiles}. Edges and vertices of those $0$-tiles are the images of edges and vertices of the faces of $\SC_0$.
By symmetry we can assume that the vertices of the $0$-tiles form a
cube, i.e., all $0$-tiles have the same size. 

Denote the set of all such $0$-tiles by $\X'_0$. Each tile $X'\in\X'_0$ is \emph{conformally a square}, meaning we can map it conformally to the unit square $[0,1]^2$, where vertices map to vertices.  
Consider two \emph{neighboring} \index{neighboring!tiles} tiles
$X',Y'\in\X'_0$ (i.e., which share an edge). By the definition of our
charts they are conformal reflections of each other. So we could start
with one tile and get all other tiles by repeated reflection along
the edges. Such a tiling is called a conformal tiling.  

\begin{definition}
  A \emph{conformal tiling} \index{conformal tiling} of a domain
  $D\subset\CDach$ is a decomposition into tiles $D=\bigcup T$, such that:  
  \begin{itemize}
  \item Each tile $T$ is a closed Jordan region, bounded by finitely
    many analytic arcs. 
    Every such arc is part of the boundary of exactly two tiles.
  \item Two distinct tiles $T$ and $\widetilde{T}$ have disjoint
    interior, $\inte T\cap \inte \widetilde{T}=\emptyset$.

  \item Call the endpoints of the analytic arcs (from the boundaries
    of the tiles) \defn{vertices}. The tiling forms a \defn{cell
      complex}, where the tiles/analytics arcs/vertices are the
    $2$-,$1$-, and $0$-cells. This means in particular that distinct
    tiles can only intersect in the union of several such analytic arcs
    and vertices.   
  \item Two tiles sharing an analytic boundary arc (\emph{neighbors}) are
    conformal reflections along this arc. 
  \end{itemize}

\end{definition}
Conformal tilings are of course preserved under conformal maps.

Now consider the $N_1$-generator $G_1$ as a Riemann surface using
charts as above. Note that $\inte G_1$ is simply connected, and has more
than two boundary points. Thus $\inte G_1$ is conformally equivalent to the unit disk $\D$ by the uniformization theorem.
%Why is $G_1$ not conformally equivalent to the plane $\C$? 
% The first approximation $\SC_1$ of the snowball consists of $6$ copies of $G_1$. One of the copies is mapped to a bounded domain by the uniformization of $\SC_1$. 
% It follows from Liouville's theorem that $G_1$ is conformally
% equivalent to the unit disk $\D$. 
Because of symmetry, we can map
$G_1$ conformally to the unit square $[0,1]^2$ (mapping vertices to
vertices as usual). Figure \ref{fig:unifgen} shows the uniformization
of the generator $\widehat{G}$ (see Figure \ref{fig:5gen}) of the
example $\widehat{\SC}$. 
The picture was obtained by dividing the generator along the diagonals
into $4$ pieces. One such piece (a $7$-gon) was mapped to a quarter of
the unit square by a Schwarz-Christoffel map, using
Toby Driscoll's Schwarz-Christoffel Toolbox 
({\ttfamily http://www.math.udel.edu/\verb+~+driscoll/ software/});
see
\cite{DriscollTrefethen}. Thus this picture (as well as following ones)
is conformally correct, up to numerical errors.

The images of the $\delta_1$-squares in $G_1$ again form a tiling of
the unit square $[0,1]^2$.  
Map a second copy of the uniformized generator to the square
$[1,2]\times [0,1]$ (map the two tiled squares to \BOX). The tilings
are symmetric with respect to the line $\{1\}\times[0,1]$ because of
the symmetry of the generator $G_1$. So we get a conformal tiling of
\BOX.

\begin{figure}[p]
  \begin{center}
    \scalebox{0.65}{\includegraphics{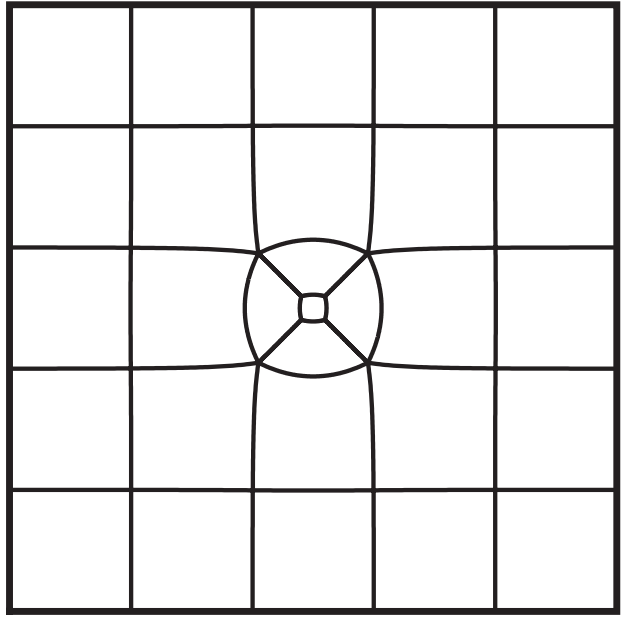}}
  \end{center}
  \caption{
    Uniformization of the generator
    $\widehat{G}$ of the snowsphere $\widehat{\mathcal{S}}$.}
  \label{fig:unifgen}  
\end{figure}
    
\begin{figure}%[p]
  \begin{center}
    \scalebox{0.65}{\includegraphics{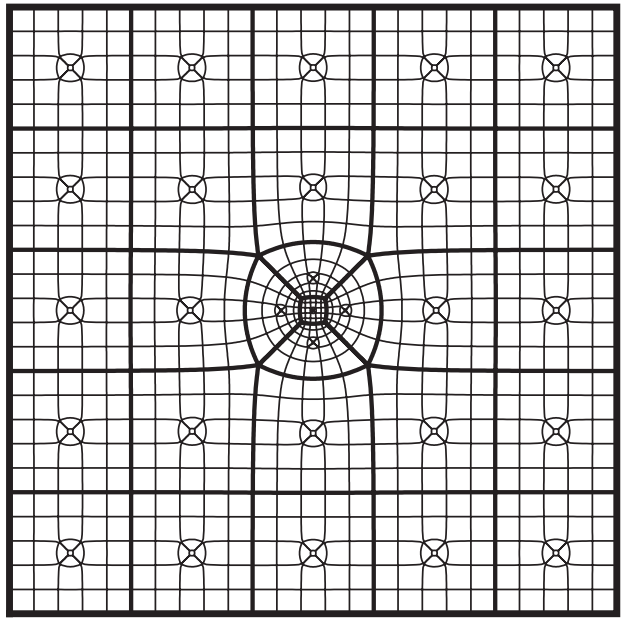}}
  \end{center}
  \caption{\label{fig:2tiles} $2$-tiles of $\widehat{\mathcal{S}}$.}
\end{figure}

\begin{convention}
  When we have a conformal map from a square to a tile $g\colon [0,1]^2\to X'$ we always assume that it maps vertices onto each other. The same normalization is used when mapping a tile to another $X'\to Y'$. 
\end{convention}
\medskip
The uniformized generator $G_1$ and each $0$-tile $X'_0$ are
conformally equivalent to a square. So we can map the uniformization
of $G_1$ (the unit square tiled by images of $\delta_1$-squares) to
$X'_0$. The images of the tiles of $[0,1]^2$ under this map are called
the \emph{$1$-tiles} $X'_1\subset\CDach$. We denote the set of all $1$-tiles by $\X'_1$. 
\subsubsection{Properties of the Tiling}
\begin{itemize}
\item Every $1$-tile is conformally a square, meaning we can map it to
  the unit square $[0,1]^2$ by a conformal map (mapping vertices to vertices).
\item Each $1$-tile is contained in exactly one $0$-tile.

\item Two neighboring $1$-tiles $X'_1,Y'_1$ (tiles which share an edge) may be mapped conformally to the rectangle \BOX. This is clear when $X'_1$ and $Y'_2$ are contained in the same $0$-tile $X'_0$. 

Assume they are contained in different $0$-tiles, $X'_1\subset X'_0\in\X'_0$ and $Y'_1\subset Y'_0\in \X'_0$. Then $X'_0\cup Y'_0$ can be mapped conformally to the rectangle  \BOX. In this chart the tiles in the left and right square are symmetric with respect to the line $\{1\}\times [0,1]$. So $X'_1$ and $Y'_1$ are conformal reflections of each other. 

\item The set $\X'_1$ forms a conformal tiling of the sphere $\CDach$. 

\item Each $\delta_1$-square $Q\in \SC_1$ is mapped to a $1$-tile. Squares which share a (vertex, edge) are mapped to $1$-tiles which share a (vertex, edge) under this map.
  
\item The tiling $\X'_1$ is a uniformization of the approximation $\SC_1$ of the snowsphere. By this we mean the following. Map a $\delta_1$-square $Q$ to its corresponding $1$-tile $X'$ by the Riemann map (normalized by mapping corresponding vertices onto each other). By reflection this extends to a neighboring $\delta_1$-square $\widetilde{Q}$, where it is the Riemann map to the neighboring $1$-tile $\widetilde{X}'$ (again with the ``right'' normalization at vertices). The map extends to all of $\SC_1$ by reflection and is well defined.
The extension is conformal (with respect to the conformal structure on $\SC_1$ as described above). 
\end{itemize}
The above procedure is now iterated. Let the $j$-th tiling of the
sphere $\CDach$ be given, and let the set of $j$-tiles be denoted by
$\X'_j$.\index{X bp@$\X'_j$} We map the uniformized
$N_{j+1}$-generator to each $j$-tile $X'_j\in\X'_j$ to get the
$(j+1)$-tiles $X'_{j+1}\in \X'_{j+1}$. 
Tiles are always compact.
\index{j -@$j$-!tile} All the
above statements hold (where $0$ is replaced by $j$ and $1$ by
$j+1$). Figure \ref{fig:2tiles} shows the $2$-tiles for the example
$\widehat{\SC}$. It will be convenient to call the whole sphere $\SB$
the (only) $-1$-tile. Let us record the properties of the tilings.

\begin{lemma}
  \label{lem:proptilings}
  The tiles satisfy the following:
  \begin{enumerate}
  \item Each $j$-tile is conformally a square, meaning we can map it
    conformally to the square (mapping vertices to vertices).
  \item The set of $j$-tiles forms a conformal tiling for every $j\geq
    0$.
  \item \label{item:lem_prop_tiles_Fj}
    The $j$-th tiling is a uniformization of the approximation
    $\SC_j$. This means there are conformal maps (with respect to the
    structure from Subsection \ref{sec:conf-coord-appr})
    \begin{equation*}
      F_j\colon \SC_j\to \CDach=\SB,
    \end{equation*}
    such that $F_j(Q)\in \X'_j$ for every $\delta_j$-square
    $Q\subset\SC_j$.
  \item The $(j+1)$-th tiling \defn{subdivides} the $j$-th tiling. This
    means that for each $(j+1)$-tile $X'_{j+1}$ there exists exactly one
    $j$-tile $X'_j\supset X'_{j+1}$.
  \item \label{item:lem_prop_tiles_comb}
    Call the images of $\delta_j$-edges/vertices under the map
    $F_j$ above $j$\defn{-edges/vertices}. View the $j$-th tiling as a
    cell complex ($j$-tiles/edges/vertices are the $2$-, $1$-, and
    $0$-cells). Then the $j$-th tiling, the
    approximation $\SC_j$, and the set of
    $j$-cylinders are combinatorially equivalent by Lemma
    \ref{lem:comb_cyl_squares}.
  \item \label{item:lem_prop_tiles_incl}
    Inclusions of
    tiles and cylinders are preserved. This means the following.
    Consider a $\delta_j$-square $Q_j\subset \SC_j$ and a
    $\delta_k$-square $Q_k\subset \SC_k$.
    Let $X_j=X_j(Q_j)\in \X_j$, $X_k=X_k(Q_k)\in \X_k$,
    and $X'_j=F_j(Q_j)\in \X'_j$, $X'_k=F_k(Q_k)\in \X'_k$ be the
    corresponding cylinders (in $\SC$) and tiles (in $\SB$). Then
    \begin{equation*}
      X_j \subset X_k \Leftrightarrow X'_j \subset X'_k. 
    \end{equation*}
  \end{enumerate}
\end{lemma}

\smallskip
A \emph{neighbor} \index{neighboring!tiles} of a $j$-tile $X'_j$ is a
$j$-tile $Y'_j$ which shares an edge with $X'_j$. 
%To emphasize that an edge or a vertex belongs to a $j$-tile we sometimes speak of $j$-edges and $j$-vertices. \index{j -@$j$-!edge}\index{j -@$j$-!vertex}

\subsection{Construction of the Map $f\colon \SC\to \SB$}\label{sec:constr-map-f}
% In the last subsection we mapped each $\delta_j$-square $Q_j\subset \SC_j$ to a $j$-tile $X'_j\subset\SB$. Thus we get a map from the cylinder with base $Q_j$ to a $j$-tile, 
% \begin{equation}
%   \label{eq:FXtoX}
%   F\colon X_j:=X_j(Q_j)\mapsto X'_j, \; F\colon \X_j\to \X'_j. 
% \end{equation}
Recall that for any $x\in \SC$ there is a sequence 
\begin{equation}
  \label{eq:Xjxseq}
  X_0\supset X_1 \supset X_2\dots,\; X_j\in \X_j,\; \bigcap X_j=\{x\}. 
\end{equation}
Consider the tiles $X'_j:=F_j(X_j)$, where $F_j$ are the maps from
Lemma \ref{lem:proptilings} (\ref{item:lem_prop_tiles_Fj}). They
satisfy by Lemma \ref{lem:proptilings}
(\ref{item:lem_prop_tiles_incl}) $X'_0\supset X'_1 \supset X'_2\dots\;$. 

\begin{lemma}
  \label{diamX0}
  The tiles shrink to a point, 
  $$\diam X'_j\to 0,\; \text{ as } j\to \infty.$$
In fact $\diam X'_j\lesssim \lambda^j$, for a (uniform) constant $\lambda<1$ (and a uniform constant $C(\lesssim)$).
\end{lemma}

We postpone the proof until the next subsection.
By the previous lemma we can now define $f\colon \SC\to \SB$ by
\begin{equation}
  \label{eq:deff}
  f(x)=x', \; \text{ where } \{x'\}=\bigcap_j X'_j.  
\end{equation}

\begin{lemma}
  \label{lem:fwelldef}
  The map $f$ is well defined.
\end{lemma}
\begin{proof}
  Given $x\in \SC$ let the sequence $(X_j)_{j\in\N}$ be as in (\ref{eq:Xjxseq}). 
  Assume now that there is a second sequence $Y_0\supset
  Y_1\supset\dots$, $Y_j\in\X_j$, satisfying $\bigcap
  Y_j=\{x\}$. 
  Then 
  \begin{equation*}
    (X_0\cap Y_0) \supset (X_1\cap Y_1)\supset\dots,
  \end{equation*}
  where each $X_j \cap Y_j$ is compact and non-empty.
  Let $X'_j:=F_j(X_j)$, $Y'_j:=F_j(Y_j)$, and $\{y'\}:=\bigcap Y'_j$.  
  By Lemma \ref{lem:proptilings}
  (\ref{item:lem_prop_tiles_comb}) and
  (\ref{item:lem_prop_tiles_incl}) 
  \begin{equation*}
    (X'_0\cap Y'_0) \supset (X'_1\cap Y'_1)\supset\dots,
  \end{equation*}
  where each $X'_j\cap Y'_j$ is compact and non-empty. Thus
  \begin{align*}
    \emptyset & \neq \bigcap (X'_j\cap Y'_j) \subset \bigcap
    Y'_j=\{y'\}  \text{ and}
    \\
    \emptyset & \neq \bigcap (X'_j\cap Y'_j) \subset \bigcap
    X'_j=\{x'\}. 
  \end{align*}
  Thus $x'=y'$.
\end{proof}

\subsection{Combinatorial Equivalence and Finiteness}
\label{sec:combequi}
The ideas in this subsection should be considered the ``guts'' of the
proof of Theorem 1A. Let $v$ be a vertex of a $j$-tile; the
\emph{$j$-degree}\index{j -@$j$-!degree}\label{not:jdegree} of $v$ is
the number of $j$-tiles containing $v$: 
\begin{equation}\label{eq:combfinite}
  \deg_j(v):=\#\{X'\in \X'_j:v\in X'\}.
\end{equation}
Consider $j$-edges and $j$-tiles of $\SC_j$ containing $v$. Note that
each such $j$-edge is incident to two $j$-tiles, and each such
$j$-tile is incident to two $j$-edges. So the number of $j$-tiles
containing $v$ is equal to the number of $j$-edges containing $v$. In
the grid $\Z^3$ 
there are $6$ edges that intersect at each vertex.     
Thus the degree of vertices is uniformly bounded, namely
\begin{equation}
  \label{eq:main1}
  \deg_j(v)\leq 6,
\end{equation}
for all vertices $v$ and numbers $j$.

\smallskip
Now consider a set of $j$-tiles 
\begin{equation}
  \label{eq:settiles1}
  \X'=\{X'_1,\dots,X'_n\},\;\text{ where } X'_1,\dots,X'_n\in \X'_j. 
\end{equation}
As before view $\X'$ as a \emph{cell complex} $\Sigma(\X')$, where $j$-tiles, $j$-edges, and $j$-vertices in $\bigcup \X'$ are the $2$-, $1$-, and $0$-cells of the cell complex. 
A second set of $k$-tiles
\begin{equation}
  \label{eq:settiles2}
  \mathbf{Y}'=\{Y'_1,\dots,Y'_n\},\;\text{ where } Y'_1,\dots,Y'_n\in \mathbf{Y}'_k, 
\end{equation}
is said to be 
\emph{combinatorially equivalent} \index{combinatorially!equivalent}
to $\X'$, if they are equivalent when viewed as cell complexes. More
precisely, there is a cell complex isomorphism 
\begin{equation}\label{eq:combeq}
  \Phi\colon \Sigma(\X')\to \Sigma(\mathbf{Y}'),
\end{equation}
which is orientation preserving. The equivalence class of combinatorially equivalent sets of tiles is called the \emph{combinatorial type} \index{combinatorial type}of $\X'$. Otherwise $\X'$ and $\mathbf{Y}'$ are called \emph{combinatorially different}. \index{combinatorially!different}
 Combinatorial equivalence implies conformal equivalence. 
\begin{lemma}
  \label{lem:conbconfeq}
  Let $\X'$ and $\mathbf{Y}'$ as above be combinatorially equivalent. Then there is a conformal map 
  \begin{equation*}
    g=g_{\X',\mathbf{Y}'}\colon \inte \bigcup \X'\to \inte\bigcup\mathbf{Y}',
  \end{equation*}
  which maps $j$-(tiles, edges, vertices) to $k$-(tiles, edges, vertices). 
\end{lemma}
\begin{proof}
  Let $\Phi$ be the cell complex isomorphism 
  in (\ref{eq:combeq}). Without loss of generality assume that
  $\Phi(X'_i)=Y'_i$, for $i=1,\dots, n$. Let $g\colon\inte X'_i\to \inte
  Y'_i$ be the conformal map, normalized by mapping each vertex $v\in
  X'_i$ to the vertex $\Phi(v)\in Y'_i$. 
  Neighboring tiles (in $\X'$ and $\mathbf{Y}'$) are the conformal image
  of \BOX. Thus if $X'_i,X'_l$ are neighbors, $g$ extends conformally
  to $\inte (X'_i\cap X'_l)$. Interior vertices are removable singularities.   
\end{proof}

The next lemma shows how one can use the tiling to define holomorphic
maps of the form $z\mapsto z^n$. It will be applied to a covering of
our conformal tilings. Recall that a conformal tiling may be viewed 
as a cell complex, 
where the $1$-cells are the (analytic) boundary arcs of the tiles.
\begin{lemma}
  \label{cor:ztozn}
  Let $V=\bigcup \{X'\in \X'\}$ and $W=\bigcup\{Y'\in \mathbf{Y}'\}$
  be two conformal 
  tilings, where each tile is a conformal square. 
  Let $v\in V$ and $w\in W$ be vertices, such that
  the degree at $v$ (number of tiles intersecting in $v$) is a
  multiple of the degree at $w$,  
  $$\deg(v)=n\deg(w),$$ 
 for some $n\in \N$. 
 Let
 \begin{align*}
   U(v)&:=\bigcup \{X'\in\X':v\in X'\} \setminus
   \bigcup\{1\text{-cells of } V \text{ not containing } v\} 
   \text{ and}
   \\
   U(w)&:=\bigcup \{Y'\in\mathbf{Y}':w\in Y'\} \setminus
   \bigcup\{1\text{-cells of } W \text{ not containing } w\}   
 \end{align*}
 be neighborhoods of $v$ and $w$.  
Then there is an analytic map 
$$U(v)\to U(w)$$ 
mapping $j$-tiles to $k$-tiles, which is conformally conjugate to
$z\mapsto z^n$.
\end{lemma}

\begin{proof}
  Label the tiles around $v$ by $X'_1,\dots, X'_{nm}$, and the tiles
  around $w$ by $Y'_1,\dots,$ $Y'_m$ positively around the vertices. 
  Map the first tile $X'_1$ conformally to $Y'_1$, such that $v$ is mapped to $w$. 
By reflection this extends conformally to map $X'_2$ to $Y'_2$. Continuing to extend the map in this fashion $X'_{nm}$ gets mapped to $Y'_{m}$. Again this extends by reflection to a conformal map from $X'_1$ to $Y'_1$, agreeing with the previous definition of the map on $X'_1$. By changing coordinates we can write the map in the form $z\mapsto z^n$.  
\end{proof}

\begin{proof}[Proof of Lemma \ref{diamX0}] 
  One way to prove the lemma would be to use the rational maps that
  can be constructed as in \cite{snowemb}.
  Since it is well known that the occurring \emph{postcritically
    finite} rational maps are \emph{sub-hyperbolic},  
  the statement is true in the \emph{orbifold metric} (see
  \cite{Carleson} and \cite{Milnor}). 

  We give a self-contained proof here. The following may in fact be
  viewed as an explicit construction of the orbifold metric. It was
  somewhat inspired by a conversation with W.\ Floyd and W.\ Parry. 
  
Consider first a uniformized generator as in Figure
\ref{fig:unifgen}. The conformal maps $g$ from the unit square to a tile are
contractions in the hyperbolic metric $d_h(x,y)$ of $\inte [0,1]^2$
by the Schwarz-Pick lemma; they are \emph{strict contractions} for
compact subsets of $\inte[0,1]^2$.  

We want to consider a neighborhood $U$ of the unit square $[0,1]^2$
so that we can extend the maps $g\colon [0,1]^2\to \text{tile}$ to
$U$. By Schwarz-Pick the map $g$ will then be strictly contracting on
the compact set $[0,1]^2\subset U$ in the hyperbolic metric of $U$.  

Let the number $M\in\N$ be the least common multiple of all occurring
degrees $\deg_j(v)$ (recall that this was the number of $j$-tiles
intersecting in a vertex $v$).  
It is well known that the hyperbolic plane can be tiled with
hyperbolic squares with angles $2\pi/M$ if $M\geq 5$ (see \cite{caratheodory2},
sections 398--400).  
Alternatively one may construct a cell complex
consisting of squares where at each vertex $M$ squares intersect, put
a conformal structure on the complex (as in Subsection
\ref{sec:conf-coord-appr}), and invoke the uniformization 
theorem (it is not hard to show that the \emph{type} will be
hyperbolic).   
Let $Q$ be one hyperbolic square of the tiling, and $U$ be the
neighborhood consisting of all hyperbolic squares with non-empty
intersection with $Q$. The hyperbolic squares in $U$ form a conformal
tiling. Each vertex of $Q$ belongs to $M$ tiles.  

\medskip
Now consider a uniformized generator, which is a conformal tiling of
the unit square $[0,1]^2$ as in Figure \ref{fig:unifgen}. Map this
tiling by conformal maps to each hyperbolic square in $U$. Images of
the tiles of $[0,1]^2$ under these maps will be denoted by $T$. The
tiles $T$ are again a conformal tiling of $U$. 

Let $g_T$ be a conformal map from the hyperbolic square $Q$ to such a
tile, 
\begin{equation}
  \label{eq:defgT}
  g_T\colon Q\to T\subsetneq Q.   
\end{equation}
By the previous lemma $g_T$ extends
to $U$ analytically, $g_T\colon U\to U$. Since $Q$ is compactly
contained in $U$, the map $g_T$ is strictly contracting on $Q$ in
the hyperbolic metric $d_U$ of $U$ (by Schwarz-Pick, see for example
\cite{AhlforsConfInv}): 
$$d_U(g(x_T),g(y_T))\leq \lambda_T d_U(x,y),\; \text{where } \lambda_T<1 \text{ for all } x,y\in Q .$$

Since there are only finitely many different generators (each with
finitely many squares/tiles), all these  maps are contracting with a
uniform constant $\lambda<1$.  

Consider a $0$-tile $X'_0\in\X_0'$. Let $V$ be the neighborhood of all
$0$-tiles having non-empty intersection with $X'_0$. As before we can
extend the conformal map $h\colon Q\to X'_0$ to an analytic map
$h\colon U\to V$. Since $X'_0$ is compactly contained in $V$, and by
Schwarz-Pick, 
\begin{equation*}
  \abs{h(x)-h(y)}\asymp d_V(h(x),h(y))\leq d_U(x,y),\text{ for all } x,y\in Q, 
\end{equation*}
where $d_V$ denotes the hyperbolic metric of $V$. 

Now consider a $j$-tile $X'_j\subset X'_{j-1}\subset \dots \subset
X'_0$, where $X'_k\in\X'_k$ for $0\leq k\leq j$. 
Let $Y'_k:=h^{-1}(X'_k)\subset Q$ be their preimages. 

Set $T_1:=Y'_1$. Since $Y'_2\subset T_1$, we can let
$T_2:=g^{-1}_{T_1}(Y'_2)$; the map $g_{T_1}$ is the one from
(\ref{eq:defgT}). Define inductively 
\begin{equation*}
  T_{k}:=g^{-1}_{T_{k-1}}\circ \dots \circ g^{-1}_{T_1}(Y'_{k}),
\end{equation*}
for $1\leq k\leq j$. 
Note that $Y'_k\subset T_1, g^{-1}_{T_1}(Y'_k)\subset T_2,
g^{-1}_{T_2}\circ g^{-1}_{T_1}(Y'_{k})\subset T_3$, and so on. Thus
$T_k$ is well defined. 

Note also that $T_k$ is one of the (finitely many) tiles as
above. This is seen as follows. Consider all $k$-tiles
$\widetilde{X}'_k\subset X'_{k-1}$ and the corresponding sets
$\widetilde{Y}'_k, \widetilde{T}_k$. Then the sets
$\widetilde{T}_k\subset Q$ are the conformal image of the tiling of
$[0,1]^2$ obtained as the uniformization of the $N_k$-generator.  
Then 
\begin{align*}
  g_{T_1}\circ\dots \circ g_{T_j}(Q)=g_{T_1}\circ\dots \circ
  g_{T_{j-1}}(T_j)= Y'_j.   
\end{align*}

For $x'=h(x),y'=h(y)\in X'_j$, where $x,y\in Y'_j\subset Q$, we have by
the above $$\abs{x'-y'}\lesssim d_U(x,y)\lesssim\lambda^j.$$ The
result follows.   
\end{proof}

\subsection{Combinatorial Distance on $\SB$}
\label{sec:comb-dist-sb}

Recall how $j(x,y)$ was defined in (\ref{eq:defjxy}) by the
combinatorics of cylinders (of the snowsphere). Since tiles (of the
sphere) have the same combinatorics, we
write 
$$j(x',y')=j(x,y),$$
where $x'=f(x),y'=f(y)$.

The proof of Theorem 1A follows essentially from the next two
lemmas. The first concerns \emph{intersecting} $j$-tiles, thus the case $j<
j(x',y')$; see (\ref{eq:defjxy}). In the second we consider \emph{disjoint}
$j$-tiles, thus the case $j\geq j(x',y')$. 
The proofs are essentially the same. In each case one has to
control only finitely many combinatorial types by
(\ref{eq:main1}). Since combinatorial equivalence implies conformal
equivalence by Lemma \ref{lem:conbconfeq}, sets of the same type
cannot ``look too different'' by the Koebe distortion theorem. To
paraphrase the main idea of the proof, why do constants not blow up?
Because there are only finitely many constants, one for each
combinatorial type of suitable neighborhoods.  
\begin{lemma}
  \label{lem:diamXY}
  Let $X'$, $Y'$ be $j$-tiles that are not disjoint. Then
  \begin{equation*}
    \diam X' \asymp \diam Y',
  \end{equation*}
  with a uniform constant $C(\asymp)$.
\end{lemma}
\begin{proof}
  Let $X',Y'\in\X'_j,$ $X'\cap Y'\neq\emptyset$.
  Consider the set of tiles
  \begin{equation*}
    \mathbf{Z'}:=\{Z'\in \X'_j: Z'\cap (X'\cup Y')\neq \emptyset\}. 
%    \mathbf{Z'}:=\bigcup_{\substack{Z'\cap (X'\cup Y')\neq \emptyset\\Z'\in \X_j}} Z'.
  \end{equation*}

  There are only finitely many different combinatorial types of such
  $\mathbf{Z}'$ by inequality (\ref{eq:main1}). Thus there are only
  finitely many different conformal types of such $\mathbf{Z}'$ (by Lemma
  \ref{lem:conbconfeq}). 
  In general $\bigcup \mathbf{Z}'$ is not simply connected. 
  Fix simply connected open neighborhoods $U=U_{\mathbf{Z}'}\subset\bigcup
  \mathbf{Z}'$ of $X'\cup Y'$, 
  and Riemann maps $h=h_{\mathbf{Z}'} \colon \D \to U$ with $h(0)\in
  X'\cap Y'$. We can choose $h$ and $U$ \defn{compatible}  
  with the conformal
  equivalence. By this we mean that if $\mathbf{Z}'$ and
  $\widetilde{\mathbf{Z}}'$ are combinatorially equivalent and $g_{\mathbf{Z}',
      \widetilde{\mathbf{Z}}'}$ is the map from  Lemma
    \ref{lem:conbconfeq}, then
  \begin{align*}
    U_{\widetilde{\mathbf{Z}}'} & =g_{\mathbf{Z}',
      \widetilde{\mathbf{Z}}'} U_{\mathbf{Z}'}, 
    & h_{\widetilde{\mathbf{Z}}'} & 
    =g_{\mathbf{Z}',\widetilde{\mathbf{Z}}'} \circ h_{\mathbf{Z}'}.
  \end{align*}
  Consider preimages of $X'$ and $Y'$ by $g$ in the disk
  $\D$; they are compactly contained. There are only finitely many
  different such preimages, one for each 
  combinatorial type of $\mathbf{Z}'$. Thus
  \begin{align*}
    \diam g^{-1}(X')\asymp \diam g^{-1}(Y') \text{ and}
    \\ 
    \dist(g^{-1}(X'\cup Y'),\partial \D)\geq  \epsilon>0.
  \end{align*}
  Here $C(\asymp)$ and $\epsilon$ are uniform constants.
  The statement now follows from Koebe's distortion theorem (see for
  example \cite{AhlforsConfInv}). 
\end{proof}
Since the number of $(j+1)$-tiles that a $j$-tile contains is
uniformly bounded, one immediately concludes the following corollary. 
\begin{cor}\label{cor:XjXj+1}
  For any $(j+1)$-tile $X'_{j+1}\subset X'_{j}\in \X'_j$, we have 
  \begin{equation*}
    \diam X'_{j+1}\asymp \diam X'_j,
  \end{equation*}
  where $C(\asymp)$ is a uniform constant.
\end{cor}
\begin{lemma}
  \label{lem:distXY}
  Let $X',Y'$ be disjoint $j$-tiles. Then
  \begin{equation*}
    \dist(X',Y')\gtrsim \diam X',
  \end{equation*}
  with a uniform constant $C(\gtrsim)$. 
\end{lemma}
\begin{proof}
  Consider the neighborhood of $j$-tiles of $X'\in \X_j$,
  \begin{equation*}
    \X':=\{Z'\in \X'_j: Z'\cap X'\neq \emptyset\}.
  \end{equation*}

  The set $\bigcup \mathbf{X}'$ is simply connected. 
  There are only finitely many different combinatorial types of such
  $\mathbf{X}'$ (by inequality (\ref{eq:main1})). 
  Assume that the tiling on $\D$
  induced by the Riemann maps $g=g_{\mathbf{X}'}\colon \D \to \inte
  \bigcup \mathbf{X}'$ (with $g(0)\in X'$) depends only on the type of
  $\mathbf{X}'$ (by Lemma \ref{lem:conbconfeq}). Then 
  \begin{align*}
    & \dist(g^{-1}(X'), \partial \D) \geq \epsilon,
%    \\
%    &\dist(g^{-1}(X'),g^{-1}(\partial \bigcup \X'))\asymp 1 \text { and}
%    \\
%    &\dist(g^{-1}(\partial \bigcup \X'),\partial \D)\geq \epsilon>0.
  \end{align*}
  where $\epsilon>0$ is a uniform constant. So by Koebe distortion,
  \begin{equation*}
    \dist(X',Y')\geq \dist(X',\partial \bigcup \X')\asymp
    \abs{g'_{\X'}(0)}\gtrsim \diam X'.
  \end{equation*}
\end{proof}

The last two lemmas enable us to describe distances in combinatorial
terms. 

\begin{lemma}
  \label{lem:combdistS}
  
  For all $x',y'\in \SB$
  \begin{equation*}
    \abs{x'-y'} \asymp \diam X'_j, 
  \end{equation*}
  where $j=j(x',y')$, $x'\in X'_j\in\X'_j$. The constant
  $C(\asymp)$ is uniform.
\end{lemma}

\begin{proof}
  Let $x',y'\in \SB$ be arbitrary, $j=j(x',y')$. Then $(j-1)$-tiles
  $X'_{j-1}\ni x', Y'_{j-1}\ni y'$ are not disjoint. Thus,
  \begin{align*}
    \abs{x'-y'} & \leq\diam X'_{j-1} + \diam Y'_{j-1} &&
    \\
    & \lesssim \diam X'_{j-1} && \text{ by Lemma \ref{lem:diamXY}}
    \\
    & \lesssim \diam X'_j  && \text{ by Corollary \ref{cor:XjXj+1}}. 
  \end{align*}
  
  On the other hand there are disjoint $j$-tiles $X'_j\ni x', Y'_j\ni
  y'$. Therefore by Lemma \ref{lem:distXY},
  \begin{align*}
    \abs{x'-y'}\geq \dist(X'_j, Y'_j) \gtrsim \diam X'_j.
  \end{align*}
\end{proof}

The following is an immediate consequence of Lemma \ref{lem:pseudcomp},
Lemma \ref{diamX0}, and Lemma \ref{lem:combdistS}.
\begin{cor}\label{cor:S_S2}
  The map $f\colon \SC \to \SB$ is a homeomorphism. In particular
  $\SC$ is a topological sphere.
\end{cor}

\subsection{Proof of Theorem 1A}
\label{sec:proofsphere}
To show that spaces are quasisymmetrically equivalent can be
tedious. Therefore one often considers the following weaker
notion. 
An embedding $f\colon X\to Y$ of metric spaces is called \emph{weakly
  quasisymmetric} \index{weak quasisymmetry} if there is a number
$H\geq 1$ such that
\begin{equation*}
  |x-a|\leq |x-b| \Rightarrow |f(x)-f(a)|\leq H |f(x)-f(b)|,
\end{equation*}
for all $x,a,b\in X$. Quasisymmetric maps are ``more nicely'' behaved
than weakly quasisymmetric ones. Quasisymmetry is preserved under
compositions and inverses, which do not preserve weak quasisymmetry
in general. In many practically relevant cases however the two notions
coincide. 

A metric space is called \emph{doubling} \index{doubling (metrically)} if there is a number $M$ (the \emph{doubling constant}), such that every ball of diameter $d$ can be covered by $M$ sets of diameter at most $d/2$, for all $d>0$.
\begin{theorem*}[see \cite{JuhAn}, 10.19]
  A weakly quasisymmetric homeomorphism of a connected doubling space
  into a doubling space is quasisymmetric. 
\end{theorem*}

Obviously $\SC$ is connected. The snowsphere $\SC$ (as well as $\SB$)
is doubling as a subspace of $\R^3$.

\begin{proof}[Proof of Theorem 1A]
  We want to show that the map 
  \begin{equation*}
    f\colon \SC\to \SB,
  \end{equation*}
   defined in Subsection \ref{sec:constr-map-f}, is quasisymmetric.
  By the above it is enough to show weak quasisymmetry. Let $x,y,z\in
  \SC$, $j:=j(x,y), k:=j(x,z)$ (see (\ref{eq:defjxy})). Assume 
  \begin{align}
    \notag
    \abs{x-y} & \leq \abs{x-z}.  
%    \\ \notag
    \intertext{Then by Lemma \ref{lem:pseudcomp}}
    \label{eq:pf1A1}
    \delta_j & \lesssim \delta_k.
  \end{align}
  Let $C=C(\lesssim)=C(N_{\max})$ and choose an integer $k_0=k_0(N_{\max})$
  such that $2^{k_0}\geq C$. Then (\ref{eq:pf1A1}) implies
  \begin{equation*}
    j \geq k - k_0,
  \end{equation*}
  since $N_i\geq 2$ for all $i$.
  Lemma \ref{lem:combdistS} yields
  \begin{align*}
    \abs{x'-y'} & \asymp \diam X'_j, && \text{where } x'\in
    X'_j\in\X'_j.
    \\
    \intertext{If $k-k_0 \geq 0$ let $X'_j\subset X'_{k-k_0}\in\X'_{k-k_0}$; 
      otherwise set $X'_{k-k_0}:= \SB$. Then}
    \abs{x'-y'} & \lesssim \diam X'_{k-k_0} \asymp \diam X'_k,     
    \intertext{by Corollary \ref{cor:XjXj+1}, where $X'_{k-k_0}\supset
      X'_k\in\X'_k$, and so} 
    \\
    \abs{x'-y'} & \lesssim \abs{x'-z'}.
  \end{align*}

\end{proof}

\begin{remarks}
  It is possible to define snowspheres abstractly, i.e., not as subsets
  of $\R^3$. They will still be quasisymmetrically equivalent to the
  standard sphere $\SB$ as long as
  \begin{itemize}
  \item each generator $G_j$ is symmetric,
  \item the number of $N_j$-squares in a generator $G_j$ is bounded,
  \item the number of $\delta_j$-squares intersecting in a vertex stays uniformly bounded throughout the construction.
  \end{itemize}
  Since ultimately our goal is to show that snowspheres are
  quasisymmetric images of the sphere $\SB$ by global quasisymmetric
  maps $f\colon \R^3\to \R^3$, we do not pursue this further.
  
  \medskip
  Yet other variants of snowspheres are obtained by starting with a
  tetrahedron, octahedron, or icosahedron. A generator would then be a
  polyhedral surface built from small equilateral triangles, whose
  boundary is an equilateral triangle. While it is not too hard to
  check in individual cases whether the resulting snowspheres have
  self-intersections (i.e., are topological spheres), we are not aware of a general condition analogous
  to the ``double pyramid'' condition. This is the main reason why we
  focus on the ``square'' case. 
\end{remarks}

% ==Extending the map
%\input{ext}
%
% ext.tex
%
% Chapter about extending the map snowsphere -> sphere
% to a map R^3 -> R^3
%

\section{Elementary bi-Lipschitz Maps and Extensions}\label{sec:simplices-extensions}

This section provides several maps that are needed in the extension of
the map $f$,
i.e., in the proof of Theorem 1B.
The reader may first want to skip it and return when a
particular construction is needed. 

We will decompose the interior of the snowball into several
standard pieces. These will be mapped into the unit cube $[0,1]^3$. We
provide these maps here together with estimates to ensure that
constants are controlled. 

\smallskip
For planar vectors $v,w$ write $[v,w]:=\det(v,w)$. 
Recall that $$\sin\angle(v,w)=[v,w]/(\abs{v}\abs{w}).$$
Consider a planar
quadrilateral $Q$ with vertices $P_0,P_1,P_2,P_3$
(counterclockwise). We assume that $Q$ is \emph{strictly
  convex}. This is the case if and only if
\begin{equation}
  \label{eq:Qconvex}
  J:=\min_j [\overrightarrow{P_jP}_{j+1},\overrightarrow{P_jP}_{j-1}]>0,
\end{equation}
where indices are taken $\bmod 4$.
Consider the vectors 
\begin{align}
  \label{eq:defvtwt}
  v(t) & := (1-t)\overrightarrow{P_0 P_1} + t\overrightarrow{P_3 P_2}
  \quad\text{and}
  \\ \notag
  w(s) & := (1-s)\overrightarrow{P_0 P_3} + s\overrightarrow{P_1 P_2},
\end{align}
for $s,t\in[0,1]$, ``which connect opposite sides'' of $Q$. 
Note that (\ref{eq:Qconvex}) is equivalent to 
\begin{equation}
  \label{eq:Qconvex2}
[v(t),w(s)]\geq J \text{ for all }
(s,t)\in [0,1]^2.   
\end{equation}

Map the unit
square to $Q$ by
\begin{align}
  \label{eq:defQst}
  Q(s,t)  &:=P_0 + sv(0) + t w(s)
  \\ \notag
  & \phantom{:}= P_0 + t w(0) + s v(t).
\end{align}

\begin{lemma}
  \label{lem:QbiLip}
  Let the quadrilateral $Q$ be
  strictly convex as in \eqref{eq:Qconvex}.
  Then the map from equation \eqref{eq:defQst} is bi-Lipschitz.
\end{lemma}
\begin{proof}
  One computes 
  \begin{equation}\label{eq:Qstst}
    Q(s',t')-Q(s,t)=(s'-s)v(t) + (t'-t)w(s').
  \end{equation}
%   This immediately shows that $Q$ is quasiconformal. Indeed the Jacobien
%   is $J_Q=[v(t),w(s)]\geq J$, and we obtain for the derivative $\norm{D_Q}\leq
%   \max\{\norm{v(t)},\norm{w(s)}\} \leq \diam Q$. 
  We obtain from equation (\ref{eq:Qstst})
  \begin{equation*}
    \abs{ Q(s',t') - Q(s,t)}\leq \diam Q(\abs{s'-s}+\abs{t'-t}).
  \end{equation*}
  On the other hand note that $\abs{v}\geq\Abs{\,[v,u]\,}$ for any unit
  vector $u$. Choosing $v=Q(s',t')-Q(s,t)$,
  $u:=v(t)/\abs{v(t)}$, and
  $u':=w(s')/\abs{w(s')}$ one thus
  obtains from (\ref{eq:Qstst}) and (\ref{eq:Qconvex2})
  \begin{equation*}
    \abs{Q(s',t')-Q(s,t)}\geq
    \frac{J}{\diam Q} \max\{\abs{s'-s},\abs{t'-t}\}.
  \end{equation*}
\end{proof}
 
Now consider  
two quadrilaterals lying in parallel planes, $Q^0\subset
\{z=0\}, Q^1\subset\{z=1\}$.
The quadrilaterals $Q^u$ are given by vertices
$P_0^u,P_1^u,P_2^u,P_3^u\subset\{z=u\}$, $u=0,1$,
counterclockwise.

The points $P^u_j:= (1-u)P_j^0 + uP_j^1$, $u\in[0,1]$ define quadrilaterals
$Q^u\subset\{z=u\}$ as before. Again they are strictly convex if
\begin{equation}
  \label{eq:Quconvex}
  J :=   \min_{\substack{0\leq j\leq 3 \\  u\in[0,1]}}
  \left[\overrightarrow{P^u_jP^u}_{j+1},\overrightarrow{P^u_j,P^u}_{j-1}\right]
  > 0. 
\end{equation}
Using the points $P^u_j$ define maps
$v^u(t),w^u(s)$, and $Q^u(s,t)$
as above in equations
(\ref{eq:defvtwt}) and
(\ref{eq:defQst}). Let 
\begin{equation}
  \label{eq:defBQQ}
  B:=\bigcup_{u\in [0,1]} Q^u.   
\end{equation}
A map from the
unit cube $[0,1]^3$ to $B$ is given by
\begin{equation}
  \label{eq:defBstu}
  B(s,t,u):=Q^u(s,t)= (1-u)Q^0(s,t)+ u Q^1(s,t).
\end{equation}

\begin{lemma}
  \label{lem:BbiLip}
  Let the quadrilaterals $Q^u$ be strictly convex as in
  \eqref{eq:Quconvex}. Then the map defined in 
  \eqref{eq:defBstu} is bi-Lipschitz.
\end{lemma}

\begin{proof}
  Compute
  \begin{align}\label{eq:Bstu}
    B(s',t',u')&-B(s,t,u) = B(s',t',u')-B(s',t',u) 
     + B(s',t',u)-B(s,t,u)
%    \\ \notag
%    & \phantom{-B(s,t,u)X}
    \\ \notag
    &= (u'-u)\left(Q^1(s',t')-Q^0(s',t')\right) 
    + Q^u(s',t')-Q^u(s,t)
    \\ \notag
    & = (u'-u)\left(Q^1(s',t')-Q^0(s',t')\right) 
    \\ \notag
    & \phantom{XX} + (s'-s) v^u(t) + (t'-t) w^u (s'),
  \end{align}
  as in equation (\ref{eq:Qstst}). Thus
  \begin{equation*}
    \abs{B(s',t',u')-B(s,t,u)}\leq \diam B
    \left(\abs{s'-s}+\abs{t'-t}+\abs{u'-u} \right).
  \end{equation*}

  To see the other inequality note first that
  $$\det\left(Q^1(s',t')-Q^0(s',t'),v^u(t),w^u(s')\right)=
  [v^u(t),w^u(s')]\geq J>0,$$
  where the constant $J$ is given by (\ref{eq:Quconvex}) (use also (\ref{eq:Qconvex2})). 

  Recall that
  $\abs{v}\geq \abs{\det(v,a,b)}/(\abs{a}\abs{b})$ for all non-zero vectors $v,a,b$. Choosing
  $v=B(s',t',u')-B(s,t,u)$ and $a,b$ two of the vectors
  $Q^1(s',t')-Q^0(s',t'),
  v^u(t)$, $w^u(s')$ 
we obtain
  from equation (\ref{eq:Bstu})
  \begin{equation*}
     \abs{B(s',t',u')-B(s,t,u)}\geq \frac{J}{(\diam B)^2} \max\{\abs{s'-s},\abs{t'-t},\abs{u'-u}\}.
  \end{equation*}
\end{proof}

The \defn{Alexander trick} consists in extending a homeomorphism from
the closed disk to an isotopy. More precisely let 
$\varphi\colon \overline{\D}=\{\abs{z}\leq 1\}\to \overline{\D}$ be a
homeomorphism satisfying $\varphi|_{\partial \D} = \id$. Then the
homeomorphism  
\begin{align}
  \label{eq:Alextrick}
  \overline{\varphi}&\colon \overline{\D}\times [0,1]\to
  \overline{\D}\times [0,1], \text{ defined by }
  \\ \notag
  \overline{\varphi}&(z,t):=
  \begin{cases}
    t\varphi(z/t), &0\leq\abs{z}\leq t;
    \\
    z,             &t\leq \abs{z}\leq 1,
  \end{cases}
\end{align}
satisfies $\overline{\varphi}|_{\overline{\D}\times \{1\}}= \varphi$,
and $\overline{\varphi}=\id$ on the rest of $\partial
(\overline{\D}\times [0,1])$. 
It is easy to check that if $\varphi$ is bi-Lipschitz the extension
$\overline{\varphi}$ is as well, using
\begin{align*}
  \overline{\varphi}^{-1}&(z,t):=
  \begin{cases}
    t\varphi^{-1}(z/t), &0\leq\abs{z}\leq t;
    \\
    z,             &t\leq \abs{z}\leq 1.
  \end{cases}
\end{align*}

\smallskip
Recall the radial extension, which is only presented in the
form we will need.
Let $\varphi\colon \partial \D \to \partial \D,
\varphi(e^{i\theta})=e^{i\varphi(\theta)}$ be a homeomorphism fixing
$1,i,-1,-i$. Let $\varphi_t(\theta):= (1-t)\theta +t
\varphi(\theta)$. Then the homeomorphism
\begin{align}
  \label{eq:defradialext}
  \overline{\varphi}&\colon \overline{\D}\times [0,1]\to
  \overline{\D}\times [0,1], \text{ defined by }
  \\ \notag
  \overline{\varphi}&(r e^{i\theta},t):= (re^{i \varphi_t(\theta)},t),
\end{align}
satisfies $\overline{\varphi}|_{\overline{\D}\times
  \{0\}}=\id$. 
\begin{lemma}
  \label{lem:radialext}
  Let $\varphi$ be bi-Lipschitz. Then the extension
  $\overline{\varphi}$ from \eqref{eq:defradialext} is bi-Lipschitz as
  well. 
\end{lemma}
\begin{proof}
  It is easy to verify that
  \begin{equation*}
    \abs{r_2e^{i\theta_2}-r_1e^{i\theta_1}}\asymp \abs{r_2 - r_1} +
    r_1\abs{\theta_2-\theta_1}, 
  \end{equation*}
  for $\abs{\theta_1 -\theta_2} \leq \pi$ and $r_1,r_2\geq 0$. Let $0 \leq \theta_1 \leq
  \theta_2\leq \pi$. Then 
  \begin{equation*}
    \varphi(\theta_2)-\varphi(\theta_1)\asymp \theta_2-\theta_1,
  \end{equation*}
  since $\varphi$ is bi-Lipschitz and orientation preserving. Thus
  \begin{align*}
    \abs{r_2e^{i\varphi_t(\theta_2)} - & r_1e^{i\varphi_t(\theta_1)}} 
    \\
    \asymp & \abs{r_2-r_1} + r_1\Abs{(1-t)(\theta_2 - \theta_1) + t
      (\varphi(\theta_2) - \varphi(\theta_1))}
    \\
    \asymp & \abs{r_2-r_1} + r_1\Abs{\theta_2 - \theta_1}
    \asymp \abs{r_2e^{i\theta_2} - r_1 e^{i\theta_1}}.
  \end{align*}
  The claim follows.
\end{proof}

\smallskip
Combine the two extensions, and map the disk to the square 
to get the following variant. 
\begin{lemma}
  \label{lem:AlexanderTrick}
  Let $\varphi
  \colon[0,1]^2\to [0,1]^2$ be bi-Lipschitz fixing the vertices. 
  Then there is a bi-Lipschitz map
  \begin{equation*}
    \overline{\varphi} \colon [0,1]^3\to [0,1]^3,  
  \end{equation*}
  such that $\overline{\varphi}|_{[0,1]^2\times \{0\}}=\id$ and
  $\overline{\varphi}|_{[0,1]^2\times \{1\}}= \varphi$. Furthermore the
  extensions are compatible on neighbors in the following sense. Let
  $\varphi'\colon [1,2]\times [0,1]\to [1,2]\times [0,1]$ be another
  bi-Lipschitz map fixing the vertices such that $\varphi=\varphi'$ on
  the intersecting edge $\{1\}\times [0,1]$. Then the extensions
  $\overline{\varphi}$ and $\overline{\varphi}'$ agree on the
  intersecting side $\{1\}\times [0,1]^2$.
\end{lemma}

\begin{proof}
  Use the radial extension (\ref{eq:defradialext}) to construct a
  bi-Lipschitz map 
  \begin{align*}
    & \psi\colon [0,1]^2\times \Bigl[0,\frac{1}{2}\Bigr] \to
    [0,1]^2\times \Bigl[0,\frac{1}{2}\Bigr], \text{ such that}
    \\
    & \psi = \varphi \text{ on }\partial
    [0,1]^2 \times \Bigl\{\frac{1}{2}\Bigr\} \text{ and }
    \psi|_{z=0}=\id.     
  \end{align*}
  Use the Alexander trick (\ref{eq:Alextrick}) to construct a
  bi-Lipschitz map
  \begin{align*}
    & \phi\colon [0,1]^2\times \Bigl[\frac{1}{2},1\Bigr] \to [0,1]^2\times
    \Bigl[\frac{1}{2},1\Bigr], \text{ such that}
    \\
    & \phi|_{\{z=\frac{1}{2}\}} =
  \psi|_{\{z=\frac{1}{2}\}} \text{ and } \phi|_{\{z=1\}} =
  \varphi.   
  \end{align*}
  
  Combining $\psi$ and $\phi$ gives the extension $\overline{\varphi}$.   
\end{proof}

Let $(\omega,\rho), \omega\in \SB, \rho\geq 0$, be spherical
coordinates in $\R^3$. The Euclidean distance of points thus given is
controlled by 
\begin{equation}
  \label{eq:distwr}
  \Abs{(\omega,\rho)-(\omega',\rho')}\asymp \abs{\rho-\rho'} + \rho\abs{\omega
    - \omega'}. 
\end{equation}
The same argument as in Lemma \ref{lem:radialext} gives an extension
from the sphere to the ball.
\begin{lemma}
  \label{lem:extSB}
  Let $\psi\colon \SB\to \SB$ be bi-Lipschitz. Then the radial
  extension
  \begin{equation*}
    \overline{\psi}\colon \B\to \B, \quad \overline{\psi}(\omega,\rho):= (\psi(\omega),\rho)
  \end{equation*}
  is bi-Lipschitz. Here $(\omega,\rho)$ are spherical coordinates. 
\end{lemma}

\bigskip
The next extension lemma will be used to map the cube $[0,1]^3$. 

\begin{lemma}
  \label{lem:prism1}
  Let $X$ be a metric space (with metric denoted by $\abs{x-y}$).
  Let $\varphi \colon [0,1]^2\to X$ be bi-Lipschitz, and let
  $\rho_0\colon X\to \R$ and $\rho_1 \colon X\to \R$ be Lipschitz
  (the maps $\varphi,\rho_0,\rho_1$
  have a common (bi-)Lipschitz constant $L$), such that 
  \begin{equation*}
    \rho_0(x)+m\leq \rho_1(x)\leq \rho_0(x)+M,
  \end{equation*}
  for all $x\in X$ and constants $m,M>0$. Then the map
  $\overline{\varphi}\colon [0,1]^3 \to X\times \R$ defined by
  \begin{align*}
    \overline{\varphi}(x,t) &
    :=\big(\varphi(x),(1-t)\rho_0(\varphi(x))+t\rho_1(\varphi(x))\big),
    \\
    & \phantom{XXX}\text{ for all }x\in [0,1]^2, t\in [0,1] 
  \end{align*}
  is bi-Lipschitz with constant $\bar{L}=\bar{L}(L,M,m)$. Here we are
  using the maximum metric on $X\times \R$.
\end{lemma}
\begin{proof}
  Extension of the map $\varphi$ to $\widetilde{\varphi} \colon [0,1]^3 \to
  X\times [0,1]$ by $\widetilde{\varphi}(x,t):=(\varphi(x),t)$ is
  trivially bi-Lipschitz. It remains to show
  that the map $\phi\colon X\times [0,1]\to X\times \R$ defined by 
  \begin{align*}
    \phi(x,t)&:=(x,(1-t)\rho_0(x)+t\rho_1(x)),
    \\
    \text{with }     &\rho_0(x)+m\leq \rho_1(x)\leq \rho_0(x)+M,
  \end{align*}
  is bi-Lipschitz. 
  % Note that these $\rho_j$'s differ from the ones in the statement of the theorem by a composition with $\varphi$. 
  For any $x,y\in X$, $s,t\in [0,1]$ we have 
  \begin{align*}
    \abs{(1-t)\rho_0(x)&+t\rho_1(x)-(1-s)\rho_0(y)-s\rho_1(y)}
    \\
    &\leq (1-t)\abs{\rho_0(x)-\rho_0(y)}+t\abs{\rho_1(x)-\rho_1(y)} 
    \\
    & \phantom{XXXXXX} +|t-s|\abs{\rho_1(y)-\rho_0(y)}
    \\
    &\leq L\abs{x-y} + \abs{t-s}M.
  \end{align*}
  For the reverse inequality let $\phi(x,t)=(x,u)$ and
  $\phi(y,s)=(y,v)$. We have $t=\frac{u-\rho_0(x)}{\rho_1(x)-\rho_0(x)},
  s=\frac{v-\rho_0(y)}{\rho_1(y)-\rho_0(y)}$, where $u-\rho_0(x)\leq
  M$. Then
  \begin{align*}
    \abs{t-s}& 
    \leq
    \Abs{\frac{\left(u-\rho_0(x)\right)\left(\rho_1(y)-\rho_0(y)\right) 
        - \left(v-\rho_0(y)\right)\left(\rho_1(x)-\rho_0(x)\right)}
      {\left(\rho_1(x)-\rho_0(x)\right) 
        \left(\rho_1(y)-\rho_0(x)\right)}
    }
    \\
    & \leq 
    \frac{1}{m^2}\Big[(u-\rho_0(x))\Abs{\rho_1(y)-\rho_0(y)-\rho_1(x)+\rho_0(x)}
    \\
    &\;\;\; +\Abs{u-\rho_0(x)-v+\rho_0(x)}\left(\rho_1(x)-\rho_0(x)\right)
    \Big]
    \\
    & \leq \frac{M2L}{m^2}\abs{x-y} + \frac{M}{m^2}\abs{u-v}.
    \\
    \intertext{Hence}
    \abs{\phi^{-1}(x,u)&-\phi^{-1}(y,v)} \leq \abs{x-y} + \abs{t-s}
    \leq \left(\frac{M2L}{m^2}+1\right)\abs{x-y} +
    \frac{M}{m^2}\abs{u-v}. 
  \end{align*}

\end{proof}
We will map the sets $\overline{\varphi}([0,1]^3) \subset X\times\R$
in the unit ball, using spherical coordinates. The next lemma follows
immediately from (\ref{eq:distwr}). 
\begin{lemma}
  \label{lem:prismS}
  Let $0<r< R<\infty$ and $\psi\colon X\to \SB$ be
  $L$-bi-Lipschitz. Then the map 
  \begin{align*}
    \overline{\psi}\colon X\times [r,R] \to& \{(\omega,\rho) :
    \omega\in \SB, \rho\geq 0\} =\R^3,
    \intertext{given by}
    \overline{\psi}(x&,t)= \left(\psi(x),t\right),
  \end{align*}
  is $\bar{L}$-bi-Lipschitz, where $\bar{L}=\bar{L}(L,r,R)$. The right
  hand side is denoted in spherical coordinates.
\end{lemma}

\bigskip
A $2$\emph{-simplex} \index{simplex} 
is given by
\begin{equation}\label{not:simplex} 
  \Delta:=\{x=x_0 e_0 + x_1 e_1 + x_2 e_2 :0\leq x_k\leq 1,\;x_0+x_1+x_2=1\},
\end{equation}
where the $e_k\in\R^m$ ($m\geq 2$) do not lie on a line.
It is often convenient to consider the following metric on $\Delta$:\label{not:simpmetr}\index{metric!on simplex} 
\begin{equation}\label{eq:defDmetric}
  \norm{x-y}_\Delta := \max_{0\leq k \leq 2} \abs{x_k-y_k}.
\end{equation}
An easy computation shows that the map
$(\Delta,\norm{x-y}_\Delta)\to(\Delta,\|x-y\|_\infty)$ is bi-Lipschitz with
constant $\max\{\diam \Delta,\frac{\sqrt{3}}{h}\}$. Here $h$ denotes the
smallest distance of a vertex $e_k$ from the line through the other
two points.

\section{Decomposing the Snowball}
\label{sec:decomposing-snowball}

\subsection{Introduction}\label{sec:ext-intro}
In this and the next section we extend the map $f\colon \SC\to\SB$ to
$f\colon\R^3\to\R^3$.
The snowball will be decomposed in a Whitney-type fashion. Each piece
is mapped into the unit ball by a
\emph{quasisimilarity}\index{quasisimilarity}. This means that it is
bi-Lipschitz up to scaling; more precisely there are constants $L\geq 1$
and $l>0$ such that 
\begin{equation}\label{eq:defquasisym}
  \frac{1}{L}\abs{x-y}\leq \frac{1}{l} \abs{f(x)-f(y)}\leq L\abs{x-y}.
\end{equation}
The \emph{Lipschitz constant} $L$ will be the same for every piece,
while the \emph{scaling factor} $l$ will depend on the given piece. It
then follows directly from the definition (\ref{eq:defqc}) that $f$ is
quasiconformal. 

Let $f,g$ be quasisimilarities with Lipschitz constants
$L,L'$ and scaling factors $l,l'$. It follows immediately that the
composition $f\circ g$ is a quasisimilarity with 
Lipschitz constant $LL'$ and scaling factor $ll'$.

In this section the snowball $\BC$ is
decomposed. We break up $\BC$ into shells bounded by
polyhedral surfaces $\mathcal{R}_j$, that ``look like'' the $j$-th
approximations $\SC_j$. 
The crucial estimate from this section is Lemma
\ref{lem:HdistSjRj}; it shows that the shells do not degenerate. 
We then decompose the shells into pieces. Up to scaling there are only
finitely many different ones. Each such piece is quasisimilar to the
unit cube $[0,1]^3$ with a common constant $L$.  
 
In Section \ref{sec:DecBall} the pieces are mapped to the unit
ball and reassembled. 
Apart from controlling constants, one has to make sure that maps on
different pieces are \emph{compatible}\index{compatible}, i.e., agree
on intersecting faces.  

The construction of the map $f$ is schematically indicated in Figure
\ref{fig:fconstr}. 
This picture, as well as all others in this and the next section,
corresponds to our standard example $\widehat{\SC}$ (see Subsection
\ref{sec:example}). 

\begin{figure}
  \centering
  \scalebox{0.255}{
    \includegraphics{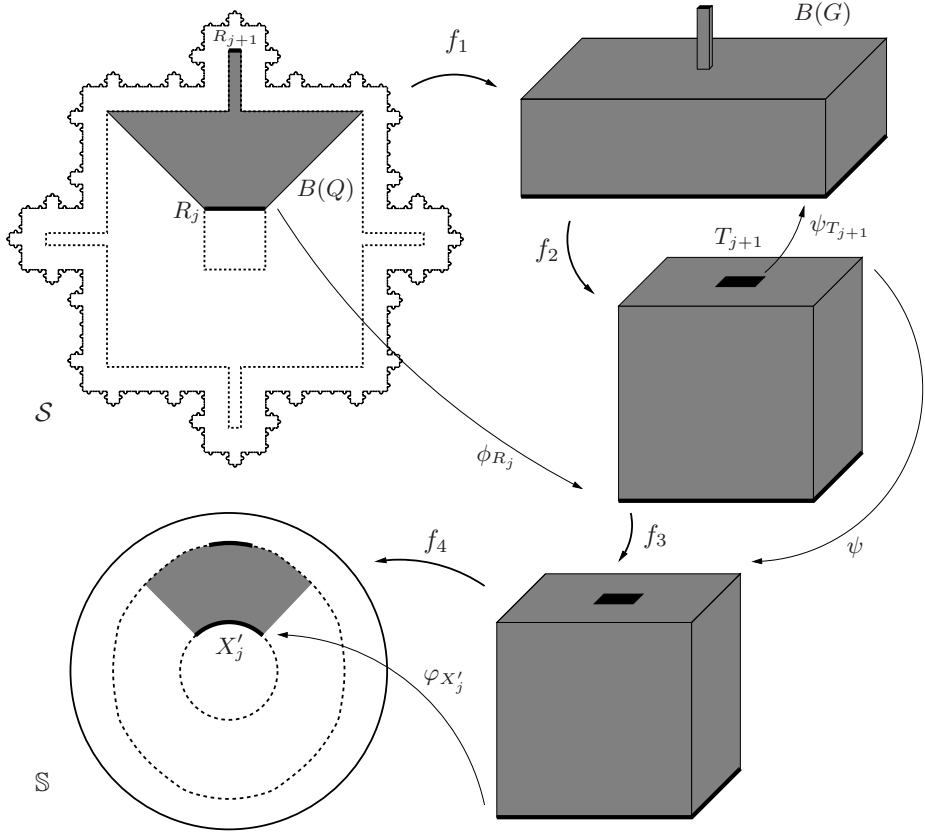}
  }
  \put(-337,155){$\SC$}
  \put(-238,240){\small $B(Q)$} 
  \put(-285,232){\small $R_j$}  
  \put(-270,299){\tiny $R_{j+1}$} 
  \put(-182,296){$f_1$}  
  \put(-50,307){\small $B(G)$}  
  \put(-148,217){$f_2$}  
  \put(-170,140){\small $\phi_{R_j}$} 
  \put(-80,222){\small $T_{j+1}$}
  \put(-30,105){\small $\psi$}  
  \put(-107,109){$f_3$}
  \put(-190,107){$f_4$}
  \put(-269,67){\small $X'_j$}
  \put(-337,15){$\SB$}
  \put(-190,57){\small $\varphi_{X'_j}$}
  \put(-44,227){\small $\psi_{T_{j+1}}$}
  \caption{Construction of the extension $f$.}
  \label{fig:fconstr}
\end{figure}

\subsection{The Surfaces $\RC_j$.}
\label{sec:whitn-decomp-snowb}

It will be convenient to consider distances with respect to the maximum norm in $\R^3$. These will be denoted by an $\infty$-subscript, i.e., we write
\begin{equation*}\label{not:distinfty}
  \dist_{\infty}(A,B):=\inf\{\|a-b\|_{\infty}:a\in A,b\in B\}.
\end{equation*}
In the same way we denote by $\Hdist_{\infty}$ the Hausdorff distance
with respect to the maximum norm.

For a polyhedral surface $\SC_j\subset \R^3$ homeomorphic to the
sphere $\SB$, 
let
\begin{equation}
  \label{eq:defInt}
  \Inter(\SC_j):= \text{ bounded component of } \R^3\setminus \SC_j.
\end{equation}

Recall from Subsection \ref{sec:S_snowsphere} that the height of one face $\mathcal{T}$ of the snowball is at most $\frac{1}{2}-\frac{1}{N_{\max}}$. We approximate the snowsphere from the interior by the surfaces
\begin{equation}\label{not:Rj}\index{R j@$\mathcal{R}_j$}
  \mathcal{R}_j:=\left\{x\in \Inter(\SC_j) : \dist_{\infty}(x,\mathcal{S}_j)= c\delta_j\right\},
\end{equation}
where 
\begin{equation*}\label{not:c}\index{c@$c$}
  c:=\left(\frac{1}{2}-\frac{1}{2N_{\max}}\right).
\end{equation*}

We chose the maximum norm in the definition of $\mathcal{R}_j$ to
again get a polyhedral surface. Had we used the Euclidean distance
instead, $\mathcal{R}_j$ would have some spherical pieces. Note that
$c=\frac{1}{2}-\frac{1}{2N_{\max}}=(N_{\max}-1)\frac{1}{2N_{\max}}$.
Consider one $\delta_j$-square $Q\subset \SC_j$. Then the set
$\{x\in \R^3 : \dist_\infty(x,Q) \geq c\delta_j\}$ lives in the grid $\delta_j\frac{1}{2N_{\max}}\Z^3$. We conclude that
$$
\text{the surface } \mathcal{R}_j \text{ lives in the grid }\delta_j\frac{1}{2N_{\max}}\Z^3. $$
%This grid is indicated in Figure \ref{pic:PdefQ}. 
In particular $\RC_j$ is again a polyhedral surface.

% pic:ext012
\begin{figure}
  \centering
  \scalebox{0.19}{\includegraphics{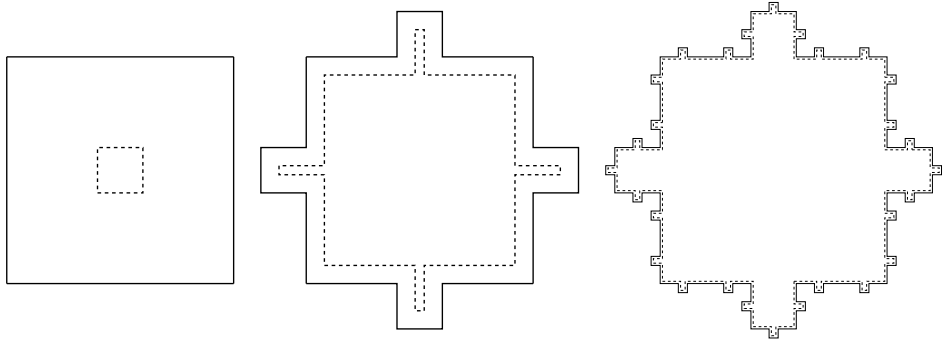}}
  \label{pic:ext012}
  \caption{$\mathcal{R}_0,\mathcal{R}_1,\mathcal{R}_2$ and $\mathcal{S}_0,\mathcal{S}_1,\mathcal{S}_2$.}
\end{figure}

Figure \ref{pic:ext012} shows a $2$-dimensional picture (the
intersection with the plane $y=\frac{1}{2}$) of
$\mathcal{R}_0,\mathcal{R}_1,\mathcal{R}_2$ (dashed line) and
$\mathcal{S}_0,\mathcal{S}_1,\mathcal{S}_2$ (solid line) for the
standard example $\widehat{\SC}$ of Subsection \ref{sec:example}.

\smallskip
We give a more detailed outline of the following subsections:
\begin{itemize}
\item In the next subsection we will see that the surfaces $\RC_j$
  ``look combinatorially'' like 
  $\SC_j$. More precisely, we will define a bijective projection
  $\pi_j\colon \SC_j\to \RC_j$, so the decomposition of $\SC_j$ into
  $\delta_j$-squares is carried to $\RC_j$. This shows that the
  surfaces $\RC_j$ are topological spheres.
\item In Subsection \ref{sec:shells-between-rc_j} we show that $\RC_j$
  and $\RC_{j+1}$ are \defn{roughly parallel}. This enables us to
  decompose the snowball $\BC$ into \defn{shells}, which are bounded
  by these surfaces. 
\item Such a shell is then (Subsection \ref{sec:decomposing-shells})
  decomposed into \defn{pieces}. Up to scaling there are only finitely many
  different such pieces that occur. 
\end{itemize}

We orient the approximations $\SC_j$ by the normal pointing to the
unbounded component of $\R^3\setminus \SC_j$. Thus each
$\delta_j$-square $Q$ from which $\SC_j$ is built obtains an
orientation. The two parts of the double pyramid of $Q$ are called
\defn{outer} and \defn{inner} pyramids of $Q$ accordingly. To
facilitate the discussion we will often map a $\delta_j$-square to the unit
square $[0,1]^2\subset \R^3$ by an (orientation preserving)
similarity, where the inner pyramid is mapped to $\PC^+$, the one with tip
$(\frac{1}{2},\frac{1}{2},\frac{1}{2})$ 
(and the tip of the outer one to $(\frac{1}{2},\frac{1}{2},-\frac{1}{2})$).
It amounts to setting $\delta_j=1$. 
This normalizing map (defined on all of $\R^3$) is denoted by
$\Phi=\Phi_Q$.  It maps other
$\delta_j$-squares to unit squares in $\Z^3$. Let $\Phi(\RC_j):=\RC$.
We will often say that we work in the
\defn{normalized picture}, meaning that the local geometry around $Q$
($\SC_j,\RC_j$, and so on) was mapped by $\Phi$. 

\subsection{The $\RC_j$ are topological Spheres}
\label{sec:rc_j-are-topological}

Here we define a bijective projection
\begin{equation}
  \label{eq:1}
  \pi_j\colon \SC_j\to \RC_j.
\end{equation}
We will define $\pi_j$ as a map later (see the Remark on
page \pageref{rem:Qsquare}). For now we only have need for
the following. We will define $\pi_j$ on the \emph{$1$-skeleton}
of $\SC_j\,$, as well as define $\pi_j(Q)$ as a \emph{set}, for any
$\delta_j$-square $Q\subset \SC_j$.  
The construction will be done locally, meaning we consider one such
$\delta_j$-square 
$Q$ at a time.

Assume first that $\SC_j$ is \defn{flat} at $Q$, meaning all
$\delta_j$-squares $Q'\subset\SC_j$ intersecting $Q$ are parallel. In
the normalized picture let 
\begin{equation}
  \label{eq:piflat}
  \pi(x_1,x_2,0):=(x_1,x_2,c)
\end{equation}
be the projection of $[0,1]^2$ to $\RC$. Then $\pi_j|_Q=\Phi_Q^{-1}\circ \pi
\circ \Phi_Q$.

\smallskip
To define $\pi_j$ in general first 
consider a $\delta_j$-\defn{vertex} $v$ of $\SC_j$ ($v\in \SC_j\cap
\delta_j\Z^3$). At $v$ several $\delta_j$-squares from which $\SC_j$
is built intersect. The projection of $v$ onto $\RC_j$ is indicated in
Figure \ref{fig:localRj}. Here all possibilities (up to
rotations/reflections) of how $\delta_j$-squares (drawn in white) can
intersect in $v$ are shown. The shaded surfaces are the corresponding
surfaces $\RC_j$. The large dot shows the projection of $v$ onto
$\RC_j$. The formal (somewhat cumbersome) definition is as follows.  

\begin{figure}
  \centering
  \scalebox{0.7}{
  \includegraphics{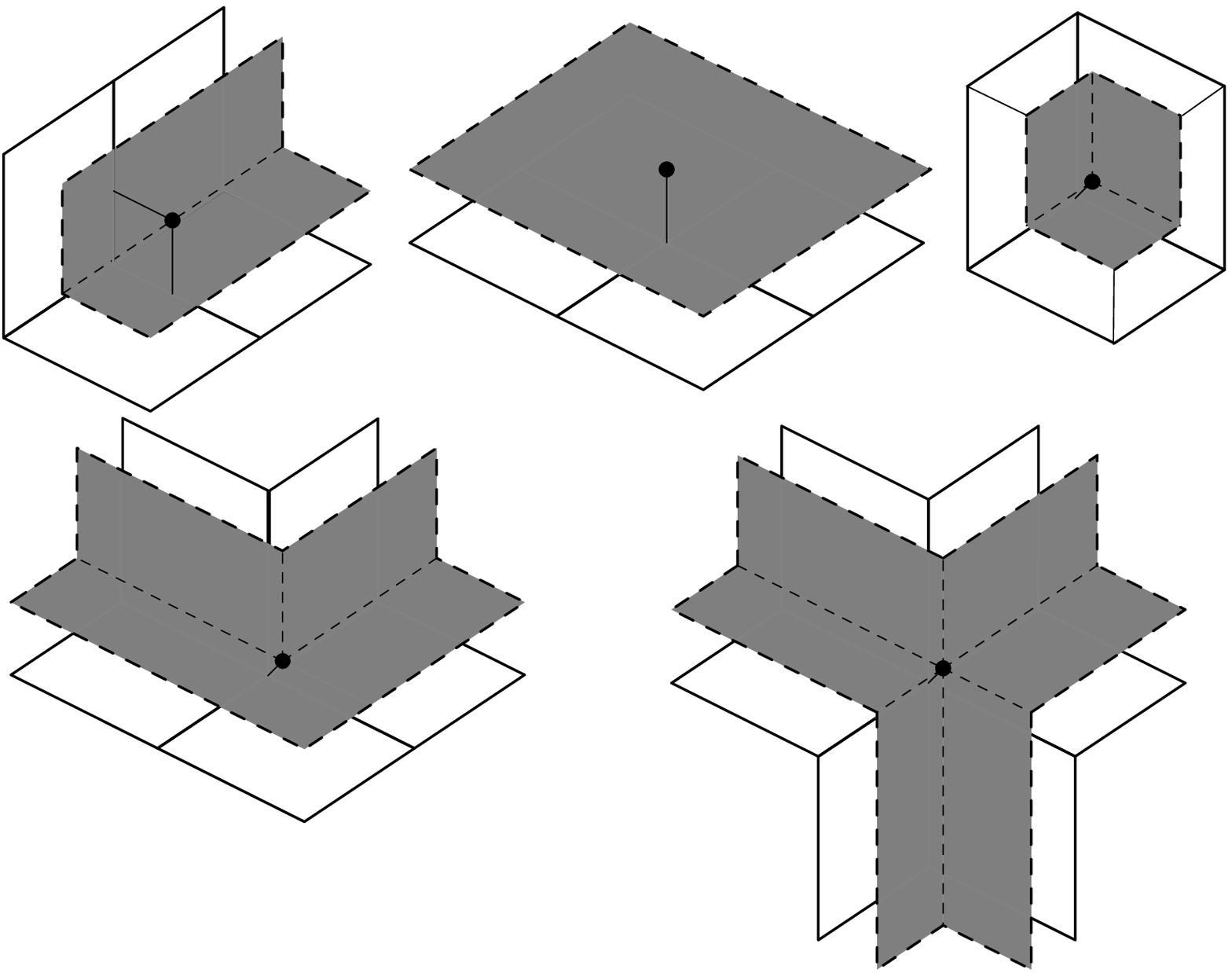}
  }
  \caption{Projections of $v$ onto $\RC_j$.}
  \label{fig:localRj}
\end{figure}

Let $\deg_j(v)$ be the number of $\delta_j$-squares of $\SC_j$ intersecting in $v$. Two
such $\delta_j$-squares are \defn{neighbors} if they share an edge (of
size $\delta_j$). 
We have to consider the case when $\deg_j(v)=5$ separately. So assume
now that $\deg_j(v)=3,4,$ or $6$. Consider the planes through the
intersecting edges bisecting the angle between neighbors. The
intersection of all these planes and $\RC_j$ is exactly one point $p=:
\pi_j(v)$ such that
$\norm{p-v}_\infty= c\delta_j$.

Consider now the case $\deg_j(v)=5$. Note that the planes as above do
not intersect $\RC_j$ in a single point.  
Neighbors are either parallel or perpendicular. Consider only the
planes through edges of perpendicular neighbors, bisecting their
angle. The intersection of all these planes and $\RC_j$ is exactly one
point $p=: \pi_j(v)$ such that $\norm{p-v}_\infty= c\delta_j$.

\smallskip
This defines $\pi_j$ for all vertices $v$ of $\SC_j$. 
Let us record the properties:
\begin{itemize}
\item $\norm{v-\pi_j(v)}_\infty = c \delta_j$.
\item Let $v$ be a vertex of a $\delta_j$-square $Q\subset\SC_j$, and
  let 
  $\pi_j(v)$ be the projection onto $\RC_j$. In the normalized picture
  (where $v$ mapped to the origin) the possible $x$- and
  $y$-coordinates of the projection are $c,0,-c$ (the $z$-coordinate
  is always $c$). 
\end{itemize}

% fig:pvposs
\begin{figure}
  \centering
  \scalebox{0.7}{\includegraphics{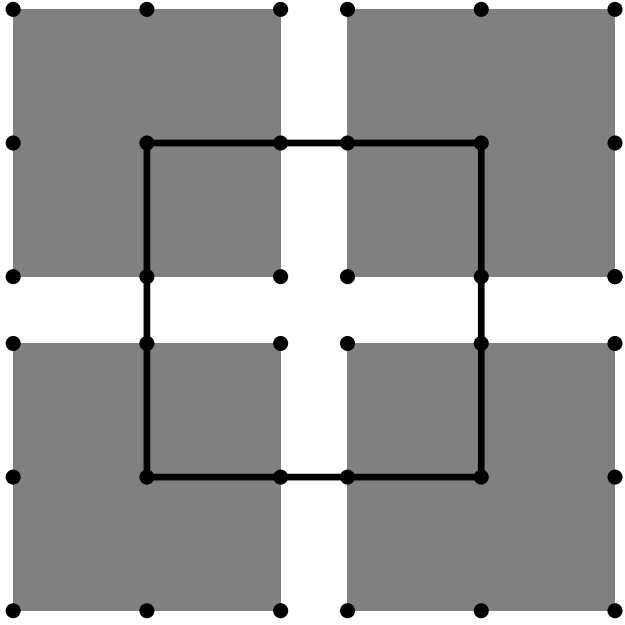}}
  \caption{Possibilities of $\pi_j(v)$.}
  \label{fig:pvposs}
\end{figure}

There are
nine different possibilities for $\pi_j(v)$.  
Figure \ref{fig:pvposs} shows these possibilities for the $4$ vertices of a
square. Note that projections of different points lie in disjoint
squares. The distance of the squares is given by the following. 
Consider two different $\delta_j$-vertices $v,v'\in\SC_j$. Then 
\begin{align*}
  \norm{\pi_j(v)-\pi_j(v')}_\infty
  \geq 
  & \norm{v-v'}_\infty - \norm{v-\pi_j(v)}_\infty - \norm{v'-\pi_j(v')}_\infty
  \\
  & \geq \delta_j -2c\delta_j= \frac{1}{N_{\max}}\delta_j.
\end{align*}

\begin{remark}\label{rem:forbiddenconf}
  If at vertex $v\in \SC_j$ the $\delta_j$-squares intersect as in the
  forbidden configuration (see Figure \ref{fig:annoy}), the surface
  $\RC_j$ has \emph{two} corners corresponding to $v$. Exclusion of
  this case thus simplifies the decomposition considerably.
\end{remark}

Let $E$ be an edge of a $\delta_j$-square $Q\subset\SC_j$ with
vertices $v,v'$. Map $E$ affinely to the line segment with endpoints
$\pi_j(v)$ and $\pi_j(v')$. This defines $\pi_j$ on $E$, thus on the
$1$-skeleton of $\SC_j$.  

\smallskip
Given a $\delta_j$-square $Q\subset\SC_j$ with vertices $v_1,v_2,v_3,v_4$,
the projection $\pi_j(Q)\subset\RC_j$ will be the quadrilateral 
with vertices $\pi_j(v_k)$. It will in general not
be a rectangle, in fact not even convex. Note also that we did not yet
specify how 
\emph{individual points} of $Q$ get mapped by $\pi_j$.

\begin{lemma}
  \label{lem:RjcombSj}
  The projections $\pi_j$ satisfy the following:
  \begin{enumerate}
  \item \label{item:distQpQ}
    For every $\delta_j$-square $Q\subset\SC_j$, we have
    \begin{equation*}
      \dist_\infty(Q,\pi_j(Q))=\Hdist_\infty(Q,\pi_j(Q))=c\delta_j.
    \end{equation*}
  \item \label{item:RjSjcombeq}
    Consider the sets 
    \begin{equation*}
      R_j:=\pi_j(Q),
    \end{equation*}
    where $Q$ is a $\delta_j$-square in the approximation
    $\SC_j$. These sets form a decomposition of the surface $\RC_j$
    into quadrilaterals, $\RC_j=\bigcup R_j$. View $\RC_j$ as a
    \defn{cell complex}, where images of
    $\delta_j$-squares/edges/vertices by $\pi_j$ are the $2$-,$1$-,
    and $0$-cells. Then $\RC_j$ and $\SC_j$ are \emph{isomorphic as
      cell complexes}.
  \item \label{item:Rjspheres}
    The set $\RC_j$ is a polyhedral surface homeomorphic to the unit
    sphere $\SB$.
  \item \label{item:RjbddistSj} 
    $\Inter(\RC_j)= \{x\in\Inter(\SC_j) : \dist_\infty(x,\SC_j) > c\delta_j\}$. 
  \end{enumerate}
\end{lemma}

\begin{proof}
  To see (\ref{item:distQpQ}) work in the normalized picture. Let
  $\pi\colon [0,1]^2\to \{z=c\}$ be the map conjugate to $\pi_j$ under
  the normalizing map $\Phi$.  Then
  \begin{equation*}
    Q_*:=[c,1-c]^2\times \{c\} \subset \pi([0,1]^2) \subset
    [-c,1+c]^2\times \{c\}=: Q^*;
  \end{equation*}
  see Figure \ref{fig:pvposs}, and Figure \ref{fig:projpyr}. 
  Note that
  \begin{align*}
    c & = \dist_\infty([0,1]^2,Q_*) = \Hdist_\infty([0,1]^2,Q_*)
    \\
    & =\dist_\infty([0,1]^2,Q^*)=\Hdist_\infty([0,1]^2,Q^*).
  \end{align*}
  The statement follows.

  % fig:projpyr
  \begin{figure}
    \centering 
    \scalebox{0.8}{\includegraphics{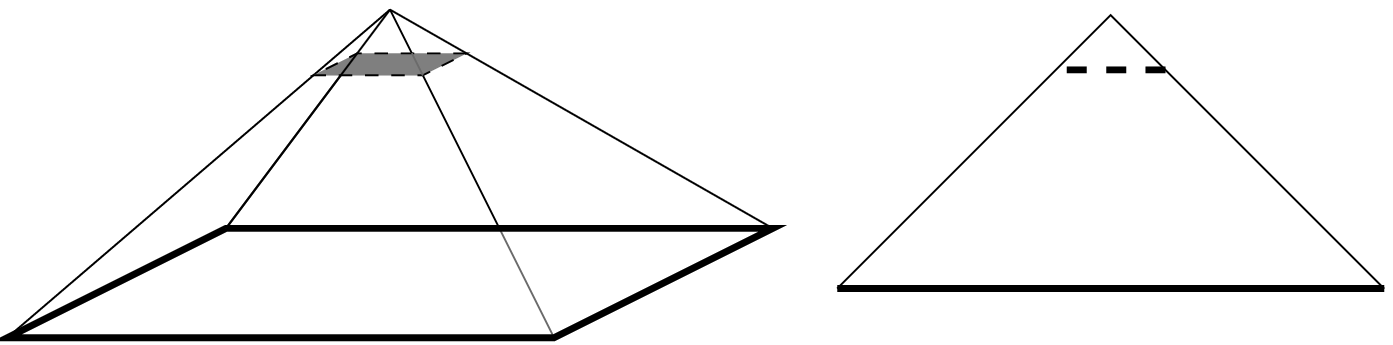}}
    % 
    % labels
    \begin{picture}(,)
      \put(-268,63){$Q_*$}
      \put(-98,63){$Q_*$}
      \put(-225,10){$Q$}
      \put(-53,17){$Q$}
    \end{picture}
    \caption{Part of $\RC_j$.}
    \label{fig:projpyr}
  \end{figure}
  
  \smallskip
  Property (\ref{item:RjSjcombeq}) is clear from the construction. 

  \smallskip
  Any homeomorphism $\pi_j\colon Q \to \pi_j(Q)$, that extends
  $\pi_j|_{\partial Q}$ yields Property (\ref{item:Rjspheres}). 

  \smallskip
  (\ref{item:RjbddistSj}) The set $\Inter(\SC_j) \setminus \RC_j$ has
  two components by the PL-Sch\"onflies theorem. The sets $\{x\in
  \Inter(\SC_j) : \dist_\infty(x,\SC_j) < c\delta_j\}, \{x\in
  \Inter(\SC_j) : \dist_\infty(x,\SC_j) > c\delta_j\}$ are both
  non-empty 
  (see Figure \ref{fig:projpyr}); pick points in the pyramid $\PC^+$
  above and below $Q_*$. Thus these sets are the two components. The
  statement follows from using PL-Sch\"onflies again.
\end{proof}

Applying the same reasoning to the unbounded component of
$\R^3\setminus \SC_j$ yields the following.  
\begin{cor}\label{cor:twocomp}
  The set $\{\dist_\infty(x,\SC_j)>c\delta_j\}$ has two components,
  one bounded (by $\RC_j$) and one
  unbounded. 
\end{cor}

\subsection{The shells between $\RC_j$ and $\RC_{j+1}$}
\label{sec:shells-between-rc_j}

We will show that the surfaces $\RC_j$ and $\RC_{j+1}$
 are \defn{roughly parallel}. 
This will enable us to decompose the snowball into \defn{shells}
bounded by two such surfaces.

Lower bounds
on the distance will be controlled by $\dist_{\infty}$, while upper
bounds of their distance will be controlled by the Hausdorff distance
$\Hdist_{\infty}$. Note that $\dist_{\infty}$ is not suited to control
upper bounds and that $\Hdist_{\infty}$ is not suited to control lower
bounds on the distance. 

\medskip 
Two sets $A$ and $B$ are called \emph{roughly $\delta$-parallel}
\index{roughly delta-parallel@roughly $\delta$-parallel} ($\delta>0$) with constant $C>0$ if
\begin{equation}
  \dist_{\infty}(A,B)\geq \frac{1}{C}\delta\; \mbox{ and } 
  \Hdist_{\infty}(A,B)\leq C\delta.
\end{equation}

\begin{lemma}
  \label{lem:HdistSjRj}
  The surfaces $\SC,\SC_j,$ and $\RC_j$ satisfy
  \begin{enumerate}
  \item\label{item:Hdist1}
    $\Hdist_\infty(\RC_j,\SC_j)=c\delta_j$.
    \newline\noindent
    So $\RC_j$ and $\SC_j$ are (roughly) $c\delta_j$-parallel with constant $C=1$.
  \item\label{item:HdistRjS}   
    $\mathcal{R}_j$ and $\mathcal{S}$ are roughly $\delta_j$-parallel
    with constant $C=C(N_{\max})$ (independent of $j$).
  \item \label{item:IntSjSk}
    $\Inter(\RC_j)$ is compactly contained in $\Inter(\RC_{j+1})$,
    i.e.,
    \begin{equation*}
      \Inter(\RC_0)\Subset \Inter(\RC_1)\Subset
      \Inter(\RC_2)\Subset  \dots \;.
    \end{equation*}
  \item\label{item:Rjpara}   
    $\mathcal{R}_j$ and $\mathcal{R}_{j+1}$ are roughly
    $\delta_j$-parallel with constant $C=C(N_{\max})$.
  \item \label{item:SjkRj}
    There is a positive integer $k_0$ such that
    \begin{equation*}
      \{\dist_\infty(x,\SC) > \delta_{j-k_0}\} \subset
      \{\dist_\infty(x,\RC_j)> c\delta_j\} \subset
      \{\dist_\infty(x,\SC) > \delta_{j+k_0}\}, 
    \end{equation*}
    for all $j\geq k_0$.
  \end{enumerate}
\end{lemma}

\begin{proof}
  % Property 1
  (\ref{item:Hdist1})
  Obviously 
  \begin{equation*}
    d_{\SC_j}(\RC_j)=c\delta_j;
  \end{equation*}
  this distance (see (\ref{eq:defdA})) is again taken with respect to $\norm{\cdot}_\infty$.

  \smallskip 
  It remains to show that $d_{\RC_j}(\SC_j)\leq c \delta_j$. 
  Work again in the normalized picture. As
  before  $Q_*=\{z=c\}\cap \PC^+= [c,1-c]^2\times
  \{c\}\subset\RC$; see Figure \ref{fig:projpyr}. Since $d_{Q_*} ([0,1]^2)= c$ it follows that
  $d_{\RC_j}(\SC_j)\leq c\delta_j$.

  % Property 2
  \medskip
  (\ref{item:HdistRjS}) For every $x\in \mathcal{R}_j$ we have by
  (\ref{eq:triag_dist})
  \begin{align*}
    \dist_{\infty}(x,\mathcal{S})
    &\geq \dist_{\infty}(x,\mathcal{S}_j)
    -\Hdist_{\infty}(\mathcal{S}_j,\mathcal{S})
    \\
    &\geq  \left(\frac{1}{2}-\frac{1}{2N_{\max}}
    \right)\delta_j-\left(\frac{1}{2}-\frac{1}{N_{\max}}\right)\delta_j 
    \mbox{, \quad by \eqref{eq:dSjS}}
    \\
    &=\frac{1}{2N_{\max}}\delta_j.
  \end{align*}
  So $\dist_{\infty}(\mathcal{R}_j,\mathcal{S})\geq
  \frac{1}{2N_{\max}}\delta_j$.
  Here we see that $c>\left(\frac{1}{2}-\frac{1}{N_{\max}} \right)$
  ensures that $\mathcal{R}_j$ does not intersect the snowsphere
  $\mathcal{S}$.   

  On the other hand,
  \begin{align}
    \Hdist_{\infty}(\mathcal{R}_j,\mathcal{S})&\leq \Hdist_{\infty}(\mathcal{R}_j,\mathcal{S}_j)+\Hdist_{\infty}(\mathcal{S}_j,\mathcal{S})\notag
    \\
    & \leq c\delta_j+\left(\frac{1}{2}-\frac{1}{N_{\max}}\right)\delta_j\leq
    \left(1-\frac{1}{N_{\max}}\right) \delta_j \label{eq:dRjSvar}
    \\
    & \leq\delta_j,\label{eq:dRjS}
  \end{align}  
  by property (\ref{item:Hdist1}) and (\ref{eq:dSjS}).

  % Property 3
  \medskip
  (\ref{item:IntSjSk})
  Consider an $x\in \R^3$ such that $\dist_\infty(x,\SC_j)\geq
  c\delta_j$. Then  
  \begin{align*}
    \dist_\infty(x,\SC_{j+1}) & - c\delta_{j+1}
    \\
    \geq & \dist_\infty(x,\SC_j) - \Hdist_\infty(\SC_j,\SC_{j+1}) -
    c\delta_{j+1}
    \\
    \geq & c\delta_j - \left( \frac{1}{2}
    -\frac{3}{2}\frac{1}{N_{j+1}} \right) \delta_j - c\delta_{j+1},
  \quad \text{ by (\ref{eq:HdistSj})} 
    \\
    = & \left[ \frac{1}{2} - \frac{1}{2N_{\max} } - \frac{1}{2} +
      \frac{3}{2}\frac{1}{N_{j+1}} - 
      \left(\frac{1}{2} - \frac{1}{2N_{\max}}\right) \frac{1}{N_{j+1}}
    \right] \delta_j
    \\
    \geq & \frac{1}{2 N_{\max}} \delta_j.
  \end{align*}
  Thus $\dist_\infty(x,\SC_{j+1})> c\delta_{j+1}$, and hence
  \begin{equation*}
    \{\dist_\infty(x,\SC_0) > c \delta_0 \} \Subset 
    \{\dist_\infty(x,\SC_1)> c \delta_1\} \Subset \dots \; . 
  \end{equation*}
  The statement follows from Corollary \ref{cor:twocomp}
  and Lemma \ref{lem:RjcombSj} (\ref{item:RjbddistSj}).

  % Property 4
  \medskip
  (\ref{item:Rjpara})
  One inequality follows immediately from inequality \eqref{eq:dRjS}:
  \begin{align*}
    \Hdist_{\infty}(\mathcal{R}_j,\mathcal{R}_{j+1}) & \leq\Hdist_{\infty}(\mathcal{R}_j,\mathcal{S})+\Hdist_{\infty}(\mathcal{S},\mathcal{R}_{j+1})   
    \\
    & \leq \delta_j+\delta_{j+1}\leq2\delta_j.
  \end{align*}
  To see the second inequality recall inequality (\ref{eq:HdistSj}).   
  Together with property (\ref{item:Hdist1}) this yields
  \begin{align*}
    \begin{split}
      \dist_{\infty}(\mathcal{R}_j,\mathcal{R}_{j+1})
      & \geq 
      \dist_{\infty}(\mathcal{R}_j, \mathcal{S}_j) 
      -\Hdist_{\infty}(\mathcal{S}_j,\mathcal{S}_{j+1})   
      %\\
      %&\phantom{XXXXXXXX}
      -\Hdist_{\infty}(\mathcal{S}_{j+1},\mathcal{R}_{j+1})
    \end{split}
    \\
    \begin{split}
      &\geq  \left(\frac{1}{2}-\frac{1}{2N_{\max}}
      \right)\delta_j-\left(\frac{1}{2}-\frac{3}{2N_{j+1}}\right)\delta_j
        \\
        &\phantom{XXXXXXXX} 
        - \left(\frac{1}{2} -\frac{1}{2N_{\max}}
        \right)\delta_j\frac{1}{N_{j+1}} 
    \end{split}
    \\
    & \geq  \left(\frac{1}{N_{j+1}}-\frac{1}{2N_{\max}}\right)\delta_j\geq\frac{1}{2N_{\max}}\delta_j.
  \end{align*}

  % Property 5
  (\ref{item:SjkRj})
  Pick an $x\in \R^3$ such that $\dist_\infty(x,\RC_j)>
  \delta_j$. Then
  \begin{align*}
    \dist_\infty(x,\SC) & \geq \dist_\infty(x,\RC_j) -
    \Hdist_\infty(\RC_j,\SC)
    \\
    & > \delta_j - \left(1-\frac{1}{N_{\max}}\right) \delta_j \text{
      by (\ref{eq:dRjSvar})}
    \\
    & = \frac{1}{N_{\max}} \delta_j.
  \end{align*}
  Now pick $y\in \R^3$ with $\dist_\infty(y,\SC)> \delta_j$. Then
  \begin{align*}
    \dist_\infty(y,\RC_j) & \geq
    \dist_\infty(y,\SC)-\Hdist_\infty(\SC,\RC_j)
    \\
    & > \delta_j - \left(1-\frac{1}{N_{\max}}\right) \delta_j \text{
      by (\ref{eq:dRjSvar})}
    \\
    & = \frac{1}{N_{\max}} \delta_j.
  \end{align*}
  Choose $j_0$ such that $2^{j_0}\geq N_{\max}$. Thus
    \begin{equation}\label{eq:lemRjprop5_1}
      \{\dist_\infty(x,\SC) > \delta_{j-j_0}\} \subset
      \{\dist_\infty(x,\RC_j)> \delta_j\} \subset
      \{\dist_\infty(x,\SC) > \delta_{j+j_0}\}, 
    \end{equation}
    for all $j\geq j_0$. Note that $N_{\max}\geq 2$ implies
    \begin{equation*}
      \frac{1}{2} > c = \frac{1}{2} - \frac{1}{2N_{\max}} \geq \frac{1}{4}.
    \end{equation*}
    Thus $\delta_{j+2} \leq c \delta_j < \delta_j$ and
    \begin{equation}
      \label{eq:lemRjprop5_2}
      \{\dist_{\infty}(x,\RC_j) > \delta_j\}
      \subset \{\dist_{\infty}(x,\RC_j) > c\delta_j\} 
      \subset \{\dist_{\infty}(x,\RC_j) > \delta_{j+2}\}.
    \end{equation}
    The statement follows by combining (\ref{eq:lemRjprop5_1}) and
    (\ref{eq:lemRjprop5_2}) with $k_0=j_0+2$.

\end{proof}

By Property (\ref{item:IntSjSk}) of the last lemma we can define for
$j\geq 0$ the \defn{shells} 
\begin{align*}
  \mathcal{B}_j:= \clos\Inter(\RC_{j+1})\setminus \Inter(\RC_j),
\end{align*}
bounded by $\mathcal{R}_j$ and $\mathcal{R}_{j+1}$. Property
(\ref{item:Rjpara}) of the previous lemma controls the ``thickness''
of these shells. By Property (\ref{item:SjkRj}) and Corollary
\ref{cor:twocomp} we obtain the following.
\begin{cor}\label{cor:BintS}
  The bounded component of  $\R^3\setminus \SC$ is
  \begin{equation*}
    \bigcup_j \Inter(\RC_j) = \bigcup_j \mathcal{B}_j\cup\Inter(\RC_0)= \inte\BC.
  \end{equation*}
\end{cor}
It is simply
connected, since each set $\Inter(\RC_j)$ is (using
Lemma \ref{lem:HdistSjRj} (\ref{item:IntSjSk})). Furthermore $\partial \BC=\SC$.

\subsection{Decomposing the Shells}
\label{sec:decomposing-shells}

We decompose the shells $\BC_j$ into pieces. 
This is the trickiest part of this section.

Fix a $\delta_j$-square $Q\subset \SC_j$. 
We want to define a set $B(Q)\subset\BC_j$ ``above'' $Q$.
Work
in the normalized picture. Let $\RC,\RC'$ be the images of $\RC_j,\RC_{j+1}$
under the normalization. The piece of $\SC_{j+1}$ bounded by $\partial
Q$ maps (under the normalization) to $G$, which is the (correctly
oriented) $N_{j+1}$-generator. It is built from squares of side-length
$\delta:=1/N_{j+1}$. Call $\pi\colon [0,1]^2\to \RC$ the map 
which is conjugate to $\pi_j\colon \SC_j\to \RC_j$ (under the
normalization), and $\pi'\colon G\to \RC'$ the one that is conjugate
to $\pi_{j+1}\colon \SC_{j+1}\to \RC_{j+1}$. Note that we will only
use $\pi,\pi'$ as \emph{maps} on $\partial [0,1]^2$ and
$\pi([0,1]^2),\pi'([0,1]^2)$ as \emph{sets}.

Assume first that all $\delta_j$-squares $Q'\subset \SC_j$
intersecting $Q$ are parallel to $Q$. Then
$\pi'(G)$ is a polyhedral surface bounded by $\partial [0,1]^2\times
\{\delta c\}$. Also $\pi([0,1]^2)=[0,1]^2 \times \{c\}$. Note that by
Lemma \ref{lem:HdistSjRj} (\ref{item:Rjpara})
$\pi([0,1]^2)\cap\pi'(G)=\emptyset$. Consider a 
$\delta$-vertex $v$ in the \defn{interior} of $G$, i.e., $v\in \delta\Z^3\cap G\setminus \partial [0,1]^2$. Then $\dist_\infty(v,\partial \PC)\geq
\frac{1}{2}\delta$, here $\PC$ denotes the double pyramid (see Section
\ref{sec:gen}, and Figure \ref{pic:gen_e_p}). Thus
\begin{align} \label{eq:distQipP}
  \dist_\infty(\pi'(v), \partial \PC) & \geq \dist_\infty(v,\partial
  \PC) - \norm{v-\pi'(v)}_\infty &&
  \\ \notag
  & \geq \frac{1}{2}\delta - c \delta = \frac{1}{2N_{\max}} \delta,
  && \text{ by Subsection \ref{sec:rc_j-are-topological}.}
\end{align}
Thus $\pi'(G)\cup \left([0,1]^2\times \{c\}\right) \cup \left(\partial
  [0,1]^2\times [c\delta, c]\right)$ is a polyhedral surface
homeomorphic to the sphere $\SB$.  

Using the PL-Sch\"onflies theorem in $\R^3$ once more, we define the
\defn{standard piece} corresponding to the generator $G$ (with given
orientation) as the set 

\begin{align}\label{eq:defBG}
  B_G & =B_G([0,1]^2) 
  \\ \notag
  & := \clos\Inter\pi'(G)\cup \left([0,1]^2\times \{c\}\right) \cup
  \left(\partial [0,1]^2\times [c\delta, c]\right). 
\end{align}
See Figure \ref{fig:B_G} for a two-dimensional picture. 
The piece $B(Q)$ will be the image of
$B_G([0,1]^2)$ under (the inverse of) the normalizing map, where $G=G_j$
is the (correctly oriented) generator by which $Q$ was replaced to
construct $\SC_{j+1}$. 

\begin{figure}
  % \centering
  \raggedright
  \scalebox{0.46}{\includegraphics{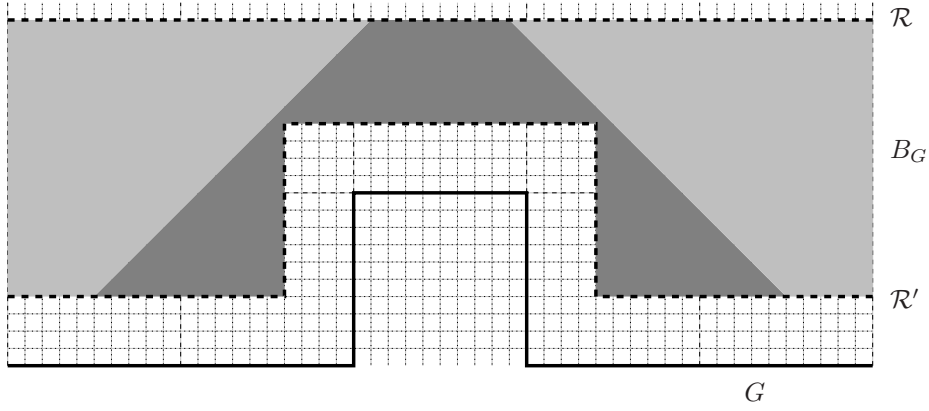}}
  \put(-50,-12){$G$}
  \put(5,23){$\RC'$}
  \put(5,130){$\RC$}
  \put(5,80){$B_G$}
  \caption{The standard piece $B_G$.}
  \label{fig:B_G}
\end{figure}

\medskip
Let the $\delta_j$-square $Q\subset\SC_j$ be arbitrary. To define the
piece $B(Q)\subset \BC_j$ we again work first in the normalized picture.

\begin{definition}
  The set $B$ is the one 
  bounded by $\pi([0,1]^2),\pi'(G)$ and the line segments with
  endpoints $\pi(v),\pi'(v)$ for all $v\in \partial [0,1]^2$. 
\end{definition}
Call $\pi([0,1]^2)$ the \defn{inner} side and $\pi'(G)$ the \defn{outer}
side of $B$; the outer side is closer to $\SC$ than the inner side. 
We will
show that $B$ is bi-Lipschitz to the standard piece $B_G$
(\ref{eq:defBG}). 

The following discussion can be paraphrased in the following way:
The piece $B$ has a ``core'' which is identical
to the one of $B_G$. The ``rest'' of $B$ has ``trivial geometry'' (not
depending on the generator $G$), which can be used to deform $B$ into
$B_G$.   

Consider a $\delta$-square $Q'\subset G$. It will be called an
\defn{interior square} if $Q'\cap\partial [0,1]^2=\emptyset$ and a
\defn{boundary square} otherwise. 
From (\ref{eq:distQipP}) we obtain $\dist_\infty(Q',\partial \PC) \geq
\frac{1}{2N_{\max}}$ for such an interior $\delta$-square $Q'\subset
G$.
Note that each boundary $\delta$-square $Q'\subset G$ lies in the
$xy$-plane.   
Define
\begin{equation*}
  \core(B):=\left\{x\in B: \dist_\infty(x,\partial \PC) \geq
  \frac {1} {4 N_{\max} }
  \delta \right\}.
\end{equation*}

See Figure \ref{fig:B_G}; here $\core(B_G)$ is the darker shaded region. 
We map $\core(B)$ to $\core(B_G)$ by the identity. 
The ``remaining set'' $B\setminus\operatorname{core}(B)$ can be broken
up into pieces and mapped to the corresponding piece in $B_G$ using
Lemma \ref{lem:BbiLip}. 

\begin{figure}
  \centering
  \scalebox{0.4}{\includegraphics{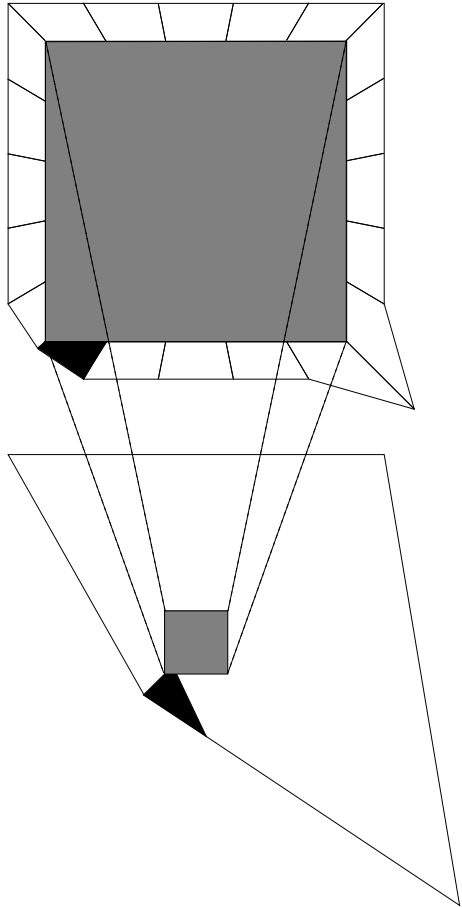}}
  \caption{Decomposing $B\setminus \core(B)$.}
  \label{fig:decompB}
\end{figure}

For the reader who is a stickler we give a precise construction.
It is illustrated in Figure \ref{fig:decompB}. The outer side is shown
on top, the inner side on the bottom. Thus the picture is ``turned
around'' compared to Figure \ref{fig:B_G}. The set $\core(B)$ is indicated
as the shaded region. Note that this is not a situation occurring for
our standard example $\widehat{\SC}$. The picture is not to scale as well.

First consider the \defn{outer side} of the remaining piece, i.e., the set
$\RC'\cap (B\setminus\core(B))$. 
The set
 $\partial \core(B)\cap \RC'$ is a square, each side of which we
decompose into $N_{j+1}$ line segments (of the same size). The other
boundary component is $\pi'(\partial G)=\pi'(\partial [0,1]^2)$. The images of the
$\delta$-edges decompose it into $4N_{j+1}$ line segments. Connect
corresponding line segments (by line segments) to obtain the
decomposition of the outer side of $B\setminus \core(B)$ into quadrilaterals.

Now consider the \defn{inner side} of the remaining piece, i.e., the
set $\RC\cap (B\setminus\core(B))$. It is bounded by a square ($\partial \core 
(B)\cap \RC$) and the quadrilateral $\pi(\partial [0,1]^2)$. Each side
of the two quadrilaterals gets decomposed into $N_{j+1}$ pieces of the same
length. Connecting corresponding edges in the two boundary components
decomposes $\RC\cap(B\setminus\core(B))$ into quadrilaterals. This is
shown only for one quadrilateral in Figure \ref{fig:decompB}. 

The set $B\setminus\core(B)$ gets decomposed into pieces between
corresponding quadrilaterals in the outer and inner face as in
equation (\ref{eq:defBQQ}). Use the map from (\ref{eq:defBstu}) to map
corresponding pieces of $B\setminus \core(B)$ to
$B_G\setminus\core(B_G)$. Note that this piecewise defined map agrees on
intersections. A tedious, but elementary computation shows that the
maps do not degenerate, i.e., that (\ref{eq:Quconvex}) is
satisfied. 

As an example, we do the computation for the piece bounded
by the black quadrilaterals indicated in Figure \ref{fig:decompB}.
The $xy$-coordinates of the vertices of the outer (black)
quadrilateral (shown on top) are
\begin{align*}
  P_0^1 & =\delta c\left<1,1\right>, 
  &&
  P_1^1 = \delta\left<1,0\right>,
  \\
  P_2^1 & = \frac{1}{2}\delta \left<1,1\right> +
  \delta(1-\delta)\left<1,0\right>, 
  &&
  P_3^1 = \frac{1}{2}\delta\left<1,1\right>.
\end{align*}
The ones for the inner (black) quadrilateral (shown at the bottom) are
\begin{align*}
  P_0^0 & =c\left<1,1\right>,
  \phantom{XXXXXXXX}
  P_1^0 = c\left<1,1\right> +\delta\left<1,-2c\right>,
  \\
  P_2^0 & = \left(c+\frac{1}{2 N_{\max}}\delta\right)\left<1,1\right> 
  + \delta\frac{1}{N_{\max}}(1-\delta)\left<1,0\right>,
  \\
  P_3^0 & = \left(c+\frac{1}{2 N_{\max}}\delta\right)\left<1,1\right>.
\end{align*}
Define $P_k^u:=(1-u)P_k^0 + uP_k^1, \; u\in [0,1]$, as in Section
\ref{sec:simplices-extensions}. For $J$ as in
(\ref{eq:Quconvex}) one computes 
$$J\geq \frac{\delta^2}{4 N_{\max}^2}.$$  
One checks the non-degeneracy (positivity of $J$) of other pieces and
types of vertices by the same type of computation. In this fashion
$B\setminus \core(B)$ is decomposed into sets bi-Lipschitz equivalent
to the cube $[0,1]^3$. Map those to corresponding pieces in the
standard piece. Note that the maps agree on intersecting faces by the
construction of the maps from (\ref{eq:defBstu}). 

We have proved the following.
\begin{lemma}
  \label{lem:BGbiLip}
  There is a bi-Lipschitz map
  \begin{equation*}
    f_1=f_{1,B}\colon B\to B_G.
  \end{equation*}
\end{lemma}
There are only finitely many different sets $B$ (and $B_G$). So we can
assume that the maps $f_{1,B}$ have a common bi-Lipschitz constant $L$.

For a $\delta_j$-square $Q_j\subset\SC_j$, now define the set
$B(Q_j)\subset\BC_j$ as the inverse of the set $B$ (defined above) under
the normalization. 

Note that $[c,1-c]^3$ is bounded by $\RC_0$.
\begin{lemma}
  \label{lem:BQWhitney}
  The sets $B(Q_j)$ together with the set $[c,1-c]^3$ form a
  Whitney-type decomposition of the snowball; this means 
  \begin{enumerate}
  \item $$\bigcup_{\substack{j\geq 0 \\ Q_j \subset \SC_j}} B(Q_j)
    \cup [c,1-c]^3=\inte\BC.$$ 
  \item The interiors of the sets $B(Q_j)$ are pairwise disjoint.
  \item\label{item:Whitney3} 
    $$\diam B(Q_j) \asymp \dist(B(Q_j),\SC)\asymp \delta_j,$$
    where $C(\asymp)=C(N_{\max})$.
  \end{enumerate}
\end{lemma}
\begin{proof}
  The first statement follows from Corollary \ref{cor:BintS}. The
  second is clear from the construction. The third follows from
  Lemma \ref{lem:HdistSjRj} (\ref{item:HdistRjS}) and (\ref{item:Rjpara}).
\end{proof}

The composition of the normalizing map and the one from Lemma
\ref{lem:BGbiLip} is still called 
\begin{equation}
  \label{eq:deff1Q}
  f_1=f_{1,Q}\colon B(Q)\to B_G.   
\end{equation}

This map is quasisimilar (see (\ref{eq:defquasisym})), where the
scaling factor is $l=1/\delta_j$ and the constant $L$ is uniform. In Figure
\ref{fig:fconstr} this map, as well as the following ones, is illustrated. 

\begin{remark}
  \label{rem:Qsquare}
  The map $f_1\colon B(Q)\to B_G$ can be used to define
  \begin{equation}
    \label{eq:defpij}
    \pi_j \colon \SC_j \to \RC_j. 
  \end{equation}
  Namely, map $Q$ isometrically to 
  $[0,1]^2\times \{c\}$, which in turn is mapped to $\pi_j(Q)\subset
  \RC_j$ by $f_1^{-1}$. Formally $\pi_j|_Q:=f_1^{-1} \circ \pi \circ
  \Phi_Q$ ($\Phi_Q$ is the normalizing map, $\pi$ from equation
  (\ref{eq:piflat})). The map $\Phi_{Q}$ has to be the same as the one
  used in the definition of $f_1$, so vertices are mapped correctly.
  Note that this definition agrees with the previous 
 definition of $\pi_j$ on the $1$-skeleton of $\SC_j$ (edges are mapped
   affinely). The maps $\pi_j$ are bi-Lipschitz with a common
   bi-Lipschitz constant $L$.

\end{remark}

Consider two distinct $\delta_j$-squares $Q,Q^*\subset\SC_j$. We
think of $B_G(Q)=f_{1,Q}(B(Q))$ and $B_{G^*}(Q^*)=f_{1,Q^*}(B(Q^*))$ as
being distinct, since they are to be mapped to different sets.
Note that $G,G^*$ are the same generators, but may have different
orientation. There are only finitely many different sets $B_G(Q)$
throughout the construction, up to isometries.
  
\begin{lemma}
  \label{lem:f1comp}
  The map $f_1$ is \emph{compatible}\index{compatible} on
  neighbors (i.e., $Q,Q^*$ intersecting in a $\delta_j$-edge). This 
  means the following.  
  Identify appropriate sides of $B_G(Q)$ and $B_{G^*}(Q^*)$ (one of the
  four sides $\partial[0,1]^2\times [\frac{c}{N_{j+1}},c]$).
  Then
  $f_1=f_1^*$ on $B(Q)\cap B(Q^*)$. 
\end{lemma}

\begin{proof}
  Work again in the normalized picture. Consider a $v\in \partial
  [0,1]^2$. The boundary of $B$ contains the line segment with
  endpoints $\pi(v),\pi'(v)$. The map $f_1$ maps this line segment
  affinely to
  $\{v\}\times [\frac{c}{N_{j+1}}, c]$. 
  The same is true for
  the map $f^*_1$ on the neighboring piece $B^*$.
\end{proof}

\smallskip
Consider (for a given generator) our standard piece $B_G$. 
Recall from Subsection \ref{sec:whitn-decomp-snowb} that $\RC_j$
lives in the 
grid $\delta_j\frac{1}{2N_{\max} }\Z^3$. Thus $B_G$ lives in the grid
$\frac{1}{2N_{j+1}N_{\max}}\Z^3$. This is indicated (for our standard
example) in Figure \ref{fig:B_G}. The boundary of $B_G$ consists of
$[0,1]^2\times \{c\}$, $\pi'(G)$, and four sides perpendicular to the
$xy$-plane \mbox{($\partial[0,1]^2\times [\frac{c}{N_{j+1}},c]$)}.

Using Corollary \ref{cor:PLSchoenbiLip}
we can map $B_G$ orientation preserving to the unit cube by a
bi-Lipschitz map 
\begin{equation}
  \label{eq:deff2}
  f_2=f_{2,B_G}\colon B_G\to [0,1]^3.
\end{equation}
We further require
that $f_2$ maps 
\begin{itemize}
\item 
  $[0,1]^2\times \{c\}$ (the inner side) isometrically to
  $[0,1]^2\times \{0\}$;
\item 
  $\pi'(G)$ (the outer side) to $[0,1]^2\times \{1\}$;
\item the sides $\partial [0,1]^2\times [\frac{c}{N_{j+1}},c]$
  affinely to $\partial [0,1]^2\times [0,1]$. 
\end{itemize}
To see that we can make these further assumptions, either go
through the proof of the PL-Sch\"onflies theorem or 
post-compose with a map from Lemma \ref{lem:extSB}.

As before we think of images of $f_2$ as distinct, i.e.,
$f_2(B_G(Q))=[0,1]^3=[0,1]^3(Q)$. Since there are only finitely many
different sets $B_G$ (up to isometries), we can assume that all maps
$f_2$ have a common bi-Lipschitz constant $L$.

It will be convenient to restrict our attention to the surfaces $\RC_j$
(and their images). Recall the sets   $R_j:=\pi_j(Q_j)$ from the
decomposition of the surfaces $\RC_j$ 
(Lemma \ref{lem:RjcombSj} (\ref{item:RjSjcombeq})), where $Q_j\subset
\SC_j$ is a $\delta_j$-square.
Define
\begin{align}
  \label{eq:defphiRjT}
  & \phi_{R_j}\colon R_j\to [0,1]^2=[0,1]^2\times \{0\} \quad\text{ by }
  \\ \notag
   & \phi_{R_j}:= f_2\circ f_1|_{R_j},
\end{align}
where $f_1=f_{1,Q_j}, f_2=f_{2,B_{G_j}}$; the \emph{inner side}
of the piece $B(Q_j)$ is mapped here.  The maps $\phi_{R_j}$ are
quasisimilarities with scaling factor $l=1/\delta_j$ 
and uniform constant $L$. Again we think of the squares
$[0,1]^2(R_j):=\phi_{R_j}(R_j)$ as being distinct.

We now turn our attention to
how the \emph{outer side} of the piece $B(Q_j)$ is mapped. 
Let $R_{j+1}$ be a set from the decomposition of $\RC_{j+1}$
contained in (the outer side of) $B(Q_j)$.
Let
\begin{equation}
  \label{eq:defTj}
  T_{j+1}:=f_2\circ f_1(R_{j+1})\subset [0,1]^2\times\{1\},
\end{equation}
where $f_1=f_{1,Q_j}, f_2=f_{2,B_{G_j}}$ as before.
All such sets 
decompose $[0,1]^2\times
\{1\}$, the ``top face'' of the cube. 
%We will sometimes write $T_{j+1}=T_{j+1}(R_{j+1})$ to emphasize
%the preimage in $\RC_{j+1}$.
To later be able to ``put
adjacent shells together'' in a compatible way, we introduce the
following maps:
\begin{align}
  \label{eq:defpsiTj}
  \psi_{T_{j+1}} & \colon T_{j+1} \to [0,1]^2=[0,1]^2\times\{0\},
  \text{ defined by} 
  \\ 
  \notag
  \psi_{T_{j+1}}& :=\phi_{R_{j+1}} \circ f_1^{-1}\circ f_2^{-1}
\end{align}
on $T_{j+1}$. Note that in this expression $f_1=f_{1,Q_j},
f_2=f_{2,B_{G_j}}$, and $\phi_{R_{j+1}}=f_{2,Q_{j+1}} \circ
f_{1,B_{G_{j+1}}}$ ($R_{j+1}=\pi_{j+1}(Q_{j+1})$). 
This means
we are comparing how $R_{j+1}$ is mapped
as a set in the outer side of the piece $B(Q_j)$ versus how it is mapped
as the inner side of the piece $B(Q_{j+1})$. There are only
finitely many different sets $T_{j+1}$, thus the maps $\psi_{T_{j+1}}$
have a common bi-Lipschitz constant $L$. Figure \ref{fig:fconstr}
again illustrates the map. Note however that the picture is incorrect
insofar as $\psi_{T_{j+1}}$ maps between cubes $[0,1]^3(Q_j),
[0,1]^3(Q_{j+1})$ coming from pieces in \emph{different} shells
$\BC_j,\BC_{j+1}$. 

\begin{remark}\label{rem:nosymm}
  In the construction of the maps $f_1$ and $f_2$ the
  symmetry of the generators was not used. We merely used the facts
  that there 
  are only finitely many different ones and that they fit inside the
  double pyramid. 
\end{remark}

\begin{guide}
  We mapped pieces $B(Q)$ and quadrilaterals $R_j$ from the decomposition
  of the snowball $\BC$ to
  ``normalized'' ones (cubes, squares). In the next section these
  cubes will be mapped into the unit ball $\B$. Maps 
  $\BC \to [0,1]^3$ are denoted by $\phi$. Maps $[0,1]^3\to \B$
  will be denoted by $\varphi$. Intermediate maps $[0,1]^3\to [0,1]^3$
  are denoted by $\psi$. Note that $\phi,\psi,\varphi$ are maps on
  surfaces, namely
  on $\RC_j$ and images of them. Again the reader is advised to
  consult Figure \ref{fig:fconstr}.  
\end{guide}

\section{Reassembling the Unit Ball}
\label{sec:DecBall}
\subsection{Conformal Triangles}
\label{sec:confTri}
Recall how in Subsection \ref{sec:mapcyl} uniformization of the \mbox{$j$-th}
approximation $\mathcal{S}_j$ was used to decompose the sphere
$\mathbb{S}=\{\abs{x}=1\}$ conformally into $j$-tiles $X'$ 
\begin{equation*}
  \mathbb{S}=\bigcup_{X'\in\mathbf{X}'_j} X'.
\end{equation*}
Since it is easier to deal with simplices, we will decompose each
conformal square $X'$ into $4$ triangles. Divide the unit square
$[0,1]^2$ along the diagonals into $4$ triangles and map them to
$X'\in\mathbf{X}'_j$ by the conformal map $[0,1]^2\to X'$ (normalized
by mapping vertices to vertices). 

Alternatively we could divide each
$\delta_j$-square in the $j$-th approximation $\mathcal{S}_j$ along
the diagonals into $4$ $\delta_j$\defn{-triangles} and use uniformization on this
polyhedral surface to get the decomposition of the sphere $\mathbb{S}$
into \emph{conformal $j$-triangles}\index{j
  -@$j$-!triangle}\index{conformal triangle}. 
Denote the set of these conformal $j$-triangles by
$\widetilde{\mathbf{X}}_j$\label{not:conftrinagles}. 
Again $\widetilde{\X}_j$  forms a
conformal tiling, i.e., every $\widetilde{X}\in\widetilde{\X}_j$ is a
conformal reflection 
of its neighbors along shared sides. Figure \ref{fig:confTriag} shows
the conformal $1$-triangles of our main example $\widehat{\SC}$. It is
again conformally correct up to numerical errors. Compare this
picture with Figure \ref{fig:unifgen}.

Each conformal $j$-triangle has edges and vertices via the conformal
map. Again we speak of edges and vertices of \emph{order} $j$ (or
$j$-edges and $j$-vertices).  

\begin{figure}
  \centering
  \scalebox{0.65}{\includegraphics{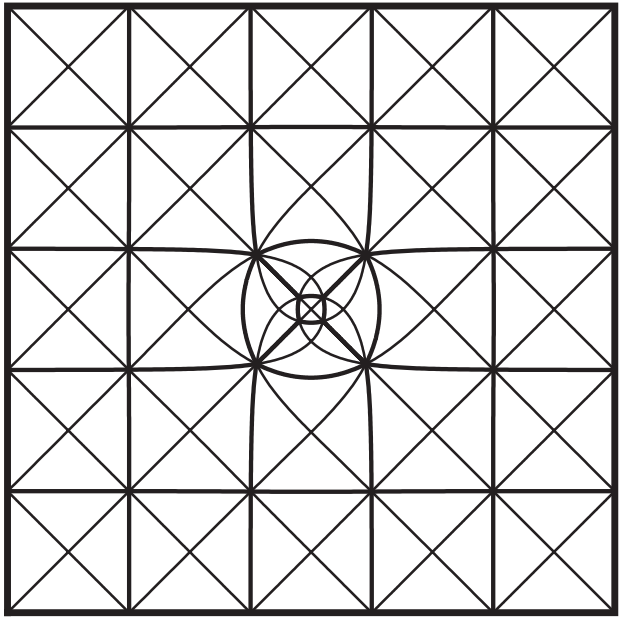}}
  \caption{Conformal $1$-triangles of $\widehat{\SC}$.}
  \label{fig:confTriag}
\end{figure}

It is true that each conformal $(j+1)$-triangle is contained in exactly
one conformal $j$-triangles. So the conformal $(j+1)$-triangles
\emph{subdivide} the conformal $j$-triangles. We do not need to prove
this here. 

Let $\widetilde{X} $ be a conformal $j$-triangle,
$\widetilde{Y} \in \widetilde{\X}_j$ have non-empty intersection
with $\widetilde{X}$, and $\widetilde{X} \subset X' \in \X'_j$ be
the $j$-tile containing it. Then using the same argument as in
Lemma \ref{lem:diamXY}
\begin{equation}
  \label{eq:conftricomp}
  \diam \widetilde{Y}  \asymp \diam \widetilde{X}  \asymp \diam X'.
\end{equation}
Here $C(\asymp)=C(N_{\max})$.

\smallskip
Map the triangulation of $\SC_j$ by $\pi_j$ (\ref{eq:defpij}) to the
surface $\RC_j$; images of $\delta_j$-triangles are called $\widetilde{R}_j$. 
We have obtained a triangulation of $\RC_j=\bigcup \widetilde{R}_j$. 
Each
quadrilateral $R_j$ thus gets divided into $4$ sets $\widetilde{R}_j$. 

Identify a quarter of the square $[0,1]^2$ with the standard
$2$-simplex $\Delta$ (\ref{not:simplex}); then
$\phi_{R_j}(\widetilde{R}_j)=\Delta=\Delta(\widetilde{R}_j)$ (see
(\ref{eq:defphiRjT}) as well as the definition of $\pi_j$
(\ref{eq:defpij})). 
We equip each such $2$-simplex with the metric $\norm{\cdot}_\Delta$
from (\ref{eq:defDmetric}) (so they are all isometric).

Every set $\widetilde{R}_j$ gets mapped by $\pi_j^{-1}$ to a $\delta_j$-triangle in $\SC_j$, which
the uniformization maps to a conformal $j$-triangle $\widetilde{X}_j\subset\SB$.
We call $\widetilde{X}_j$ the conformal triangle 
\emph{corresponding}\index{corresponding!conformal triangle} 
to $\widetilde{R}_j$
and write
$\widetilde{X}_{j}=\widetilde{X}_{j}(\widetilde{R}_j)$\label{not:XjRj}.
By the same procedure vertices and edges of $\widetilde{R}_j$ are
mapped to the \emph{corresponding}\index{corresponding!edges and
  vertices} edges and vertices of $\widetilde{X}_j$.  

Similarly every $R_j$ (from the decomposition of $\RC_j$ in
Lemma \ref{lem:RjcombSj} (\ref{item:RjSjcombeq})) is mapped by
$\pi^{-1}_j$ to a $\delta_j$-square $Q_j\subset \SC_j$, which in turn
is mapped by the uniformization to the \emph{corresponding} $j$-tile
$X'_j=X'_j(R_j)\in \X'_j$.

\subsection{Overview of the Decomposition of the unit Ball}
\label{sec:overview}

Before getting into details let us give a brief overview of this section. 
We will decompose the open unit ball $\inte \mathbb{B}=\{\abs{x} < 1\}$ into shells $\{\rho_j\leq\dist(x,\mathbb{S})\leq\rho_{j+1}\}$, which get decomposed into sets of the form 
$$\{(\omega,\rho)\in \mathbb{S}\times [0,1]=\mathbb{B}:\omega\in X_j', \; \rho_j\leq\rho\leq\rho_{j+1}\},$$  
where $X_j'\in \mathbf{X}'_j$ (using spherical coordinates). 
We will map cubes (being images of the pieces $B(Q_j)$) to these sets.

To assure quasiconformality we need $\diam X_j'\!\asymp
\rho_{j+1}-\rho_{j}$. Since $\diam X_j'/\!\diam Y_j'$ (where
$X_j',Y_j'\in\mathbf{X}'_j$) is neither bounded above nor below, radii
will not be constant on $\mathbb{S}$, but rather we will have
$\rho_j=\rho_j(\omega)$. 

\smallskip
In the next subsection our main concern is that neighboring pieces $B(Q_j)$ and $B(P_j)$ (where the $\delta_j$-squares $Q_j$ and $P_j$ are neighbors) are mapped in a \emph{compatible}\index{compatible} way, i.e., the maps agree on the intersecting face.

\smallskip
In Subsection \ref{sec:mapp-prisms-ttim} we make sure that pieces
``on top of each other'' are mapped in a compatible way. More
precisely, given a $\delta_j$-square $Q_j\subset\SC_j$ and a
$\delta_{j+1}$-square $Q_{j+1}\subset G(Q_j)\subset\SC_{j+1}$, we
require that the maps on  
$B(Q_j)$ and $B(Q_{j+1})$ agree on their intersection. Here $G(Q_j)$
is the scaled generator replacing $Q_j$ in the construction of $\SC_{j+1}$.

\subsection{Constructing the Maps $\varphi_{\widetilde{X}}\colon \Delta\to \widetilde{X}$}
\label{sec:constr-maps-tto}
First we will construct maps  $\varphi_{\widetilde{X}}$ from the
$2$-simplex $\Delta$ to a conformal $j$-triangle  $\widetilde{X}$.   

We could of
course use the Riemann map for this. The downside is that this
map will in general have singularities at the vertices, which would make the
extension to the cube $[0,1]^3$ somewhat difficult (though most
likely doable). We choose a different approach here; 
$\varphi_{\widetilde{X}}$ will be a quasisimilarity (see
(\ref{eq:defquasisym})) with 
scaling factor $l\asymp \diam \widetilde{X}$ and uniform constant $L$.
This makes extension of the map easier. 
We have to make sure that the maps 
are \defn{compatible} on
neighbors $\widetilde{X},\widetilde{Y}\in\widetilde{\X}_j$. More
precisely, if $\Delta'$ 
is a reflection of $\Delta$ along one of its edges $E=\Delta\cap \Delta'$
which is mapped to the common edge of $\widetilde{X}$ and $\widetilde{Y}$ by the maps
$\varphi_{\widetilde{X}}\colon \Delta\to \widetilde{X}$ and $\varphi_{\widetilde{Y}}\colon \Delta'\to
\widetilde{Y}$  
\begin{align}
  &\varphi_{\widetilde{X}}(E)=\widetilde{X}\cap
  \widetilde{Y}=\varphi_{\widetilde{Y}}(E),\notag 
  \intertext{then } 
  \;&\varphi_{\widetilde{X}}|_E=\varphi_{\widetilde{Y}}|_E.\label{eq:phicomp}
\end{align}
If we used the Riemann maps for $\varphi_{\widetilde{X}}$ and $\varphi_{\widetilde{Y}}$  instead, this would follow immediately by the reflection principle.

\medskip
Note that by construction the number of conformal $j$-triangles
intersecting in a $j$-vertex is always even.  Consider one such
$j$-triangle $\widetilde{X}$. If at its vertices $2n$, $2m$, and $2l$
$j$-triangles intersect (in counterclockwise order), the angles are
$\frac{\pi}{n}$, $\frac{\pi}{m}$, and $\frac{\pi}{l}$. We say
$\widetilde{X}$ is of \emph{type} $(n,m,l)$. \index{type!of conformal
  triangle}\label{not:typetrian}Consider a neighborhood of
$\widetilde{X}$ 
\begin{equation*}
  U(\widetilde{X}):=\inte \:\bigcup\{\widetilde{Z}\in\widetilde{\X}_j:\widetilde{X}\cap \widetilde{Z}\ne \emptyset\}.
\end{equation*}
One can get $U(\widetilde{X})$ by repeated reflection. Therefore the Riemann map $\psi\colon \widetilde{X}\to \widetilde{Y}$ between two conformal triangles $\widetilde{X}$ and $\widetilde{Y}$ of the same type (normalized by mapping vertices to corresponding vertices) extends to these neighborhoods $\overline{\psi}\colon U(\widetilde{X})\to U(\widetilde{Y})$. Since $\widetilde{X}$ is compactly contained in $U(\widetilde{X})$, 
$\psi$ is quasisimilar by Koebe distortion. For each occurring type $(n,m,l)$ we fix one conformal triangle $X(n,m,l)$\label{not:Xnml} of this type. There are only finitely many $X(n,m,l)$. We will now construct bi-Lipschitz maps
\begin{equation*}
  \varphi\colon \Delta\to X(n,m,l).
\end{equation*}
By composing with a Riemann map $\psi=\psi_{\widetilde{Y}}\colon X(n,m,l)\to \widetilde{Y}$
as above ($\widetilde{Y}$ is of type $(n,m,l)$), we get a quasisimilarity 
\begin{equation}\label{eq:phiY}
  \varphi_{\widetilde{Y}}:=\psi\circ\varphi \colon \Delta\to \widetilde{Y} 
\end{equation}
for any conformal triangle $\widetilde{Y}$. The scaling factor of $\varphi_{\widetilde{Y}}$
is $l=|\psi'(x)|\asymp \diam \widetilde{Y}$ for any $x\in X(n,m,l)$,
and the
bi-Lipschitz constant $L$ of $\varphi_{\widetilde{Y}}$ is uniform (by Koebe).   

\medskip
Initially the maps $\varphi$ will only be defined on the boundary
$\partial \Delta$ of $\Delta$. In fact, let us first define $\varphi$
just on one edge of $\Delta$. For simplicity we assume this edge to be
$[0,1]\subset\R^2$ and $\Delta\subset \R^2$. Now consider an edge
$E'\subset \partial \widetilde{X}$ of a conformal triangle
$\widetilde{X}\in \widetilde{\X}_j$. We say $E'$ is of
\emph{type}\index{type!of edge}\label{not:typeedge} $(n,m)$ if
$\widetilde{X}$ has angles $\frac{\pi}{n}$ and $\frac{\pi}{m}$ (in
counterclockwise order as a boundary of $\widetilde{X}$) at the vertices
of $E'$. 
% (equivalently if $2n$ and $2m$ $j$-triangles intersect as the vertices of $F$). 
For an edge $E'$ of order $j$ consider a neighborhood 
\begin{equation*}
  U(E'):=\inte
  \:\bigcup\{\widetilde{Z}\in\widetilde{\X}_j:\widetilde{Z}\cap E'\ne
  \emptyset\}. 
\end{equation*}
Let $\widetilde{X}$ be a conformal triangle of type $(n,m,l)$ and $\widetilde{Y}$ one of type $(n,m,\tilde{l})$. Then the conformal map $\varphi\colon \widetilde{X}\to \widetilde{Y}$ (normalized by mapping $1$st, $2$nd, and $3$rd vertex onto each other) extends to a map $\overline{\varphi}\colon U(E')\to U(F')$, where $E'\subset\partial \widetilde{X}$ and $F'\subset\partial \widetilde{Y}$ are the edges of type $(n,m)$. So $\varphi$ is a quasisimilarity on $E'$ by Koebe.

For each occurring type $(n,m)$ of an edge, we define $T(n,m)\subset\R^2$\label{not:Tnm} to be a (fixed)
\begin{itemize}
\item \emph{circular arc triangle}, meaning all its edges are circular arcs.
\item One edge of $T(n,m)$ is $[0,1]\subset\R^2$, which is of type $(n,m)$. We think of $[0,1]$ as the image of the edge $[0,1]\subset\partial \Delta$ under the identity. 
\item $T(m,n)$ is the reflection of $T(n,m)$ along the line
  $x=\frac{1}{2}$. This means we can put $T(n,m)$ in the upper and
  $T(m,n)$ in the lower half plane, such that
  $T(m,n)=\overline{T(n,m)}$ ($\bar{z}$ denotes complex
  conjugation). In particular $T(n,n)$ is symmetric with respect to
  $x=\frac{1}{2}$. 
\end{itemize}
The third angle of $T(n,m)$ is arbitrary. The third condition will
ensure compatibility in the sense of equation (\ref{eq:phicomp}), as
will be seen in the next lemma. For the edge $E\subset X(n,m,l)$ of
type $(n,m)$ we define the map $\varphi_E\colon [0,1]\to E$ by
$\varphi_E:=\zeta|_{[0,1]}$, where $\zeta\colon T(n,m)\to X(n,m,l)$ is
the Riemann map (normalized by mapping vertices to vertices, in
particular vertices with angles $\frac{\pi}{n}$ and $\frac{\pi}{m}$
onto each other). By the above consideration $\varphi_E$ is
bi-Lipschitz. Using the same procedure on the other edges we get a
bi-Lipschitz map $\varphi\colon \partial \Delta\to \partial X(n,m,l)$
(here we are using the fact that $X(n,m,l)$ has no zero angles). It is
well 
known that we can extend this to a bi-Lipschitz map $\varphi\colon
\Delta\to X(n,m,l)$ (Theorem A in \cite{Lipext}). 
\begin{lemma}\label{lem:phicomp}
  The maps $\varphi_{\widetilde{X}}\colon \Delta\to \widetilde{X}$,
  defined by equation \eqref{eq:phiY}, are compatible in the sense of
  equation \eqref{eq:phicomp}, meaning the maps on intersecting edges
  ``agree''. 
\end{lemma}
\begin{proof}
  The proof is illustrated in Figure \ref{fig:phicompatible}.
  Let $\widetilde{X}$ and $\widetilde{Y}$ be two neighboring
  $j$-triangles. Let $\widetilde{X}$ be of type $(n,m,l)$ and
  $\widetilde{Y}$ be of type $(m,n,\tilde{l})$. Let $E'=F'=\partial
  \widetilde{X}\cap \partial \widetilde{Y}$, where $E'\subset \partial
  \widetilde{X}$ is an edge of type $(n,m)$ and $F'\subset \partial
  \widetilde{Y}$ is an edge of type $(m,n)$. As before, assume that
  $\varphi_{\widetilde{X}}$ maps $[0,1]\subset \partial \Delta$ to
  $E'$. By construction we have  
  \begin{equation*}
    \varphi_{\widetilde{X}}|_{[0,1]}=\phi|_{[0,1]},
  \end{equation*}
where $\phi$ is the Riemann map from $T(n,m)$ to $\widetilde{X}$
(normalized by mapping vertices to vertices, in particular vertices
with angles $\frac{\pi}{n}$ and $\frac{\pi}{m}$ onto each other). By
the reflection principle $\phi$ extends to $T(m,n)$, which is mapped conformally to $\widetilde{Y}$ (and maps vertices to vertices). By definition we get
\begin{equation*}
  \varphi_{\widetilde{X}}|_{[0,1]}=\varphi_{\widetilde{Y}}|_{[0,1]}.
\end{equation*}
\end{proof}

\begin{figure}
  \centering
  % 
  % begin file generated by xfig
  \begin{picture}(0,0)%
    \includegraphics{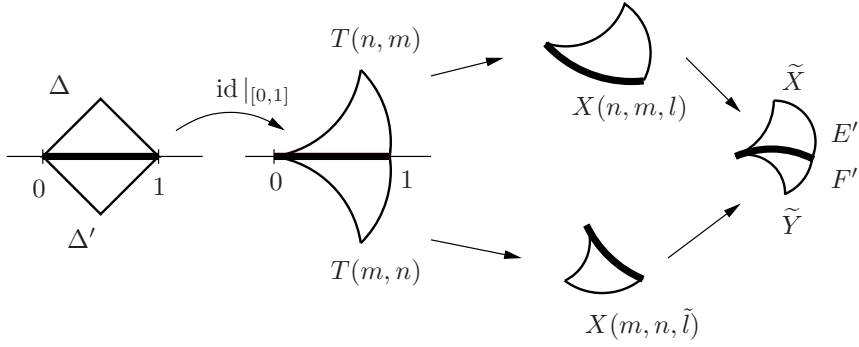}%
  \end{picture}%
  \setlength{\unitlength}{3947sp}%
  \begingroup\makeatletter\ifx\SetFigFontNFSS\undefined%
  \gdef\SetFigFontNFSS#1#2#3#4#5{%
    \reset@font\fontsize{#1}{#2pt}%
    \fontfamily{#3}\fontseries{#4}\fontshape{#5}%
    \selectfont}%
  \fi\endgroup%
  \begin{picture}(5202,2146)(53,-1420)
    \put(327,134){$\Delta$}
    \put(227,-513){$0$}          
    \put(982,-501){$1$}
    \put(444,-833){$\Delta'$}
    \put(1380,117){$\id|_{[0,1]}$}
    \put(1715,-467){$0$}
    \put(2539,-456){$1$}
    \put(2100,426){$T(n,m)$}
    \put(2100,-1028){$T(m,n)$}
    \put(3614,  2){$X(n,m,l)$}
    \put(3695,-1371){$X(m,n,\tilde{l})$}
    \put(4931,158){$\widetilde{X}$}
    \put(4931,-746){$\widetilde{Y}$}
    \put(5240,-169){$E'$}
    \put(5240,-456){$F'$}  
  \end{picture}%

  \caption{Defining $\varphi$.}
  \label{fig:phicompatible}
\end{figure}

Recall that we identified the $2$-simplex $\Delta$ with a quarter of
the square $[0,1]^2$. Thus from the maps $\varphi_{\widetilde{X}}$ we get
maps
\begin{equation}
  \label{eq:defvarphiXj}
  \varphi_{X_j'}\colon [0,1]^2 \to X_j',
\end{equation}
for every $j$-tile $X_j'$. They are quasisimilarities
(\ref{eq:defquasisym}) with scaling factor $l=\diam X'_j$ and uniform
constant $L$, since the maps $\varphi_{\widetilde{X}}$ are (see
(\ref{eq:phiY})). The lemma above means that these maps are well
defined and
compatible in the sense of (\ref{eq:phicomp}) (with simplices replaced
by squares, and conformal triangles replaced by tiles). This means
that when identifying a unit square adjacent to $[0,1]^2$ with the square
that $\varphi_{Y'_j}$ maps to a neighbor $Y'_j$ of $X'_j$, the maps
$\varphi_{X'_j}, \varphi_{Y'_j}$ agree on the intersecting edge. In
this case the simplex $\Delta'$ from (\ref{eq:phicomp}) is a reflection
of $\Delta$ along this edge.

\subsection{Connecting adjacent Layers}\label{sec:mapp-prisms-ttim}

The map $f$ will be defined on the surfaces $\RC_j$ first.
In this subsection we define their $\omega$-coordinates (of the spherical
coordinates $(\omega,\rho)\in\SB\times[0,1]$). In the
next subsection the radial-coordinate will be defined. 

Consider one quadrilateral $R_j\subset\RC_j$ (see Lemma \ref{lem:RjcombSj}
(\ref{item:RjSjcombeq}) and (\ref{eq:defphiRjT})). The
$\omega$-coordinate of $f|_{{R}_j}$ is given as the composition of the
maps
\begin{equation}
  \label{eq:deffRj}
   \phi_{R_j}\colon {R}_j\to [0,1]^2 \text{ and } \varphi_{X'_j}\colon [0,1]^2\to X'_j.
\end{equation}
Here of course ${X}'_j={X}'_j({R}_j)\in{\X}'_j$, and vertices were
mapped to corresponding ones. This means that the maps
$\varphi_{X'_j}$ (\ref{eq:defvarphiXj}) are normalized to map vertices
correctly in the above composition.   

The following construction is done to ensure that points in
$\RC_{j+1}=\BC_j\cap \BC_{j+1}$ are mapped to the same points when
the two shells $\BC_j$ and $\BC_{j+1}$ are mapped. The reader may
first want to skip the remainder of this section, and return here
before reading through (\ref{eq:layercomp}).  
 
Recall how in the last section the snowball was decomposed into
pieces $B(Q_j)$, each of which was mapped to the unit cube. Recall the
decomposition of the top face of the cube 
into sets
$T_{j+1,k}$ (\ref{eq:defTj}).

Construct a map $\psi\colon [0,1]^2=[0,1]^2\times \{1\} \to
[0,1]^2=[0,1]^2\times \{1\}$ in the following
way. 
%The reader should think of $[0,1]^2$ as the top face of the unit
%cube. 
Let $T_{j+1}\subset [0,1]^2\times \{1\}\subset [0,1]^3=
f_2\circ f_1 (B(Q_j))$ be a set from the decomposition of the top
face of the unit cube. Let $R_{j+1}$ be the set
from the decomposition of $\RC_{j+1}$ corresponding to $T_{j+1}$
($f_2\circ f_1(R_{j+1})=T_{j+1}$). On each set $T_{j+1}$ the map $\psi$ is defined as the
composition of the maps $\psi_{T_{j+1}}\colon T_{j+1}\to [0,1]^2$
(\ref{eq:defpsiTj}), $\varphi_{X'_{j+1}}\colon [0,1]^2\to X'_{j+1}$
(\ref{eq:defvarphiXj}), and $\varphi^{-1}_{X'_j}$. Here
$\varphi_{X'_j}\colon [0,1]^2\to X'_j\supset X'_{j+1}$, and
$X'_{j+1}=X'_{j+1}(R_{j+1})$,  
\begin{equation}
  \label{eq:defvarphiT}
  \psi=\psi_{Q_j}:=\varphi_{X'_j}^{-1} \circ \varphi_{X'_{j+1}} \circ \psi_{T_{j+1}}. 
\end{equation}
The map $\psi$ is well defined by Lemma \ref{lem:phicomp}, meaning
on intersections of neighbors $T_{j+1,k}\cap T_{j+1,l}$ the two maps
agree. 

\begin{lemma} \label{lem:TtoT}
  The above defined map 
  \begin{equation*}
    \psi\colon [0,1]^2\to [0,1]^2
  \end{equation*}
is bi-Lipschitz with uniform bi-Lipschitz constant.
\end{lemma}
\begin{proof}

The maps $\psi_{T_{j+1}}$ are uniformly bi-Lipschitz, and the maps
$\varphi_{X'_{j+1}}$ and $\varphi_{X'_j}$ are both quasisimilar with
scaling factor $\diam X'_j$ and uniform bi-Lipschitz constant ($\diam
X'_{j+1}\asymp \diam X'_j$ by Corollary \ref{cor:XjXj+1}).  
To show that $\psi$ is bi-Lipschitz consider $x,y\in [0,1]^2$. Break
up the line segment between $x$ and $y$ into segments that lie in one
set $T_{j+1,k}$: 
  \begin{equation*}
    |x-y|=\sum_{k=0}^M|x_k-x_{k+1}|,
  \end{equation*}
  where $x_0=x$, $x_M=y$, and $x_k,x_{k+1}\in T_{j+1,k}$. Then
  \begin{align*}
    |\psi(x)-\psi(y)|&\leq \sum_k |\psi(x_k)-\psi(x_{k+1})|
    \\
    &\leq \sum_k L|x_k-x_{k+1}|\leq L|x-y|.
  \end{align*}
  The other inequality follows by reversing the above argument. %Since there are only finitely many different generators and therefore finitely many different triangulations $\bigcup T_{j+1}=\bigcup T'_{j+1}=T$ the statement follows.
\end{proof}

Now we use the Alexander trick from Lemma \ref{lem:AlexanderTrick} to
construct a 
bi-Lipschitz map 
\begin{equation}
  f_3=f_{3,Q_j}\colon [0,1]^3 \to [0,1]^3\label{not:f3}\index{f 3@$f_3$}
\end{equation}
such that $f_3=\id$ on $[0,1]^2\times \{0\}$ and $\psi=\psi_{Q_j}$ on
the top face $[0,1]^2\times\{1\}$. The map $f_3$ is uniformly
bi-Lipschitz, since $\psi$ is. 
%The preimages come from the decomposition of the snowball $\mathcal{B}$, the images will be put together to form the unit ball $\mathbb{B}$ in the remainder of this subsection. 

\subsection{Reassembling the unit Ball}
\label{sec:recball}
In this subsection $\widetilde{X}_j$ will always denote a conformal
$j$-triangle (and $\widetilde{X}_{j+1}$ a conformal $(j+1)$-triangle
etc.) and $X'_j$ always denotes 
a $j$-tile. To ensure that constants do
not explode we will require that 
appearing constants are \emph{uniform}, meaning they depend only on
$N_{\max}$ (and not on the particular $j$-triangle at hand). We call a
Lipschitz map with uniform Lipschitz constant \emph{uniformly}
Lipschitz\index{uniformly Lipschitz}; similarly for bi-Lipschitz maps. 
 
Let $v$ be a vertex of  $\widetilde{X}_j$. We define 
\begin{equation*}\label{not:djv}
  d_j(v):=\max_{v\in \widetilde{Y}\in \widetilde{\X}'_j}\diam \widetilde{Y}.
\end{equation*}
Neighboring $j$-triangles have comparable sizes
by (\ref{eq:conftricomp}), so   
\begin{equation}\label{eq:dasymp}
  d_j(v)\asymp \diam \widetilde{X}_j,   
\end{equation}
where $C(\asymp)$ is a uniform constant.
Consider a conformal $(j+n)$-triangle $\widetilde{X}_{j+n}$, such that
$\widetilde{X}_{j+n}\cap \widetilde{X}_{j}\neq \emptyset$. 
Using Lemma \ref{diamX0} and (\ref{eq:conftricomp}) we have $\diam
\widetilde{X}_{j+n}\lesssim\lambda^n \diam \widetilde{X}_j$ for a
fixed $\lambda <1$. Thus there 
is an $n\geq 1$ such that  
\begin{equation*}
  \frac{d_{j+n}(v)}{d_j(w)}\leq c_1<1,
\end{equation*}
for every vertex $v$ of $\widetilde{X}_{j+n}$ and vertex $w$ of
$\widetilde{X}_j$ ($c_1$ is a uniform constant). Assume $n=1$;
otherwise we would redo the construction of the snowball by ``doing
$n$ steps at once.'' More precisely, consider the $n$-th approximation
of one face of the snowball $\mathcal{T}_n$ as an
$\widetilde{N}_1$-generator ($\widetilde{N}_1=N_1\cdot \ldots \cdot
N_n$), replace each $\delta_n$-face by a scaled copy of an
$\widetilde{N}_2$-generator ($\widetilde{N}_2=N_{n+1}\cdot \ldots
\cdot N_{2n}$) and so on. Note that the $\widetilde{N}_j$-generators
need not be symmetric with respect to the diagonals, since we did
allow the replacement of $\delta_j$-squares with scaled copies of
$N_{j+1}$-generators with arbitrary orientation. There will be not
only one 
$\widetilde{N}_j$-generator, but several (though finitely many). Still
the embedding and the decomposition work exactly as before. See the
Remark on page \pageref{rem:nosymm}. 
  
\medskip
So we have
\begin{equation}\label{eq:c0djc1}
  0<c_0\leq \frac{d_{j+1}(v)}{d_j(w)}\leq c_1<1
\end{equation}
for vertices $v\in \widetilde{X}_{j+1}$, and $w\in \widetilde{X}_{j}$
where $\widetilde{X}_{j+1}\cap \widetilde{X}_j\neq\emptyset$ ($c_0$ and
$c_1$ are uniform constants). The left inequality follows from
Corollary \ref{cor:XjXj+1} and (\ref{eq:conftricomp}).  
%By lemma \ref{lem:diamXY} we have for vertices $v_0,v_1\in \widetilde{X}_j$
%\begin{equation}
%  d_j(v_0)\asymp d_j(v_1),
%\end{equation}
%where $C=C(\asymp)$ is a uniform constant. 
For a vertex $v\in \widetilde{X}_j$ define
\begin{equation*}
  \rho_j(v):=1-\frac{1}{2}d_j(v),
\end{equation*}
which will be the radius at $v$ of the $j$-th sphere which is
decomposed into $j$-triangles. The factor $\frac{1}{2}$ ensures that
$\rho_0>0$. Let $v_0,v_1,v_2$ be the vertices of $\widetilde{X}_j$,
and let $\varphi=\varphi_{\widetilde{X}_j}\colon \Delta\to
\widetilde{X}_j$ be the map defined in Subsection
\ref{sec:constr-maps-tto}, normalized by $\varphi(e_k)=v_k$ (see (\ref{not:simplex})). For $\omega=\varphi(x_0e_0+x_1e_1+x_2e_2)\in \widetilde{X}_j$ define
\begin{equation*}\label{not:rhoj}\index{r hoj@$\rho_j$}
  \rho_j(\omega):=x_0\rho_j(v_0)+x_1\rho_j(v_1)+x_2\rho_j(v_2).
\end{equation*}
Note that compatibility of the maps $\varphi$ (Lemma
\ref{lem:phicomp}) ensures that $\rho_j$ is well defined on the sphere
$\SB$. 

\smallskip
Consider the decomposition of the unit sphere into conformal
$0$-triangles $\widetilde{X}_0$. Since all conformal $0$-triangles
$\widetilde{X}_0$ have the same size, we have
\begin{equation}
  \label{eq:r0const}
  \rho_0(\omega)=\const=:\rho_0,   
\end{equation}
for all $\omega \in \mathbb{S}$. 
Each conformal $0$-triangle $\widetilde{X}_0$ is contained in
one $0$-tile $X'_0$, which is compactly contained in one
hemisphere. Thus $\diam \widetilde{X}_0 < 2$ and $0<\rho_0<1$. 

\smallskip
Now consider the map 
\begin{align}%\label{not:f4}\index{f 4@$f_4$}
  \notag
  & f_4=f_{4,X'_j}\colon [0,1]^3 \to \{(\omega,\rho):\omega\in X'_j,
  \rho_j(\omega)\leq \rho \leq \rho_{j+1}(\omega)\} 
  \intertext{defined by}
  &
  \label{eq:deff4}
  f_4(x,t):=\big(\varphi(x),(1-t)\rho_j(\varphi(x))+t\rho_{j+1}(\varphi(x))\big),  
\end{align}
where $\varphi=\varphi_{X'_j}$ from equation
(\ref{eq:defvarphiXj}). The right hand side is expressed in spherical
coordinates. 
\begin{lemma}
  \label{lem:prism2ball}
  The map $f_4$ is a quasisimilarity
  \begin{equation*} 
    \frac{1}{L}\abs{(x,t)-(y,s)}\leq
    \frac{1}{l}\abs{f_4(x,t)-f_4(y,s)}\leq L\abs{(x,t)-(y,s)}, 
  \end{equation*}
with uniform bi-Lipschitz constant $L$ and scaling factor $l=\diam X'_j$. 
\end{lemma}
\begin{proof}
  We will show that the maps $\phi:= \frac{1}{\diam X'_j}\varphi$,
  $\widetilde{\rho}_0:=\frac{1}{\diam X'_j}\rho_j$ and
  $\widetilde{\rho}_1:=\frac{1}{\diam X'_j}\rho_{j+1}$ satisfy the
  conditions of Lemma \ref{lem:prism1}.  

\medskip 

\subsubsection*{$\bullet\;\phi$ is uniformly bi-Lipschitz.}
This is obvious from the fact that $\varphi=\varphi_{X'_j}$ is
quasisimilar with scaling factor $l=\diam X'_j$ and uniform constant
$L$. 

\subsubsection*{$\bullet\;\widetilde{\rho}_0$ is uniformly Lipschitz.}
For $a_0,a_1,a_2\in \R$ consider the map
 $$h(x_0e_0+x_1e_1+x_2e_2):=x_0a_0+x_1a_1+x_2a_2$$ 
on $\Delta$. One
 checks directly that $h$ is Lipschitz with constant
 $2\max_{n,m}|a_n-a_m|$ (in the $\norm{\cdot}_{\Delta}$-metric on
 $\Delta$). By (\ref{eq:dasymp}) and (\ref{eq:conftricomp}) we obtain
 \begin{equation*}%\label{eq:rhovnvm}
   \Abs{\rho_j(v_n)-\rho_j(v_m)}\leq C \diam X'_j
 \end{equation*}
for vertices $v_n,v_m\in X'_j$ and a uniform constant $C$. So
$\max_{n,m}|\widetilde{\rho}_0(v_n)-\widetilde{\rho}_0(v_m)|\leq
C$. Since $\phi$ is uniformly bi-Lipschitz by the above, we obtain that
$\widetilde{\rho}_0$ is uniformly Lipschitz. 

\subsubsection*{$\bullet\;\widetilde{\rho}_1$ is uniformly Lipschitz}
As before and using Corollary \ref{cor:XjXj+1}, it follows that
$\widetilde{\rho}_1$ is uniformly Lipschitz on any $X'_{j+1}\subset
X'_j$. Since $\varphi$ is quasisimilar with $l=\diam X'_j$ and uniform
$L$, one obtains exactly as in the proof of Lemma \ref{lem:TtoT}
that $\widetilde{\rho}_1\circ\varphi$ is Lipschitz on $[0,1]^2$, with
Lipschitz constant $\widetilde{L}\lesssim \diam X'_j$. The claim follows.

\subsubsection*{$\bullet\;\widetilde{\rho}_0+m\leq \widetilde{\rho}_1\leq \widetilde{\rho}_0+M$ (with uniform constants $m,M>0$)}
Fix a conformal $j$-triangle $\widetilde{X}_j$. 
Let
\begin{align*}
  &d_{j,\max} := \max d_j(w), 
  && d_{j,\min} := \min d_j(w),
  \\
  &d_{j+1,\max} := \max d_{j+1}(v), 
  && d_{j+1,\min}  := \min d_{j+1}(v),
\end{align*}
where the maxima/minima are taken over all $j$-vertices $w\in
\widetilde{X}_j$, and $(j+1)$-vertices $v\in \widetilde{X}_{j+1}$, where
$\widetilde{X}_{j+1}\cap \widetilde{X}_j\neq \emptyset$. 

By equation (\ref{eq:c0djc1}) we have for all $\omega \in \widetilde{X}_j$
\begin{align}
  \notag
  \rho_{j+1}(\omega) & \geq 1- \frac{1}{2}d_{j+1,\max}\geq
  1-\frac{1}{2}c_1 d_{j,\min}
  \\
  \notag
  & \geq \rho_j(\omega) + \frac{1}{2}(1-c_1)d_{j,\min},
  \intertext{as well as}
  \label{eq:rhojj}
  \rho_{j+1}(\omega) & \leq 1 - \frac{1}{2} d_{j+1,\min}\leq 1-\frac{1}{2} c_0
  d_{j,\max}
  \\
  \notag
  & \leq \rho_j(\omega) + \frac{1}{2} (1-c_0) d_{j,\max}.
\end{align}
Note that $d_{j,\min}\asymp d_{j,\max} \asymp \diam X'_j$, where
$C(\asymp)$ is uniform. The claim follows.

\smallskip
Let $\overline{\phi}$ be the extension of $\phi$ from Lemma
\ref{lem:prism1}. It is uniformly bi-Lipschitz.

\medskip
The map $f_4$ is a composition of the extension $\overline{\phi}$, a
scaling by the factor $\diam X'_j$, and the map $\overline{\psi}$ from
Lemma \ref{lem:prismS}. Here $r=\rho_0, R=1$ and $\psi=\id\colon
X'_j\to X'_j$; thus $\overline{\psi}$ is uniformly bi-Lipschitz.   
This finishes the proof of the lemma.
 
\end{proof}

Let 
\begin{equation*}
  B'(X'_j):=f_{4,X'_j}([0,1]^3)=f_4 \circ f_3 \circ f_2 \circ f_1 (B(Q_j)),   
\end{equation*}
where $f_1=f_{1,Q_j}, f_2=f_{2,B_{G_j}}, f_3=f_{3,Q_j}$,
and $X'_j$ is the $j$-tile corresponding to the cylinder $X_j(Q_j)$. 
The following follows directly from the definition of $f_4$.

\begin{lemma}
  \label{lem:WhitneyBall}
  The sets $B'(X'_j)$ together with the set $\{\abs{x} \leq \rho_0\}$
  form a Whitney decomposition of the unit ball $\B$.
  \begin{enumerate}
  \item $$\bigcup_{j,X'_j\in\X'_j} B'(X'_j) \cup \rho_0\B= \inte \B.$$
  \item The interiors of the sets $B'(X'_j)$ are pairwise disjoint.
  \item \label{item:WhitneyBall3}
    \begin{align*}
      \diam X'_j       
      & \asymp
      \diam B'(X'_j) 
      \\
      & \asymp
      \dist(B'(X'_j),\SB) = \dist(B'(X'_j),X'_j)
      \\
      & \asymp \Hdist(B'(X'_j),X'_j),
    \end{align*}
    where $C(\asymp)=C(N_{\max})$.
  \end{enumerate}
\end{lemma}

\begin{proof}
  The first two assertions are clear.
  From expressions (\ref{eq:dasymp}) and Corollary \ref{cor:XjXj+1} we
  obtain ($X'_{j+1}\subset X'_j$)
  \begin{equation*}
    \dist(B'(X'_j),X'_j) \asymp \diam X'_{j+1} \asymp \diam X'_j
  \end{equation*}
  immediately. From expression (\ref{eq:rhojj}) we obtain
  \begin{equation*}
    \diam B'(X'_j) \asymp \diam X'_j.
  \end{equation*}
  It is obvious that $\dist(B'(X'_j), X'_j)=\dist(B'(X'_j), \SB)$. The
  two expressions above imply $\Hdist(B'(X'_j),X'_j)\asymp \diam X'_j$.
\end{proof}
\subsection{Defining the Map $f$}
\label{sec:map-f}
On each set $B(Q_j)$ the map $f$ is defined as 
\begin{equation*}
  f:=f_4\circ f_3 \circ f_2 \circ f_1.
\end{equation*}
Here $f_1=f_{1,Q_j}, f_2=f_{2,B_{G_j}}, f_3=f_{3,Q_j}, f_{4,X'_j}$,
where $X'_j$ is the $j$-tile corresponding to the cylinder $X_j(Q_j)$. 
We need to check that $f$ is well defined.
 
\begin{lemma}
  \label{lem:fextwelldef}
  The map $f$ is well defined on $\bigcup_{j\geq 0} \BC_j$.
\end{lemma}

\begin{proof} 
  \begin{mylist}
    \item\label{item:compRj}
      Consider first neighboring $\delta_j$-squares $Q_j,Q^*_j\subset
      \SC_j$. Map the sets $R_j:=\pi_j(Q_j),
      R^*_j:=\pi_j(Q^*_j)\subset \RC_j$ as 
      inner sides of the pieces $B(Q_j), B(Q^*_j) \subset \BC_j$. 
      Let $f$ be as above, and let $f^*$ be the
      corresponding map for $B(Q^*_j)$. Note that the $\omega$-coordinate of
      $f$ is $\varphi_{X'_j}\circ
      \phi_{R_j}$ on $R_j$ by construction, where
      $X'_j=X'_j(R_j)\in\X'_j$ (see (\ref{eq:defphiRjT}),
      (\ref{eq:defvarphiXj}), (\ref{not:f3}), and
      (\ref{eq:deff4})). The maps  
      $\phi_{R_j}, 
      \phi_{R^*_j}$ are affine on (the line segment) $R_j\cap R^*_j$. Let
      $Y'_j$ be the $j$-tile 
      corresponding to $R^*_j$. The maps $\varphi_{R_j}, \varphi_{R^*_j}$ are
      compatible by Lemma \ref{lem:phicomp}. The $\omega$-coordinates of $f$
      and $f^*$ thus agree on $R_j\cap R^*_j$.    
      
      Since the radii $\rho_j$ were well defined on $\SB$, it follows that
      $f=f^*$ on $R_j\cap R^*_j$.
    \item\label{item:complayers}
      We next check compatibility of different layers.
      Let $R_{j+1}$ be a set from the decomposition of $\RC_{j+1}$
      (Lemma \ref{lem:RjcombSj} (\ref{item:RjSjcombeq})). Let $B(Q_j)\subset
      \BC_j, B(Q_{j+1})\subset\BC_{j+1}$ be the pieces containing $R_{j+1}$,
      so $B(Q_j)\cap B(Q_{j+1})=R_{j+1}$. Here $Q_j,Q_{j+1}$ denote
      $\delta_j$-, $\delta_{j+1}$-squares in $\SC_j,\SC_{j+1}$. Let
      $f,f_1,f_2$ be as above, and let $f'$ be the map
      corresponding to $B(Q_{j+1})$. 
      
      Let $R_j=\pi_j(Q_j)$ be the inner side of $B(Q_j)$, and let $X'_j=X'_j(R_j)\in
      \X'_j, X'_{j+1}=X'_{j+1}(R_{j+1})\in \X'_{j+1}$ be the tiles corresponding to
      $R_j,R_{j+1}$.
      Finally let $T_{j+1}=f_2\circ f_1(R_{j+1})$ (see (\ref{eq:defTj})).
      The $\omega$-coordinate of $f$ on $R_{j+1}$ is the
      $\omega$-coordinate of $f_{4,X'_j} \circ f_{3,Q_j} \circ
        f_{2,B_{G_j}}\circ f_{1,Q_j}|_{R_{j+1}}$, which is given by 
      \begin{align}
        \label{eq:layercomp}
        \varphi_{X'_j} & \circ \psi_{Q_j} \circ f_{2,B_{G_j}}\circ
        f_{1,Q_j} 
        && \text{ by (\ref{eq:deff4}) and (\ref{not:f3}),}
        \\ \notag
        & = \varphi_{X'_{j+1}} \circ \psi_{T_{j+1}} \circ f_{2,B_{G_j}} \circ
        f_{1, Q_j} && \text{ by (\ref{eq:defvarphiT}),}
        \\ \notag
        & = \varphi_{X'_{j+1}} \circ \phi_{R_{j+1}} &&\text { by
          (\ref{eq:defpsiTj})}.  
      \end{align}
      This is equal to the $\omega$-coordinate of $f'$ on $R_{j+1}$ as
      above. It is clear that the radii of $f$ and
      $f'$ agree on $R_{j+1}$ from the construction. 
      Thus $f=f'$ on $R_{j+1}$.
    \item \label{item:neighpiecescomp}
      It remains to show that $f$ is well defined on neighboring
      pieces $B(Q_j)$, $B(Q^*_j)\subset \BC_j$. Here the notation from
      (\ref{item:compRj}) is used again. Maps $f,f_1,f_2,f_3,f_4$ are the
      ones corresponding to $B(Q_j)$, $f^*,f^*_1,f^*_2,f^*_3,f^*_4$
      the ones corresponding to $B(Q^*_j)$.

      By Lemma \ref{lem:f1comp} it holds that $f_1=f^*_1$ on
      $B(Q_j)\cap B(Q^*_j)$ 
      (where appropriate sides of $B_G(Q_j)$ and $B_{G^*}(Q^*_j)$ are
      identified).   
      
      The fact that $f_2=f^*_2$ on $B_G(Q_j)\cap B_{G^*}(Q^*_j)$ is clear
      (again with proper identification of sides). The map is an
      affine map from a rectangle to a square.

      Now consider the maps $\psi_{Q_j}, \psi_{Q^*_j}$ from Subsection
      \ref{sec:mapp-prisms-ttim}. The intersection of their domains is
      (after proper identification) one edge of 
      $[0,1]^2\times \{1\}$. From (\ref{item:complayers}) and
      Lemma \ref{lem:phicomp} we obtain that
      $f_3=f^*_3=\psi_{Q_j}=\psi_{Q^*_j}$ on this edge. This implies
      that $f_3=f^*_3$ on the intersecting square (in which
      the properly identified unit cubes intersect); see Lemma
      \ref{lem:AlexanderTrick}.   

      Finally $f_4=f^*_4$ on the intersecting square (with proper
      identification). This follows again by Lemma \ref{lem:phicomp}
      and the construction. Thus
      $f=f^*$ on $B(Q_j)\cap B(Q^*_j)$. 
  \end{mylist} 
 \end{proof}

It remains to define $f$ on the cube $[c,1-c]^3\subset [0,1]^3$, which
is the cube bounded by (see Subsection \ref{sec:whitn-decomp-snowb})
$\mathcal{R}_0=\{x\in \BC:\dist_{\infty}(x,\partial[0,1]^3)=c\}$ (recall that
$c=\frac{1}{2}-\frac{1}{2N_{\max}}$).

The map $f$ maps $\RC_0$ bi-Lipschitz to $\rho_0 \SB$.
Extend this map radially to $[c,1-c]^3$ using (a variant of) Lemma 
\ref{lem:extSB}. The extension is bi-Lipschitz on $[c,1-c]^3$.

\smallskip
On the complement of the snowball $\mathcal{B}$ the map $f$ is defined analogously. The snowsphere is approximated from the outside by the surfaces
\begin{equation*}
  \mathcal{R}^+_j:=\{x\notin\mathcal{B}:\dist_{\infty}(x,\mathcal{S}_j)=c\delta_j\}.
\end{equation*}
The shells
\begin{equation*}
  \mathcal{B}^+_j:=\{x\notin
  \mathcal{B}:\dist_{\infty}(x,\mathcal{S}_j)\leq c\delta_j \text{ and
  } \dist_{\infty}(x,\mathcal{S}_{j+1})\geq c\delta_{j+1}\}
\end{equation*}
are decomposed as before and mapped to
\begin{equation*}
  \{(\omega,\rho):\rho^+_{j+1}(\omega)\leq\rho(\omega)\leq\rho^+_{j}(\omega)\},
\end{equation*}
where $\rho^+_j(v):=1+\frac{1}{2}d_j(v)$ for vertices $v\in X'_j$ and extended to the sphere $\mathbb{S}$ as before. Again the maps are piecewise quasisimilarities with uniform bi-Lipschitz constant.

\smallskip
One gets a map from the the cube $[-c,1+c]^3$ to the ball
$\rho^+_0\mathbb{B}$ as before. Extending this map to $\R^3$ is easy. 
For example, given $x_0\in \partial [-c,1+c]^3$ map the ray
$\{tx_0 : t\geq 1\}$ linearly to the ray $\{tf(x_0):t\geq 1\}$. It is
straightforward to check that this extension is bi-Lipschitz on
$\R^3\setminus (-c,1+c)^3$ (use (\ref{eq:distwr})).

\section{Proof of the main Theorem}
\label{sec:proof-main-theorem}
%It remains to show that $f$ is quasiconformal on the snowsphere
%$\SC$. 

\subsection{Combinatorial Distance}
\label{sec:comb-dist}
We want to express the distance between a point in the interior of the
 snowball $x\in \inte\BC\setminus (c,1-c)^3$ and a point on the snowsphere $y\in
\SC$ in purely combinatorial terms.

Let $Q_k\subset\SC_k$ be a $\delta_k$-square such that $x\in B(Q_k)$
(see Subsection \ref{sec:decomposing-shells}). Let $X_k=X_k(Q_k)\subset\SC$ be
the $k$-cylinder having base $Q_k$. Recall the definition of $j(x,y)$
(\ref{eq:defjxy}) 
to set 
\begin{equation*}
  j:=\sup_{z\in X_k} j(z,y).
\end{equation*}
Note that $j=\infty$ if $y\in X_k$. 
Finally let $n:=\min\{k,j\}\in \N$. 

\begin{lemma}
  \label{lem:dnxy}
  With notation as above
  \begin{equation*}
    \abs{x-y}\asymp \delta_n,
  \end{equation*}
  where $C(\asymp)=C(N_{\max})$.
\end{lemma}

\begin{proof}
  By definition of $n$ there exists $z\in X_k$ and $(n-1)$-cylinders
  $Z_{n-1}\ni z, 
  Y_{n-1}\ni y$ that are not disjoint. Thus
  \begin{align*}
    \abs{x-y}  \leq & \diam B(Q_k) + \dist(B(Q_k), X_k) 
    \\
    & + \diam X_k + \diam Z_{n-1} + \diam Y_{n-1}
    \\
    \lesssim & \; \delta_k + \delta_k + \delta_k + \delta_{n-1} +
    \delta_{n-1}
    \\
    \lesssim & \; \delta_n,
  \end{align*}
  by Lemma \ref{lem:BQWhitney} (\ref{item:Whitney3}).

  \smallskip
  To see the other inequality we first need to fix the relevant
  constants. Let $C_0=C(\asymp)$ be the constant from Lemma
  \ref{lem:pseudcomp}. In particular 
  \begin{equation*}
    \abs{y-z}\geq \frac{1}{C_0}\delta_j,
  \end{equation*}
  for all $z\in X_k$. Let $C_1=C(\asymp)$ be the constant from Lemma
  \ref{lem:HdistSjRj} (\ref{item:HdistRjS}); in particular
  \begin{equation*}
    \Hdist(X_k,B(Q_k)) \leq C_1 \delta_k.
  \end{equation*}
  Let the integer $k_0=k_0(C_0,C_1)=k_0(N_{\max})\geq 0$ be such that
  \begin{equation*}
    C_2:=\frac{1}{C_0}-C_1 2^{-k_0} >  0.
  \end{equation*}

  \noindent\emph{Case 1:} $n\leq k \leq n + k_0$. Then
  \begin{equation*}
    \abs{x-y} \geq \dist(B(Q_k), \SC) \asymp \delta_k \asymp \delta_n. 
  \end{equation*}
  
  \noindent\emph{Case 2:} $k > n + k_0$. Then $n=j$ and
  \begin{align*}
    \delta_k & =\delta_j \frac{1}{N_{j+1}}\times \dots \times
    \frac{1}{N_{k}}\leq \delta_j 2^{-k_0}, \text{ yielding }
    \\
    \abs{x-y} & \geq \dist(y,B(Q_k))
    \\
    & \geq \dist(y, X_k) - \Hdist(X_k,B(Q_k)) \text{ by
      (\ref{eq:triag_dist})}
    \\
    & \geq \frac{1}{C_0}\delta_j - C_1\delta_k \geq
    \left(\frac{1}{C_0}-C_1 2^{-k_0}\right) \delta_j = C_2 \delta_n.
  \end{align*}

\end{proof}

Next we express the distance of images by $f$ in combinatorial
terms. Images of $x,y,B=B(Q_k), X_k$ are denoted by
$x',y',B'=B'(X'_k),X'_k$. So $x'\in B', y'\in \SB$. The numbers $k,j,$ and $n$ are the same as before. 

\begin{lemma}
  \label{lem:diamXnxy}
  With notation as before,
  \begin{equation*}
    \abs{x'-y'} \asymp \diam X'_n, 
  \end{equation*}
  where $X'_k\subset X'_n\in \X'_n$, and $C(\asymp)=C(N_{\max})$.
\end{lemma}

\begin{proof}
  The argument is almost the same as in the previous proof. 
  Throughout the whole proof $X'_l$ will denote an $l$-tile satisfying
  $X'_l\cap X'_k\neq \emptyset$.

  There is a point $z'\in X'_k$ and 
  $(n-1)$-tiles $Z'_{n-1}\ni z'$, $Y'_{n-1}\ni y'$ that are not
  disjoint. Hence by Lemma
  \ref{lem:WhitneyBall} (\ref{item:WhitneyBall3}) (as well as Lemma \ref{lem:diamXY} and
  Corollary \ref{cor:XjXj+1})
  \begin{align*}
    \abs{x'-y'}  \leq & \diam B' + \dist(B', X'_k) 
    \\
    & + \diam X'_k + \diam Z'_{n-1} + \diam Y'_{n-1}
    \\
    \lesssim & \diam X'_k + \diam X'_k + \diam X'_k + \diam X'_{n-1} +
    \diam X'_{n-1} 
    \\
    \lesssim & \diam X'_n.
  \end{align*}

  For the other inequality let $C_0=C(N_{\max})$ be the constant from
  Lemma \ref{lem:combdistS}. In particular
  \begin{equation*}
    \dist(y',X'_k) \geq \frac{1}{C_0} \diam X'_j.
  \end{equation*}
  We set the right hand side to $0$ if $y'\in X'_k$ ($\Leftrightarrow
  j=\infty$). 
  The constant $C_1=C(N_{\max})$ is obtained from Lemma
  \ref{lem:WhitneyBall} (\ref{item:WhitneyBall3}) such that
  \begin{equation*}
    \Hdist(X'_k,B'(X'_k))\leq C_1 \diam X'_k.
  \end{equation*}
  Let $0<c_1<1$ be the constant from expression (\ref{eq:c0djc1});
  in particular 
  \begin{equation*}
    \diam X'_k\leq c_1^{k_0} \diam X'_{k-k_0}.
  \end{equation*}
  Choose the integer $k_0=(C_0,C_1)=C(N_{\max}) \geq 0$ such that
  \begin{equation*}
    C_2 := \frac{1}{C_0}-C_1 c_1^{k_0} > 0.
  \end{equation*}

  \noindent\emph{Case 1:} $n\leq k \leq n+k_0$. Then (by Corollary \ref{cor:XjXj+1})
  \begin{equation*}
    \abs{x'-y'} \geq \dist(B'(X'_k),\SB) \asymp \diam X'_k\asymp \diam
    X'_n.
  \end{equation*}

  \noindent\emph{Case 2:} $k> n+k_0$. Then $n=j$ and
  \begin{align*}
    \abs{x'-y'} & \geq \dist(y',B'(X'_k))
    \\
    & \geq \dist(y',X'_k) - \Hdist(X'_k,B'(X'_k))
    \\
    &  \geq \frac{1}{C_0}\diam X'_j - C_1 \diam X'_k
    \\
    & \geq \frac{1}{C_0}\diam X'_n - C_1 \diam X'_{n+k_0}
    \\
    & \geq \left(\frac{1}{C_0} - C_1 c_1^{k_0}\right) \diam X'_n = C_2
    \diam X'_n.
  \end{align*}
\end{proof}

\begin{remark}[1]
  If $x\in \SC$ (equivalently $x'\in \SB$) set $k=\infty$. The
  statements of the last two lemmas remain valid with $j=n$ (by Lemma
  \ref{lem:pseudcomp} and Lemma \ref{lem:combdistS}).  
\end{remark}

\begin{remark}[2]
  Analogous statements of the last two lemmas hold if $x$ ($x'$) is
  outside the snowball (the unit ball). 
\end{remark}

\begin{remark}[3]
  Recall from the proof of the last lemma that there is $z'\in X'_k$
  and non-disjoint 
  $(n-1)$-tiles $Y'_{n-1}\ni y', Z'_{n-1}\ni z'$. 
Thus
  (using Lemma \ref{lem:diamXY} and Corollary \ref{cor:XjXj+1})
  \begin{equation}
    \label{eq:XnYnpf1B}
    \diam Y'_n\asymp \diam X'_n
  \end{equation}
  for any $n$-tile $Y'_n\ni y'$.
\end{remark}

We note the following (using Lemma \ref{diamX0} as well).
\begin{cor}
  The map $f\colon \R^3\to \R^3$ is a homeomorphism. 
\end{cor}

\subsection{Proof of Theorem 1B}
\label{sec:proof-theorem-1b}
The map $f$ is quasisimilar (\ref{eq:defquasisym}) with uniform
constant $L$ on $\R^3\setminus \SC$ by construction. Thus it is
quasiconformal on $\R^3\setminus \SC$ by definition (\ref{eq:defqc}).  

It remains to show quasiconformality on $\SC$. Let $y\in \SC$ and
$x,z\in [-c,1+c]^3\setminus (c,1-c)^3$. 
The number $n$ is defined as
in the last section, the number $m$ analogously for the points
$z,y$. 
Let $x',y',z'$ be the images
of $x,y,z$ under $f$. Throughout the proof $Y'_l$ will
always denote an $l$-tile containing $y'$. The $n$-tile $X'_n$ is the
one from Lemma \ref{lem:diamXnxy}, the $m$-tile $Z'_m$ the
corresponding one for the points $z',y'$. 
Assume
\begin{align*}
  \abs{y-x} & =\abs{y-z}.
  \intertext{This implies by Lemma \ref{lem:dnxy} and Remarks (1) and (2)}
  \delta_n & \asymp \delta_m, \text{ hence}
  \\
  n-k_0 & \leq m \leq n+k_0,
\end{align*}
for a constant integer $k_0=k_0(N_{\max})$. Thus
\begin{align*}
  \diam Y'_n & \asymp \diam Y'_m,
\end{align*}
by Corollary \ref{cor:XjXj+1}. By Remark (3) from the last section
\begin{align*}
  \diam X'_n & \asymp \diam Y'_n \asymp \diam Y'_m \asymp \diam Z'_m,
  \text{ and so}
  \\
  \abs{y'-x'} & \asymp\abs{y'-z'}, 
\end{align*}
by Lemma \ref{lem:diamXnxy}.
This finishes the proof.

\qed

\section{Open Problems}
\label{sec:open-problems}

The main open problem remains to geometrically characterize
quasiballs/quasi\-spheres. This seems to be a very hard problem in
$\R^3$ and out of reach at the moment in $\R^n$, $n\geq 4$.

\smallskip
The snowspheres constructed here have (many) rectifiable curves. This
contrasts with the surfaces constructed in \cite{MR2001c:49067} (see
also \cite{MR99j:30023}). They admit parametrizations $f\colon \R^2
\to \R^3$ satisfying $\abs{f(x)-f(y)}\asymp \abs{x-y}^\alpha$. Here
$\alpha=1-\epsilon$ with a (tiny) $\epsilon>0$. One may think of such
a parametrization as being \defn{uniformly expanding}. Are there
uniformly expanding maps $f\colon 
\R^2\to \R^3$ such that the Hausdorff-dimension of 
the image is arbitrarily close to $3$? This means that $\alpha$ is
arbitrarily close to $2/3$.

The same question can be asked
in more generality: are there maps $f\colon \R^n\to \R^m$, $n<m$,
satisfying $\abs{f(x)-f(y)}\asymp \abs{x-y}^\alpha$, where $\alpha$ is
arbitrarily close to $n/m$? It is relatively easy to construct such a
map for $n=1,m=2$ (see \cite{MR2003b:30022}). This implies that the
answer is yes for $m=2n$. The general case however seems to be quite
difficult. 

\section*{Acknowledgments}
\label{sec:acknowledgments}
  The author wishes to thank his former advisor Steffen Rohde for his
  patience and guidance. Some discussions with Jang-Mei Wu initiated
  the ideas that led to the extension. Mario Bonk deserves much credit
  for carefully reading the manuscript, exposing various flaws, and
  making many helpful suggestions.

% ==========   Bibliography
\bibliographystyle{alpha}
\bibliography{litlist}

%%% Local Variables: 
%%% mode: latex
%%% End: 

\end{document}